\documentclass[a4paper,reqno,10pt]{amsart}
\usepackage{amsfonts} 
\usepackage{amssymb}
\usepackage{amsthm}
\usepackage{amsmath} 
\usepackage{mathrsfs} 
\usepackage{dsfont}  
\usepackage{stmaryrd} 
\usepackage[english]{babel} 

\usepackage[left=3.5cm,right=3.5cm,top=3.5cm,bottom=3cm,headsep=0.7cm]{geometry}

\usepackage{pgf,tikz}
\usetikzlibrary{arrows}
\usetikzlibrary{intersections}
\usepackage{ifthen}

\usepackage{hhline}
\usepackage{tabularx, booktabs}

\usepackage[hyphens,spaces]{url}

\usepackage[scriptsize,bf]{caption}
\usepackage{floatrow} 

\usepackage{enumerate}
\usepackage{enumitem}

\usepackage{hyperref}

\theoremstyle{plain}
\begingroup
\theoremstyle{plain}
\newtheorem{theorem}{Theorem}[section]

\newtheorem{proposition}[theorem]{Proposition}

\theoremstyle{definition}
\newtheorem{definition}[theorem]{Definition}
\theoremstyle{remark}
\newtheorem{remark}[theorem]{Remark}

\endgroup

\theoremstyle{definition}
\theoremstyle{remark}

\numberwithin{equation}{section}

\newcommand{\EEE}{\color{black}}

\newcommand{\R}{\mathbb{R}}
\newcommand{\N}{\mathbb{N}}

\newcommand{\Q}{\mathbb{Q}}
\renewcommand{\S}{\mathcal{S}}

\newcommand{\J}{\mathcal{J}}

\renewcommand{\L}{\mathcal{L}}

\newcommand{\Norm}{\mathcal{N}}

\renewcommand{\r}{\mathbf{r}}
\newcommand{\C}{\mathcal{C}}
\newcommand{\U}{\mathcal{U}}
\newcommand{\W}{\mathcal{W}}

\newcommand{\F}{\mathcal{F}}
\renewcommand{\P}{\mathbb{P}}
\newcommand{\Pcal}{\mathcal{P}}
\newcommand{\E}{\mathbb{E}}

\mathchardef\emptyset="001F
\renewcommand{\d}{\mathrm{d}}
\newcommand{\de}{\partial}

\renewcommand{\tilde}{\widetilde}
\newcommand{\x}{{\times}}

\newcommand{\sm}{\setminus}

\newcommand{\supp}{\mathrm{supp}}

\newcommand{\muc}{\mu^{\mathrm{c}}}
\newcommand{\rhoc}{\rho^{\mathrm{c}}}
\newcommand{\mup}{\mu^{\mathrm{p}}}
\newcommand{\rhop}{\rho^{\mathrm{p}}}
\newcommand{\mubp}{\bar \mu^{\mathrm{p}}}

\newcommand{\nuc}{\nu^{\mathrm{c}}}
\newcommand{\nup}{\nu^{\mathrm{p}}}
\newcommand{\nubp}{\bar \nu^{\mathrm{p}}}

\newcommand{\Hcp}{H^\mathrm{cp}}
\newcommand{\Kcp}{K^\mathrm{cp}}

\newcommand{\Kpc}{K^\mathrm{pc}}

\newcommand{\Kpg}{K^\mathrm{pg}}

\newcommand{\Hd}{H^\mathrm{d}}
\newcommand{\cHd}{\check{H}^{\mathrm{d}}}

\newcommand{\Kgg}{K^\mathrm{gg}}
\newcommand{\Law}{\mathrm{Law}}

\newcommand{\ev}{\mathrm{ev}}

\let\div\relax
\DeclareMathOperator{\div}{div}

\newcommand{\Id}{\mathrm{Id}}

\newcommand{\weak}{\rightharpoonup}
\newcommand{\wstar}{\stackrel{*}{\rightharpoonup}}

\newcommand{\mynorm}{{\vert\kern-0.25ex\vert\kern-0.25ex\vert}}

\newcommand{\ie}{{\itshape i.e.}}
\newcommand{\eg}{{\itshape e.g.}}
\newcommand{\cf}{{\itshape cf.}}

\newcommand{\step}[1]{{ \underline{\itshape Step #1}}.}
\newcommand{\substep}[2]{{\itshape Substep #1.#2}.}

\author{Gianluca Orlando}
\address[]{Dipartimento di Meccanica, Matematica e Management. Politecnico di Bari. Via E. Orabona~4, 70125 Bari BA, Italy.}
\email{gianluca.orlando@poliba.it}

\title[Mean-field optimal control for maritime crime]{Mean-field optimal control in a multi-agent interaction model for prevention of maritime crime}

\begin{document}

\begin{abstract}
    We study a multi-agent system for the modeling maritime crime. The model involves three interacting populations of ships: commercial ships, pirate ships, and coast guard ships. Commercial ships follow commercial routes, are subject to traffic congestion, and are repelled by pirate ships. Pirate ships travel stochastically, are attracted by commercial ships and repelled by coast guard ships. Coast guard ships are controlled. We prove well-posedness of the model and existence of optimal controls that minimize dangerous contacts. Then we study, in a two-step procedure, the mean-field limit as the number of commercial ships and pirate ships is large, deriving a mean-field PDE/PDE/ODE model. Via $\Gamma$-convergence, we study the limit of the corresponding optimal control problems.  
\end{abstract}

\maketitle


\vspace{1em}



\setcounter{tocdepth}{1}
\tableofcontents

\section{Introduction}
 
Systems featuring interactions among multi-agents have attracted much attention of the scientific community in recent years as they find applications in various fields. They are a proper tool to study, \eg, biological aggregation as in flocks, swarms, or fish schools~\cite{CucSma,CarForRosTos, TopBer}, crowd dynamics~\cite{AlbBonCriKal}, emergent economic behaviors~\cite{Cha, DraYak}, consensus in collective decision-making~\cite{CarGia, MasGiaCar}, coordination and cooperation in robotics~\cite{ChuHuaDOBer, PerGomElo}. In this framework, Mathematical Analysis has played a role in the proof of well-posedness of the models, in the derivation of mean-field limit, and in the analysis of optimal control problems for this kind of models~\cite{ForSol, BonForRosSol, AlbChoForKal, AlbBonRosSol, MorSol, AmbForMorSav, AlmMorSol, AlbAlmMorSol, AlmDEMorSol, FagKauRad}.
 
In this paper, we exploit the tools developed for the analysis of multi-agent systems to study optimal control in a model for the prediction of maritime crime. The majority of world's goods is carried by sea~\cite{EU1}, but freedom of navigation is affected by the presence of modern maritime piracy, which poses serious threats to international traffic and individual safety. It is a priority to prevent crimes and suppress them~\cite{EU2}.  

To face this problem, we devise a model featuring three populations of agents, representing the type of ships. Our model is inspired by the macroscopic model (\ie, with a large number of ships) introduced in~\cite{CocGarSpi}, but it differs from it in that our derivation starts from a microscopic model (\ie, with a finite number of ships). We briefly outline it in this introduction, referring to Section~\ref{sec:description of the model} for the precise description of all the features and assumptions on the model. 

We consider three populations: $N$ commercial ships with trajectories $X_1, \dots, X_N$, $M$ pirate (criminal) ships with trajectories $Y_1,\dots, Y_M$, and $L$ coast guard (patrol) ships with trajectories $Z_1,\dots,Z_L$. The trajectory of each ship evolves in a time interval $[0,T]$ according to a specific dynamical law based on its type and on the presence of other surrounding ships, as we illustrate now. 

Commercial ships tend to follow commercial routes, but their motion is affected by traffic congestion: a commercial ship obstructed by a high density of commercial ships travels slower than one with free space. Moreover, in presence of pirate ships, commercial ships are repelled by them and adjust their trajectory to travel far from danger. Hence, the $n$-th commercial ship evolves according to 
\begin{equation} \label{eqintro:Xn}
    \frac{\d X_n}{\d t}(t) = v^N_n(X(t)) \Big( \r(X_n(t)) + \frac{1}{M} \sum_{m=1}^M \Kcp(X_n(t) - Y_m(t)) \Big) \, ,
\end{equation}
where $v^N_n$ is a suitable function depending on all the other commercial ships needed for the congestion phenomenon, $\r$ is the vector field indicating the commercial route, and $\Kcp$ is term due to the repulsion from pirate ships that adjusts the direction of the trajectory. 

Pirate ships are attracted by commercial ships and are repelled by coast guard ships. Moreover, in absence of other ships, they travel randomly in search of targets. Hence, the $m$-th pirate ship evolves stochastically, according to 
\begin{equation} \label{eqintro:Ym}
    \d Y_m(t) = \Big( \frac{1}{L} \sum_{\ell=1}^L \Kpg(Y_m(t) - Z_\ell(t)) -   \frac{1}{N} \sum_{\ell=1}^N \Kpc(Y_m(t) - X_n(t)) \Big) \, \d t + \sqrt{2 \kappa} \, \d W_m(t) \, ,
\end{equation}
where $\Kpg$ and $\Kpc$ are the repulsion and attraction terms with coast guard ships and commercial ships, respectively. The term $(W_m(t))_{t \in [0,T]}$ is a Brownian motion accounting for the stochastic behavior mentioned above. Its effect is a white noise with coefficient $\sqrt{2 \kappa}$ added to the velocity of $Y_m$.

Finally, for coast guard ships we only impose that are repelled one by each other and that their trajectory is controllable, at a cost. Hence, the $\ell$-th coast guard ship evolves according to
\[
    \frac{\d Z_\ell}{\d t}(t) =  \frac{1}{L} \sum_{\ell'=1}^L \Kgg(Z_\ell(t) - Z_{\ell'}(t)) + u_\ell(t) \, ,
\]
where $\Kgg$ is the repulsion term among coast guard ships and the $u_\ell$'s are the control. 

The search of coast guard ships for dangerous contacts between commercial and pirate ships will be driven by the optimal control of the system, based on the cost defined as follows. The cost of a control $u = (u_1,\dots,u_L)$ takes into account the effort in modifying the trajectories of coast guard ships (it can be thought as the cost of fuel), and the total number of dangerous contacts among commercial and pirate ships
\begin{equation} \label{eqintro:JNM}
    \J_{N,M}(u) :=  \frac{1}{2}  \int_0^T    |u(t)|^2 \d t + \E \Big(\int_0^T \frac{1}{N} \frac{1}{M} \sum_{n=1}^N \sum_{m=1}^M \Hd(X_n(t) - Y_m(t))  \, \d t  \Big) \, ,
\end{equation}
where $\Hd$ is a compactly supported convolution kernel used for counting dangerous contacts and $\E$ denotes the expected value. We study the problem of finding a control that minimizes~$\J_{N,M}$.

In Section~\ref{sec:well-posedness} we prove well-posedness of the model that describes the evolution and we prove existence of an optimal control. 

Next, we proceed with the derivation of the mean-field limit of the optimal control problem. We carry out this analysis in two steps: first, we let $M \to +\infty$ (large number of pirate ships), and then $N \to +\infty$ (large number of commercial ships). The reason thereof is that the limit as $M \to +\infty$ is interesting {\itshape per se}, as we explain forthwith. 

Under suitable conditions, in Section~\ref{sec:M to infty} (see Theorem~\ref{thm:limit to averaged system} and Proposition~\ref{prop:PDE formulation of ODE/SDE/ODE}) we show that, as $M \to +\infty$, the mean-field behavior of pirate ships is described by a probability distribution~$\mubp$. The trajectories of commercial ships $\bar X_n$ in this mean-field model satisfy 
\begin{equation} \label{eqintro:barXn}
    \frac{\d \bar X_{n}}{\d t}(t) = v^N_n(\bar X(t))\Big(  \r(\bar X_{n}(t)) + \Kcp*\mubp(t)(\bar X_{n}(t))    \Big) \, ,
\end{equation}
which corresponds to~\eqref{eqintro:Xn} with the trajectories of pirate ships replaced by their mean-field behavior. The probability distribution~$\mubp$ of pirate ships solves the diffusive PDE 
\begin{equation} \label{eqintro:parabolic PDE}
    \de_t \mubp - \kappa \Delta_y \mubp + \div_y\Big( \Big( \frac{1}{L} \sum_{\ell=1}^L \Kpg(\cdot - Z_\ell(t)) - \frac{1}{N} \sum_{n=1}^N \Kpc(\cdot - \bar X_n(t))  \Big)\mubp \Big)  = 0 \, .
\end{equation}
This mean-field model is interesting \emph{per se} when the precise location of pirate ships is not known, but one can only predict the probability of finding them in certain regions of the sea. Proving convergence of solutions of the original model to the mean-field model as $M \to +\infty$ requires some technical steps, mainly done following the guidelines in~\cite{AscCasSol}. First, in Section~\ref{sec:averaged model} we introduce an auxiliary averaged model where the evolution of pirate ships is replaced by a single stochastic process evolving according to the same dynamics of~\eqref{eqintro:Ym}, \ie,
\[
    \d \bar Y(t) = \Big( \frac{1}{L} \sum_{\ell=1}^L \Kpg(\bar Y(t) - Z_\ell(t))  - \frac{1}{N} \sum_{n=1}^N \Kpc(\bar Y(t) - \bar X_n(t)) \Big) \d t + \sqrt{2\kappa} \, \d W(t) \, ,
\]
where $\bar X_n$ evolves according to~\eqref{eqintro:barXn}, $\mubp$ being the law of the stochastic process $(\bar Y(t))_{t \in [0,T]}$. In Section~\ref{sec:averaged model} we prove well-posedness for this averaged model, using a fixed-point argument. Solutions to the original model converge, as $M \to +\infty$, to solutions of this auxiliary averaged model. To see this, in Propositions~\ref{prop:identically distributed}--\ref{prop:propagation of chaos} we rely on a propagation of chaos principle~\cite{ChaDie}, from which we deduce that solutions to~\eqref{eqintro:Ym} are independent and identically distributed stochastic processes, if so are the initial conditions. Then, a Glivenko-Cantelli-type result allows us to deduce convergence of the empirical measures of the $Y_m$'s to their common law~$\mubp$. The parabolic PDE~\eqref{eqintro:parabolic PDE} is then the Fokker-Planck equation for pirate ships, as shown in Proposition~\ref{prop:PDE formulation of ODE/SDE/ODE}.

After deriving the mean-field limit as $M\to +\infty$, in Theorem~\ref{thm:Gamma convergence for M} we show that the costs $\J_{N,M}$ defined in~\eqref{eqintro:JNM} $\Gamma$-converge, as $M \to +\infty$, to the cost for the limit problem 
\begin{equation} \label{eqintro:JN}
    \J_N(u) :=   \frac{1}{2}   \int_0^T    |u(t)|^2 \d t + \frac{1}{N}   \sum_{n=1}^N  \int_0^T  \int_{\R^2} \Hd(\bar X_n(t) - y) \, \d \mubp(t)(y) \, \d t  \, . 
\end{equation}
As a consequence, optimal controls for the original problem converge as $M \to +\infty$ to optimal controls for the limit problem, see Proposition~\ref{prop:existence of optimal control N}. This concludes the analysis as $M \to +\infty$. 

The next step is to study the mean-field limit as the number of commercial ship is large, \ie, when $N \to +\infty$. In Theorem~\ref{thm:N to infty} and Proposition~\ref{prop:PDE formulation of PDE/SDE/ODE}, we show that the mean-field limit of commercial ships is described in terms of their distribution $\muc$, which solves a scalar conservation law with a nonlocal flux, apt to describe traffic flow in sea. More precisely, $\muc$ is a solution to the PDE
\[
    \de_t \muc + \div_x \Big( v\big( \eta *_2 \muc \big) \big(  \r + \Kcp*\mup   \big) \muc \Big) = 0 \, ,
\]
where $v\big( \eta *_2 \muc \big)$ arises from the limit of the congestion velocities and $\mup$ is the probability distribution of pirate ships, evolving according to the parabolic PDE
\[
    \de_t \mup - \kappa \Delta_y \mup + \div_y\Big( \Big( \frac{1}{L} \sum_{\ell=1}^L \Kpg(\cdot - Z_\ell(t)) -  \Kpc*\muc  \Big)\mup \Big)  = 0 \, .
\]
Under suitable assumptions, in Theorem~\ref{thm:uniqueness PDE/PDE/ODE} we prove uniqueness of solutions to this PDE system and, as observed in Remark~\ref{rmk:absolutely continuous}, that the measures are absolutely continuous, \ie, $\muc = \rhoc \, \d x$ and $\mup = \rhop \, \d y$. 
 
We conclude the paper by finding in Theorem~\ref{thm:Gamma convergence for N} the $\Gamma$-limit of the costs $\J_{N}$ defined in~\eqref{eqintro:JN} as $N \to +\infty$. It is given by the cost for the latter mean-field system 
\begin{equation*} 
    \J(u) :=   \frac{1}{2}   \int_0^T    |u(t)|^2 \d t +  \int_0^T  \int_{\R^2 \x \R^2} \Hd(x - y) \, \d \muc(t) \x \mup(t)(x,y) \, \d t  \, , 
\end{equation*}
Also in this case, we deduce convergence of optimal controls as $N \to +\infty$, see Proposition~\ref{prop:existence of optimal control}. The limit problem is an optimal control problem with a finite number of coast guard ships driving the densities of commercial and criminal ships.  
 
\section{Notation and preliminary results} \label{sec:preliminaries}

\subsection{Basic notation and preliminary results}

Given a matrix $A$, we let $|A|$ its Frobenius norm. We shall often consider matrices of the form $A \in \R^{2 \x d}$. By writing $A = (A_1,\dots,A_d)$ we make explicit its columns $A_i \in \R^2$. 

If $\Omega$, $\Omega'$ are measurable spaces, $\mu$ is a measure on $\Omega$, $\psi \colon \Omega \to \Omega'$ is a measurable map, then the push-forward $\psi_\# \mu$ is the measure on $\Omega'$ satisfying $\int_{\Omega'} \phi(\omega') \, \d \psi_\# \mu(\omega') := \int_{\Omega} \phi(\psi(\omega)) \, \d \mu(\omega)$ for every measurable function $\phi$. 

Throughout the paper, we shall systematically apply Gr\"onwall's inequality. We recall that if $u,v, w \colon [0,T] \to \R$ are continuous and nonnegative functions satisfying 
\[
u(t) \leq w(t) + \int_0^t v(s) u(s) \, \d s \quad \text{for every } t \in [0,T] \, ,    
\]
then
\[
    u(t) \leq w(t) + \int_0^t v(s) w(s) e^{\int_s^t v(r) \, \d r} \, \d s\quad \quad \text{for every } t \in [0,T]  \, ,
\]
\cf~\cite[Theorem~1.3.2]{Pac}.  If, in addition, $w \colon [0,T] \to \R$ is continuous, positive, and nondecreasing,
then 
\[
u(t) \leq w(t)e^{\int_0^t v(s)\, \d s} \quad \quad \text{for every } t \in [0,T]  \, ,
\]
\cf~\cite[Theorem~1.3.1]{Pac}

If not specified otherwise, we let $C$ denote a constant that might change from line to line. We make precise the dependence of $C$ on other constants when it is relevant for the discussion. 

\subsection{Stochastic processes and Brownian motion} \label{subsec:SDE} For the theory of stochastic processes and stochastic differential equations we refer to the monographs~\cite{KarShr, Mao, Oks}. Here we recall some basic facts and definitions used in the paper.

We fix a probability space $(\Omega, \F, \P)$ used throughout the paper. By a.s.\ (almost surely) we mean $\P$-almost everywhere.  We let $\E$ denote the expectation. 

A filtration on $(\Omega, \F, \P)$ is a collection of $\sigma$-algebras $(\F_t)_{t \in [0,T]}$ increasing in $t$, \ie, $\F_s \subset \F_t$ for $s \leq t$. When $(\Omega, \F, \P)$ is a complete probability space, $(\F_t)_{t \in [0,T]}$ is said to satisfy the usual conditions  if it is right-continuous (\ie, $\F_s = \bigcap_{t > s} \F_t$ for all $s$) and if $\mathcal{N}_{\P} \subset \F_0$, where $\mathcal{N}_{\P} = \{A \subset \Omega \text{ s.t. } A \subset B \text{ with } B \in \F \text{ and } \P(B) = 0\}$ (if $(\Omega, \F, \P)$ is complete, this means that $\F_0$ contains $\P$-null sets).
 
A stochastic process is a parametrized collection of random variables $(S(t))_{t \in [0,T]}$ defined on $(\Omega,\F,\P)$ and assuming values in $\R^d$ (equipped with $\sigma$-algebra of Borel sets). Given $t \in [0,T]$ and $\omega \in \Omega$, we will write $S(t,\omega) = S(t)(\omega)$ for the realization of the random variable $S(t)$ at $\omega$. A path of the stochastic process is a curve in $\R^d$ obtained as the realization $t \mapsto S(t,\omega)$ for some $\omega \in \Omega$. A stochastic process $(S(t))_{t \in [0,T]}$ is adapted to a filtration $(\F_t)_{t \in [0,T]}$ if $S(t)$ is $\F_t$-measurable for every $t \in [0,T]$. 

Let $(\F_t)_{t \in [0,T]}$ be a filtration. A $d$-dimensional Brownian motion (or Wiener process) is a $\R^d$-valued stochastic process $(W(t))_{t \in [0,T]}$, adapted to $(\F_t)_{t \in [0,T]}$, a.s.\ with continuous paths such that: $W(0) = 0$ a.s.; $W(t) - W(s) \sim \Norm(0,(t-s)\mathrm{Id}_d)$; $W(t) - W(s)$ is independent of $\F_s$ for $t \geq s$.\footnote{One can speak of a Brownian motion without introducing the filtration $(\F_t)_{t \in [0,T]}$ by replacing the condition that $W(t) - W(s)$ is independent of $\F_s$ with the requirement that it has independent increments. In this case, one implicitly considers a filtration constructed from $(W(t))_{t \in [0,T]}$ by letting $\F^W_t$ be $\sigma$-algebra generated by $\{W(s) \ | \ s \leq t\}$. If the filtration needs to satisfy the usual conditions, $\F^W_t$ is modified with the augmentation $\F_t$ defined as the $\sigma$-algebra generated by $\F^W_t$ and $\mathcal{N}_\P$, see~\cite[p.\ 16]{Mao} or~\cite[Proposition~2.7.7]{KarShr}. \label{footnote:augmented}} Equivalently, it has components $W(t) = (W_1(t), \dots, W_d(t))$ with $(W_1(t))_{t \in [0,T]}, \dots, (W_d(t))_{t \in [0,T]}$ independent 1-dimensional Brownian motions.
 
\subsection{Stochastic Differential Equation} For the general theory about SDEs, we refer to~\cite{KarShr, Mao, Oks}. We recall here some basic facts. Let $(\F_t)_{t \in [0,T]}$  be a filtration satisfying the usual conditions, let $(W(t))_{t \in [0,T]}$ be a $d$-dimensional Brownian motion, and let us consider an initial datum $S^0$ given by a $\F_0$-measurable random variable.\footnote{If this is not the case, the construction explained in Footnote~\ref{footnote:augmented} is modified by considering the $\sigma$-algebra generated by $S^0$, $\{W(s) \ | \ s \leq t\}$, and $\mathcal{N}_\P$.} 

However, in this paper we are only interested to a specific class of SDEs, \ie, those with a constant dispersion matrix of the form 
\begin{equation} \label{SDE compact}
    \left\{
        \begin{aligned}
            \d S(t) & = b(t,S(t)) \d t + \sigma \, \d W(t) \, , \\ 
            S(0) & = S^0 \ \text{a.s.}
        \end{aligned}
    \right.
\end{equation}
A stochastic process $(S(t))_{t \in [0,T]}$ is a strong solution to~\eqref{SDE compact} if $(S(t))_{t \in [0,T]}$ has a.s.\ continuous paths, it is adapted to the filtration $(\F_t)_{t \in [0,T]}$, satisfies a.s.\ $\int_0^T |b(t,S(t))| \d t < \infty$ and for every $t \in [0,T]$
\[
S(t) = S^0 + \int_0^t b(s,S(s)) \, \d s  +  \sigma \, \d W(s) \, .
\] 

For this class of SDEs, it is well-known that the well-posedness theory is simpler~\cite[Equation~(2.34)]{KarShr} and requires weaker assumptions on the initial datum $S^0$ than those usually stated in general theorems. For the reader's convenience we state and prove the result in the form needed in this paper, as we did not find a precise reference in the literature. Besides, some of the tools used in the proof will be exploited later in the paper. The result is stated with the Euclidean norm $|\cdot|$ on $\R^d$, but we remark that it holds true when replacing it with any equivalent norm, \eg, also $\max_h |S_h|$, as long as the assumptions on $b$ are satisfied with that norm.
\begin{proposition} \label{prop:SDE simple} 
    Let $b \colon [0,T] \x \R^d \mapsto \R^d$ be a Carath\'eodory function satisfying 
    \begin{itemize}
        \item $|b(t,S)| \leq C_b(1+|S|)$ for every $t \in [0,T]$ and $S \in \R^d$;
        \item for every $R > 0$ there exists $C_R$ such that $|b(t,S) - b(t,S')| \leq \mathrm{Lip}_{R}(b)|S - S'|$ for all $t \in [0,T]$ and $S,S'\in \R^d$ such that $|S|,|S'|\leq R$.
    \end{itemize}
    Let $\sigma \in \R^{d \x d}$. Let $(W(t))_{t \in [0,T]}$ be a $\R^d$-valued Brownian motion and let $S^0$ be a random variable such that a.s.\ $|S^0| < +\infty$. Then there exists a unique strong solution $(S(t))_{t \in [0,T]}$ to~\eqref{SDE compact}. Moreover, if $\E(|S^0|) < +\infty$, then $\E(\|S\|_\infty) \leq C(1+\E(|S^0|) )$, where the constant $C$ depends on $C_b$, $T$, and $W$. 
\end{proposition} 
\begin{proof}
    The scheme of the proof is the classical one, see~\cite[Theorem~3.3]{Mao}.

    Let us fix $\omega \in \Omega$ such that $|S^0(\omega)| < +\infty$ and $t \mapsto W(t,\omega)$ is continuous, which occurs almost surely. We consider the Picard iterations 
    \begin{align} 
        \tilde S^0(t,\omega) & := S^0(\omega) \quad \text{for } t \in [0,T] \, , \label{eq:2212062029} \\
        \tilde S^{j+1}(t,\omega) & := S^0(\omega) + \int_0^t b(s,\tilde S^{j}(s,\omega)) \, \d s + W(t,\omega) \quad \text{for } t \in [0,T]\, , \ j \geq 0 \, , \label{eq:2212062030}
    \end{align}
    Note that the curve $t \mapsto \tilde S^j(t,\omega)$ is continuous.
    First of all, let us prove that for all $j$ and for all $t \in [0,T]$
    \begin{equation} \label{eq:2212062232}
        |\tilde S^j(t,\omega)| \leq (e^{C_b t} - 1)  + (|S^0(\omega)| + \|W(\cdot,\omega)\|_\infty ) e^{C_b t} \, ,
    \end{equation}
    where $C_b$ is the constant appearing in $|b(t,S)| \leq C_b(1+|S|)$. For $j = 0$, \eqref{eq:2212062232} is trivially satisfied. Assume that~\eqref{eq:2212062232} is true for $j$. Then, by~\eqref{eq:2212062030} and by the linear growth of $b$, 
    \[ 
        \begin{split}
            & |\tilde S^{j+1}(t,\omega)|  \leq |S^0(\omega)| + \int_0^t C_b(1+|\tilde S^{j}(s,\omega)|) \, \d s + |W(t,\omega)|\\
            &  \leq C_b t + |S^0(\omega)| + \|W(\cdot,\omega)\|_\infty + \int_0^t C_b (e^{C_b s} - 1)  + C_b(|S^0(\omega)| + \|W(\cdot,\omega)\|_\infty ) e^{C_b s} \, \d s \\
            & \leq C_b t + |S^0(\omega)| + \|W(\cdot,\omega)\|_\infty  + (e^{C_b t} - 1) - C_b t +  (|S^0(\omega)| + \|W(\cdot,\omega)\|_\infty)  ( e^{C_bt} - 1) \\
            & = (e^{C_b t} - 1)  + (|S^0(\omega)| + \|W(\cdot,\omega)\|_\infty ) e^{C_b t} \, ,
        \end{split}
    \]
    which proves~\eqref{eq:2212062232}. In particular, 
    \begin{equation} \label{eq:2212062319}
        \|\tilde S^j(\cdot,\omega)\|_\infty \leq  (1 + |S^0(\omega)| + \|W(\cdot,\omega)\|_\infty ) e^{C_b T} =: R(\omega) \, .
    \end{equation}

    Since $b$ is locally Lipschitz, there exists a constant $\mathrm{Lip}_{R(\omega)}(b)$ such that $|b(t,S) - b(t,S')| \leq \mathrm{Lip}_{R(\omega)}(b)|S - S'|$ for all $t \in [0,T]$ and $S,S'\in \R^d$ such that $|S|,|S'|\leq R(\omega)$. Thanks to this, we show that 
    \begin{equation} \label{eq:2212062011} 
        \sup_{0 \leq s \leq t} |\tilde S^{j+1}(s,\omega) - \tilde S^{j}(s,\omega)|  \leq C(\omega) \frac{(\mathrm{Lip}_{R(\omega)}(b) t)^j}{j!} \, ,
    \end{equation}
    for a suitable constant $C(\omega)$ depending on $\omega$. Indeed, for $j = 0$, by the linear growth of $b$ we have that for every $s \in [0,T]$  
    \[
       |\tilde S^1(s,\omega) - \tilde S^0(s,\omega)| \leq  \int_0^s |b(r, S^0(\omega)) | \, \d r + |W(s,\omega)| \leq C_b s (1 + |S^0(\omega)|) + |W(s,\omega)|\,,
    \]
    hence 
    \begin{equation} \label{eq:Comega}
        \sup_{0\leq s \leq t}   |\tilde S^1(s,\omega) - \tilde S^0(s,\omega)| \leq C_b T (1 + |S^0(\omega)|) + \|W(\cdot,\omega)\|_\infty =: C(\omega)\, .
    \end{equation}
    Moreover, by the local Lipschitz continuity of $b$ we have that for every $s \in [0,T]$  
    \[
        \begin{split}
            |\tilde S^{j+1}(s,\omega) - \tilde S^{j}(s,\omega)| & \leq \int_0^s |b(r,\tilde S^j(r,\omega)) - b(r,\tilde S^{j-1}(r,\omega))| \, \d r \\
            & \leq \mathrm{Lip}_{R(\omega)}(b)\int_0^s |\tilde S^j(r,\omega) - \tilde S^{j-1}(r,\omega)| \, \d r  \, .
        \end{split}
    \]  
    Assuming~\eqref{eq:2212062011} true for $j-1$, we have that 
    \[
        \begin{split}
            & \sup_{0\leq s \leq t} |\tilde S^{j+1}(s,\omega) - \tilde S^{j}(s,\omega)| \leq \mathrm{Lip}_{R(\omega)}(b)\int_0^t \sup_{0\leq r \leq s}|\tilde S^j(r,\omega) - \tilde S^{j-1}(r,\omega)| \, \d s \\
            & \leq  \mathrm{Lip}_{R(\omega)}(b) \int_0^t C(\omega) \frac{(\mathrm{Lip}_{R(\omega)}(b)s)^{j-1}}{(j-1)!} \, \d s = C(\omega) \frac{(\mathrm{Lip}_{R(\omega)}(b) t)^j}{j!} \, .
        \end{split}
    \]
    This implies that  $\tilde S^j(\cdot, \omega)$ is a Cauchy sequence in the uniform norm, since for $j \geq i$
    \begin{equation}\label{eq:2212071453}
        \|\tilde S^{j}(\cdot,\omega) - \tilde S^{i}(\cdot,\omega)\|_\infty  \leq C(\omega) \sum_{h=i}^{+\infty}\frac{(\mathrm{Lip}_{R(\omega)}(b) T)^h}{h!}  \to 0 \quad \text{as } i \to +\infty \, ,
    \end{equation}
    Thus there exists a continuous curve $S(\cdot,\omega)$ such that  
    \begin{equation*} 
        \|\tilde S^j(\cdot,\omega) - S(\cdot, \omega)\|_\infty \to 0 \, .
    \end{equation*}
    
    We have constructed $S(\cdot,\omega)$ for a.e.\ $\omega \in \Omega$. The stochastic processes $(\tilde S^j(t))_{t \in [0,T]}$ are adapted to the filtration $(\F_t)_{t \in [0,T]}$ and have a.s.\ continuous paths. This implies that the limit $(S(t))_{t \in [0,T]}$ is a stochastic process adapted to the filtration $(\F_t)_{t \in [0,T]}$ and has a.s.\ continuos paths. Moreover, passing to the limit in~\eqref{eq:2212062029} for a.e.\ $\omega \in \Omega$, it is a strong solution to~\eqref{SDE compact}. 

    Uniqueness is proven in a more general setting in~\cite[Theorem~2.5]{KarShr} via a stopping time argument.

    Assume now $\E(|S^0|) < + \infty$ and let us prove the estimate on $\E(\|S\|_\infty)$. Passing to the limit in \eqref{eq:2212062319}, we get that for a.e.\ $\omega \in \Omega$
\begin{equation} \label{eq:2212071153}
    \|S(\cdot,\omega)\|_\infty \leq  (1 + |S^0(\omega)| + \|W(\cdot,\omega)\|_\infty ) e^{C_b T}  \, .
\end{equation}
By Doob's maximal inequality~\cite[Chapter~1, Theorem~3.8-(iv)]{KarShr} we have that 
    \begin{equation*}
        \E ( \|W\|_\infty^2 )  \leq 4 \E( |W(T)|^2) \, ,
    \end{equation*}
    and thus, by H\"older's inequality, 
    \begin{equation} \label{eq:Doob}
        \E ( \|W\|_\infty ) \leq \Big(\E ( \|W\|_\infty^2) \Big)^\frac{1}{2} \leq 2 \E( |W(T)|^2)^\frac{1}{2} \, .
    \end{equation}
Hence, taking the expectation in~\eqref{eq:2212071153},
\[
\E(\|S\|_\infty) \leq  (1 + \E(|S^0|) + 2 \E\big( |W(T)|^2 \big)^\frac{1}{2} ) e^{C_b T}  \, ,
\]
which concludes the proof.  
\end{proof}

\begin{remark} \label{rmk:Picard itearations}
    A comment about the Picard iterations used in the proof of Proposition~\ref{prop:SDE simple} is in order. If $b$ is globally Lipschitz, \ie, $|b(t,S) - b(t,S')| \leq \mathrm{Lip}(b) |S - S'|$ and $\E(|S^0|) < +\infty$, then the convergence of the Picard iterations can be improved. Indeed, $\E(|S^0|) < +\infty$ and~\eqref{eq:Doob} yield $\E(C(\omega)) < +\infty$, where $C(\omega)$ is the constant defined in~\eqref{eq:Comega}. Then, taking the expectation in~\eqref{eq:2212071453} and replacing $\mathrm{Lip}_{R(\omega)}(b)$ by the global Lipschitz constant $\mathrm{Lip}(b)$, we deduce that $\E(\|\tilde S^j - S\|_\infty) \to 0$.  
\end{remark}

\subsection{Wasserstein space} Given a complete metric space $(B,d)$, we let $\Pcal_1(B)$ denote the $1$-Wasserstein space, \ie, the space of Borel probability measures $\mu \in \Pcal(B)$ such that 
\[ 
\int_B d(x,x_0) \d \mu(x) < +\infty \, ,  
\]
where $x_0 \in B$ is fixed. The $1$-Wasserstein space is equipped with the $1$-Wasserstein distance defined for every $\mu_1, \mu_2 \in \Pcal_1(B)$ by (see~\cite[Definition~6.1]{Vil})
\[
\W_1(\mu_1,\mu_2) :=  \inf_{\gamma} \int_{B \x B} d(x,x') \, \d \gamma(x,x') \, ,    
\] 
where the infimum is taken over all transport plans $\gamma \in \Pcal(B \x B)$ with marginals $\pi^1_\# \gamma = \mu_1$ and $\pi^2_\# \gamma = \mu_2$, where $\pi^i$ is the projection on the $i$-th component. 

We shall often exploit the dual formulation of the $1$-Wasserstein distance. By Kantorovich's duality, \cite[Theorem~5.10]{Vil} we have that 
\[
    \W_1(\mu_1,\mu_2) = \sup_{\substack{ \psi \in L^1(\mu_1) \\ \psi \text{ $d$-convex}}} \Big( \int_{B} \psi^d(x') \, \d \mu_2(x') - \int_{B} \psi(x) \, \d \mu_1(x) \Big) \, ,  
\]
where $\psi^d$ is the $d$-transform $\psi^d(x') = \inf_{x \in B} (\psi(x) + d(x,x'))$. Since $d$ is a distance on a metric space, a $d$-convex function $\psi$ is a Lipschitz function with Lipschitz constant 1 and it coincides with its $d$-transform, \cf~\cite[Particular Case 5.4]{Vil}. Hence, if $\psi$ is a Lipschitz function with Lipschitz constant $\mathrm{Lip}(\psi)$, we have that 
\[
     \Big| \int_{B} \psi(x) \, \d \Big( \mu_2 - \mu_1\Big) (x) \Big| \leq \mathrm{Lip}(\psi) \W_1(\mu_1,\mu_2)  \, .
\]
When in this paper we refer to Kantorovich's duality, we apply this inequality. Note that the condition $\psi \in L^1(\mu_1) \cap L^1(\mu_2)$ is satisfied since $|\psi(x)| \leq |\psi(0)| + \mathrm{Lip}(\psi) |x|$ and $\mu_1, \mu_2 \in \Pcal_1(B)$.
 
\subsection{Wiener space} Given an interval $[0,T]$, we shall consider the so-called Wiener space of $\R^d$-valued continuous functions $C^0([0,T];\R^d)$, equipped with the uniform norm. Given $t \in [0,T]$, we consider the evaluation function $\ev_t \colon C^0([0,T];\R^d) \to \R^d$ defined by $\ev_t(\varphi) := \varphi(t)$ for every $\varphi \in C^0([0,T];\R^d)$. The family of evaluation functions $\{\ev_t\}_{t \in [0,T]}$ generates a $\sigma$-algebra on $C^0([0,T];\R^d)$, which coincides with the Borel $\sigma$-algebra with respect to the uniform norm in $C^0([0,T];\R^d)$.\footnote{The reason for this is that the evaluation maps $\ev_t$ are continuous with respect to the uniform norm, thus Borel measurable; conversely, open balls in the Wiener space (which is separable) are measurable with respect to the $\sigma$-algebra generated by $\{\ev_t\}_{t \in [0,T]}$, since $\|\varphi\|_\infty = \sup_{t \in [0,T] \cap \Q} |\ev_t(\varphi)|$.}  This is generated by cylindrical sets of the form $\{\varphi \in C^0([0,T];\R^d) \ : \ \varphi(t_1) \in A_1 , \dots , \varphi(t_k) \in A_k\}$, where $A_1, \dots, A_k \subset \R^d$ are Borel sets.

Let $(S(t))_{t \in [0,T]}$ be a $\R^d$-valued stochastic process a.s.\ with continuous paths. This means that there exists an event $E \in \F$ such that $\P(E) = 1$ and $t \mapsto S(t,\omega)$ is continuous for all $\omega \in E$. We can redefine $S(t,\omega) = 0$ for all $t \in [0,T]$ when $\omega \in \Omega \sm E$. This new stochastic process is indistinguishable from the previous one and satisfies $S(\cdot,\omega) \in C^0([0,T];\R^d)$ for all $\omega \in \Omega$. The stochastic process $(S(t))_{t \in [0,T]}$ can be regarded as the random variable $S \colon \Omega \to C^0([0,T];\R^d)$ such that $\omega \mapsto S(\cdot , \omega)$. 
 
The $\sigma$-algebra generated by this random variable is the $\sigma$-algebra generated by sets of the form $S^{-1}(A)$ where $A \subset C^0([0,T];\R^d)$ is a cylindrical Borel set. This means that $A = \{\varphi \in C^0([0,T];\R^d) \ : \ \varphi(t_1) \in A_1 , \dots , \varphi(t_k) \in A_k\}$, where $A_1, \dots, A_k \subset \R^d$ are Borel sets. For these sets we have that  
\[
    S^{-1}(A) = \{\omega \in \Omega \ : \ S(\cdot, \omega) \in A \} = \{\omega \in \Omega \ : \ S(t_1, \omega) \in A_1, \dots, S(t_k, \omega) \in A_k \}
\]
and thus the $\sigma$-algebra generated by $S \colon \Omega \to C^0([0,T];\R^d)$ coincides with the $\sigma$-algebra generated by the family $\{S(t)\}_{t \in [0,T]}$ of random variables $S(t,\cdot) \colon \Omega \to \R^d$, \ie, the $\sigma$-algebra generated by the stochastic process.

In particular, if $(S_1(t))_{t \in [0,T]}, \dots, (S_K(t))_{t \in [0,T]}$ are stochastic processes a.s.\ with continuous paths, then they are independent as stochastic processes if and only they are independent as random variables $S_1, \dots, S_K \colon \Omega \to C^0([0,T];\R^d)$. 

Finally, we remark that a random variable $S \colon\Omega \to C^0([0,T];\R^d)$ induces the probability measure $S_\# \P$ on the space $C^0([0,T];\R^d)$. We let $\Law(S) := S_\# \P \in \Pcal\big( C^0([0,T];\R^d) \big)$.\footnote{This discussion applies, in particular, to a Brownian motion $(W(t))_{t \in [0,T]}$. The probability measure $\Law(W)$ is known as Wiener measure on $C^0([0,T];\R^d)$.} 

If $\mu \in \Pcal\big( C^0([0,T];\R^d) \big)$, we let $\mu(t) := (\ev_t)_\# \mu \in \Pcal(\R^d)$.

\subsection{Empirical measures} \label{subsec:empirical measure} Given random variables $X_1,\dots, X_K \colon \Omega \to \R^d$ with $\E(|X_k|) < +\infty$, we define their empirical measure as the random measure \footnote{The map $\mu_K \colon \Omega \to \Pcal_1(\R^d)$ is indeed measurable with respect to the Borel $\sigma$-algebra on the $1$-Wasserstein space $\Pcal_1(\R^d)$. To see this, we observe that $\Pcal_1(\R^d)$ endowed with the $1$-Wasserstein distance is separable, see, \eg, \cite[Theorem~6.18]{Vil}, hence the Borel $\sigma$-algebra is generated by balls $\{\mu \in \Pcal_1(\R^d) \ : \ \W_1(\mu,\mu_0) < r \}$. The pre-image of such ball through $\mu_K$ is the event $\{\omega \in \Omega \ : \ \W_1(\frac{1}{K} \sum_{k=1}^K \delta_{X_k(\omega)},\mu_0) < r\}$. This is measurable since the function $(x_1,\dots,x_K) \mapsto \W_1(\frac{1}{K} \sum_{k=1}^K \delta_{x_k},\mu_0)$ is Lipschitz continuous.  \label{footnote:random measures} } $\mu_K \colon \Omega \to \Pcal_1(\R^d)$ given by 
\[
\mu_K(\omega) := \frac{1}{K} \sum_{k = 1}^K \delta_{X_k(\omega)}    
\]
for a.e.\ $\omega \in \Omega$. Note that indeed $\mu_K \in \Pcal_1(\R^d)$ a.s., since 
\[
\E \Big( \int_{\R^d} |x| \, \d \mu_K(x) \Big) = \frac{1}{K} \sum_{k=1}^K \E(|X_k|) <  +\infty \, .
\]

Empirical measures of independent samples from a law approximate the law itself. More precisely, let us fix a law $\mu \in \Pcal_1(\R^d)$ and $(X_k)_{k \in \N}$ a sequence of i.i.d.\ random variables with law $\mu$ (which thus satisfy $\E(|X_k|) < +\infty$). Let $\mu_K$ be the empirical measure of $X_1, \dots, X_K$. Then $\E(\W_1(\mu_K,\mu)) \to 0$ as $K \to + \infty$, see, \eg, \cite[Lemma~4.7.1]{PanZem}. In fact, also precise rates of convergence are available in the literature, see~\cite[Theorem~1]{FouGui}.

\subsection{$\Gamma$-convergence} For the theory of $\Gamma$-convergence we refer to the monograph~\cite{DM}. In this paper it will be used to find the limits of optimal control problems. 

\section{Description of the model} \label{sec:description of the model}

To better describe the phenomena that we aim to capture, we introduce all the ingredients that enter in the model step by step. For the reader's convenience, all the unknowns, the parameters, and the initial data of the model are summarized in Tables~\ref{table:ingredients 1}--\ref{table:ingredients 5}. 

The model is an evolutionary system, analyzed in a fixed time interval $[0,T]$. 

\subsubsection*{Ships} The system describes the evolution of $N$ commercial ships, $M$ pirate (criminal) ships, and $L$ coast guard (patrol) ships, whose trajectories are curves $X_n \colon [0,T] \to \R^2$ for $n \in \{1,\dots, N\}$, $Y_m \colon [0,T] \to \R^2$ for $m \in \{1,\dots, M\}$, and $Z_\ell \colon [0,T] \to \R^2$ for $\ell \in \{1,\dots, L\}$, respectively. 

We shall often collect the trajectories based on their type by considering the matrix-valued curves $X = (X_1,\dots,X_N) \colon [0,T] \to \R^{2 \x N}$, $Y = (Y_1,\dots,Y_M) \colon [0,T] \to \R^{2 \x M}$, and $Z = (Z_1,\dots,Z_M) \colon [0,T] \to \R^{2 \x L}$. The letters $X$, $Y$, $Z$ will unambiguously indicate the type of ship, even when decorated, \eg, as $\bar X$, $\tilde X$, or with superscripts and subscripts.

Hereafter, whenever a variable is related to commercial, pirate, or guard ships, it is be indexed with the superscript $\mathrm{c}$, $\mathrm{p}$, or $\mathrm{g}$, respectively.

\begin{table}[H]
    \resizebox{0.8\textwidth}{!}{%
    \def\arraystretch{1.5}
    \begin{tabularx}{\linewidth}{ >{\hsize=.25\hsize}X  >{\hsize=.40\hsize}X   >{\hsize=.30\hsize}X }
        \toprule
        {\bfseries Item} & {\bfseries Meaning} & {\bfseries Comment} \\
        \toprule  
        $[0,T]$ & time interval & fixed \\
        \midrule
        $X = (X_1,\dots,X_N)$ & trajectories of $N$ commercial ships & unknown of the system \\
        \midrule
        $Y = (Y_1,\dots,Y_M)$ & trajectories of $M$ pirate ships & unknown of the system \\
        \midrule
        $Z = (Z_1,\dots,Z_L)$ & trajectories of $L$ guard ships & unknown of the system \\
        \midrule
        $\cdot^\mathrm{c}$ & related to commercial ships &  \\
        \midrule
        $\cdot^\mathrm{p}$ & related to pirate ships &  \\
        \midrule
        $\cdot^\mathrm{g}$ & related to guard ships &  \\
        \bottomrule
    \end{tabularx}
    }
    \caption{Summary of the notation for ships.}
    \label{table:ingredients 1}
\end{table}

\subsubsection*{Evolution of commercial ships} \step{1} We start by describing the evolution of commercial ships in safe waters (absence of pirate ships) and in absence of congestion in the traffic. We assume that there is a vector field $\r \colon \R^2 \to \R^2$ indicating safe commercial routes. In this ideal setting, commercial ships evolve according to the ODEs
\[
    \left\{
        \begin{aligned}
             \frac{\d X_n}{\d t}(t) & = \r(X_n(t)) \, ,      \\
             X_n(0) & = X_n^0 \, , \quad n = 1,\dots,N \, ,
        \end{aligned}
    \right.
\]
where $X^0 = (X_1^0,\dots, X_N^0) \in \R^{2 \x N}$ is the initial position of commercial ships. 

We shall assume that $\r$ is globally Lipschitz continuos. 

\vspace{1em}

\step{2} To include congestion in the model, we introduce $v^N = (v^N_1,\dots,v^N_N) \colon \R^{2\x N} \to [0,v_{\mathrm{max}}]^N$. The component $v^N_n$ weighs the speed of the trajectory of the $n$-th commercial ship according to the presence of all the other commercial ships:
\[
    \left\{
        \begin{aligned}
             \frac{\d X_n}{\d t}(t) & = v^N_n(X(t))\r(X_n(t)) \, ,      \\
             X_n(0) & = X_n^0 \, , \quad n = 1,\dots,N \, .
        \end{aligned}
    \right.
\] 

The assumptions on $v^N$ needed throughout the paper are the following: $v^N$ is Lipschitz continuous with respect to the $\max$ norm with a Lipschitz constant independent of $N$, \ie, $|v^N(X) - v^N(X')| \leq C \max_n |X_n - X'_n|$. 

For $v^N$ we have in mind a precise expression, that will be used in Section~\ref{sec:N to infty}. We consider a globally Lipschitz smooth convolution kernel $\eta \colon \R^2 \x \R^2 \to [0,1]$ satisfying $\eta(X,0) = 0$. The quantity 
\[
\sum_{n'=1}^N \eta( X_n(t), X_n(t) - X_{n'}(t))  
\] 
suitably counts\footnote{For example, let $\hat \eta \in C^\infty_c(\R^2)$ be supported in a ball $B_{2 \delta}$ of radius $2 \delta$ with $\hat \eta = 1$ on $B_{\delta}$. If $\eta(X,X') = \hat \eta(X - X')$, then $\sum_{n'=1}^N \hat \eta( X_n(t) - X_{n'}(t))$ (approximately) counts  the number of ships in a $\delta$-neighborhood of $X_n(t)$ (around all directions). Instead, If $\eta(X,X') = \hat \eta(X - X' - \delta \r(X))$, then $\sum_{n'=1}^N \hat \eta( X_n(t) - X_{n'}(t) - \delta \r(X_n(t)) )$ (approximately) counts the number of commercial ships obstructing the commercial route in front of $X_n(t)$.}  the number of commercial ships around the $n$-th commercial ship at time~$t$. Hence, the quantity 
\[
    \frac{1}{N-1}\sum_{n'=1}^N \eta( X_n(t), X_n(t) - X_{n'}(t))  
\] 
can be regarded as the density of commercial ships around the $n$-th commercial ship at time~$t$. The precise expression of the scaling factor $\frac{1}{N-1}$ is relevant only to interpret the previous expression as a density and can, in fact, be replaced by a sequence converging to zero with the same rate of $\frac{1}{N}$.  Given a Lipschitz function $v \colon [0,1] \to [0,v_{\mathrm{max}}]$, the corrected speed of the $n$-th commercial ship depends on the density of its surrounding ships as follows:
\[
    v^N_n(X(t)) = v\Big(\frac{1}{N-1} \sum_{n' = 1}^N \eta\big(X_n(t),X_n(t) - X_{n'}(t)\big)\Big) \, .
\]
To model congestion, $v$ must be assumed to be non-increasing in the density.

\vspace{1em}

\step{3} Eventually, let us modify the dynamics of commercial ships in presence pirate ships. We consider a globally Lipschitz vector-valued interaction kernel $\Kcp \colon \R^2 \to \R^2$ (here $\mathrm{cp}$ stands for commercial-pirate). To model repulsion of the $n$-th commercial ship from the pirate ships, we modify the direction of the trajectory $X_n(t)$ by averaging the vectors $\Kcp(X_n(t) - Y_m(t))$, \ie,
\[
    \left\{
        \begin{aligned}
             \frac{\d X_n}{\d t}(t) & = v^N_n(X(t)) \Big( \r(X_n(t)) + \frac{1}{M} \sum_{m=1}^M \Kcp(X_n(t) - Y_m(t)) \Big) \, ,      \\
             X_n(0) & = X_n^0 \, , \quad n = 1,\dots,N \, .
        \end{aligned}
    \right.
\] 
For $\Kcp$ we have in mind the following expression
\begin{equation} \label{eq:Kcp and Hcp}
    \Kcp(X_n(t) - Y_m(t)) = \Hcp(X_n(t) - Y_m(t)) (X_n(t) - Y_m(t)) \, ,
\end{equation}
where $\Hcp$ has compact support with a radius given by the length for which the presence of a pirate ship at $Y_m(t)$ affects the trajectory $X_n(t)$. An example of $\Hcp$ is $\Hcp(w) = \frac{h(|w|)}{|w|}$, where $h$ is compactly supported in $(0,+\infty)$, so that $\Kcp(X_n(t) - Y_m(t)) = h(|X_n(t) - Y_m(t)|) \frac{X_n(t) - Y_m(t)}{|X_n(t) - Y_m(t)|}$ and $\frac{X_n(t) - Y_m(t)}{|X_n(t) - Y_m(t)|}$ is, for $X_n(t)$, the direction pointing opposite to~$Y_m(t)$. 

\begin{table}[H]
    \resizebox{0.8\textwidth}{!}{%
    \def\arraystretch{1.5}
    \begin{tabularx}{\linewidth}{ >{\hsize=.25\hsize}X  >{\hsize=.375\hsize}X   >{\hsize=.375\hsize}X }
        \toprule
        {\bfseries Item} & {\bfseries Meaning} & {\bfseries Comment} \\
        \toprule  
        $\eta$ & kernel to compute density of commercial ships & smooth and globally Lipschitz \\
        \midrule
        $v$ & velocity as a function of the density & Lipschitz continuous  \\
        \midrule
        $v^N = (v^N_1,\dots,v^N_N)$ & obtained from $\eta$ and $v$ &  Lipschitz continuous, with Lipschitz constant independent of~$N$ \\
        \bottomrule
    \end{tabularx}
    }
    \caption{Summary of functions used in the model for evolution of commercial ships.}
    \label{table:ingredients 2}
\end{table}

\subsubsection*{Evolution of pirate ships} \step{1} Pirate ships are are repelled by guard ships and are attracted by commercial ships. To model this, we consider globally Lipschitz vector-valued interaction kernels $\Kpg \colon \R^2 \to \R^2$ and $\Kpc \colon \R^2 \to \R^2$. Then 
\[
    \left\{
        \begin{aligned}
             \frac{\d Y_m}{\d t}(t) & = \frac{1}{L} \sum_{\ell=1}^L \Kpg(Y_m(t) - Z_\ell(t)) -   \frac{1}{N} \sum_{\ell=1}^N \Kpc(Y_m(t) - X_n(t))  \, ,      \\
             Y_m(0) & = Y_m^0 \, , \quad m = 1,\dots,M \, .
        \end{aligned}
    \right.
\]
where $Y^0 = (Y_1^0,\dots, Y_M^0) \in \R^{2 \x M}$ is the initial position of pirate ships. 

For the precise form of $\Kpg$, $\Kpc$, see the analogous discussion for commercial ships done after~\eqref{eq:Kcp and Hcp}.

\step{2} In absence of commercial and guard ships, pirate ships explore the environment in search of targets by navigating randomly. To model this, we add a stochastic term in the evolution of pirate ships, by considering $M$ Brownian motions $(W_1(t))_{t \in [0,T]}, \dots, (W_M(t))_{t \in [0,T]}$. The pirate ships then evolve according to the following SDEs
\[
    \left\{
        \begin{aligned}
             \d Y_m(t) & = \Big( \frac{1}{L} \sum_{\ell=1}^L \Kpg(Y_m(t) - Z_\ell(t)) -   \frac{1}{N} \sum_{\ell=1}^N \Kpc(Y_m(t) - X_n(t)) \Big) \, \d t + \sqrt{2 \kappa} \, \d W_m(t)  \, ,      \\
             Y_m(0) & = Y_m^0 \quad \text{a.s.,} \quad m = 1,\dots,M \, ,
        \end{aligned}
    \right.
\]
where $\kappa > 0$.

\subsubsection*{Evolution of guard ships} The last part of the system describes guard ships. In absence of other ships, guard ships tend to repel each other. To model this, we consider globally Lipschitz vector-valued interaction kernel $\Kgg \colon \R^2 \to \R^2$. In this setting, the guard ships evolve according to 
\[
    \left\{
        \begin{aligned}
             \frac{\d Z_\ell}{\d t}(t) & =  \frac{1}{L} \sum_{\ell'=1}^L \Kgg(Z_\ell(t) - Z_{\ell'}(t)) \, ,      \\
             Z_\ell(0) & = Z_\ell^0 \, , \quad \quad \ell = 1,\dots,L \, ,
        \end{aligned}
    \right.
\]
where $Z^0 = (Z_1^0,\dots, Z_L^0) \in \R^{2 \x L}$ is the initial position of guard ships. We do not require more on the dynamics of guard ships, as we want the global dynamics of the system to be governed by the optimal control policy for guard ships. 

\begin{table}[H]
    \resizebox{0.8\textwidth}{!}{%
    \def\arraystretch{1.5}
    \begin{tabularx}{\linewidth}{ >{\hsize=.25\hsize}X  >{\hsize=.375\hsize}X   >{\hsize=.375\hsize}X }
        \toprule
        {\bfseries Item} & {\bfseries Meaning} & {\bfseries Comment} \\
        \toprule
        $\Kcp$ & effect of pirate ships on commercial ships & Lipschitz continuous  \\
        \midrule 
        $\Kpg$ & effect of guard ships on pirate ships & Lipschitz continuous   \\
        \midrule 
        $\Kpc$ & effect of commercial ships on pirate ships & Lipschitz continuous  \\
        \midrule 
        $\Kgg$ & effect of guard ships on guard ships & Lipschitz continuous  \\
        \bottomrule
    \end{tabularx}
    }
    \caption{Interaction kernels used in the model.}
    \label{table:ingredients 3}
\end{table}

\subsubsection*{Controls} We consider a set of admissible controls $\U \subset \R^{2\x L}$. We assume $\U$ to be compact. A fixed control $u = (u_1, \dots, u_L) \in L^\infty([0,T];\U)$ drives the evolution of guard ships as follows:
\[
    \left\{
        \begin{aligned}
             \frac{\d Z_\ell}{\d t}(t) & =  \frac{1}{L} \sum_{\ell'=1}^L \Kgg(Z_\ell(t) - Z_{\ell'}(t)) + u_\ell(t) \, ,      \\
             Z_\ell(0) & = Z_\ell^0 \, , \quad \quad \ell = 1,\dots,L \, .
        \end{aligned}
    \right.
\]

\subsubsection*{Full model} In conclusion we are interested in the following ODE/SDE/ODE model:
\begin{equation} \label{eq:full model}
    \left\{
        \begin{aligned}
            & \d X_n(t)  = v^N_n(X(t))\Big(  \r(X_n(t)) + \frac{1}{M} \sum_{m=1}^M \Kcp(X_n(t) - Y_m(t) )    \Big)\d t \, ,      \\
            & \d Y_m(t)  = \Big( \frac{1}{L} \sum_{\ell=1}^L \Kpg(Y_m(t) - Z_\ell(t))  - \frac{1}{N} \sum_{n=1}^N \Kpc(Y_m(t) - X_n(t)) \Big) \d t + \sqrt{2\kappa} \, \d W_m(t) \, ,  \\
            & \frac{\d Z_\ell}{\d t}(t)  =  \frac{1}{L}\sum_{\ell'=1}^L \Kgg(Z_\ell(t) - Z_{\ell'}(t))  + u_\ell(t)  \, ,  \\ 
            & X_n(0) = X_n^0 \ \text{ a.s.,} \  \ Y_m(0) = Y_m^0 \ \text{ a.s.,}  \  \ Z_\ell(0) =  Z_\ell^0 \, , \\
            & n = 1,\dots,N \, , \ m = 1,\dots,M \, , \ \ell = 1,\dots, L \, .
        \end{aligned}
    \right.
\end{equation}
(The first equation is expressed as an SDE to stress that the solution $X$ is a stochastic process. However, given a trajectory $Y$, the first equation is, in fact, and ODE.)

We prove well-posedness for \eqref{eq:full model} in Subsection~\ref{subsec:well-posedness for original model}.

\subsubsection*{Initial data} The initial data in~\eqref{eq:full model} will be given by $X^0 = (X^0_1,\dots,X^0_N) \in \R^{2 \x N}$ with $|X^0_n| \leq R_0$ for some $R_0 > 0$; $\R^2$-valued i.i.d.\ random variables $Y^0_1,\dots,Y^0_M$; $Z^0 = (Z^0_1,\dots,Z^0_L) \in \R^{2 \x L}$. 
 
\begin{table}[H]
    \resizebox{0.8\textwidth}{!}{%
    \def\arraystretch{1.5}
    \begin{tabularx}{\linewidth}{ >{\hsize=.25\hsize}X  >{\hsize=.375\hsize}X   >{\hsize=.375\hsize}X }
        \toprule
        {\bfseries Item} & {\bfseries Meaning} & {\bfseries Comment} \\
        \toprule
        $X^0 = (X^0_1,\dots,X^0_N)$ & initial positions of commercial ships & points in $\R^2$, $|X^0_n| \leq R_0$ \\
        \midrule 
        $Y^0_1,\dots,Y^0_M$ & initial positions of pirate ships  &  random variables in $\R^2$ \\
        \midrule 
        $Z^0 = (Z^0_1,\dots,Z^0_L)$ & initial positions of guard ships & points in $\R^2$ \\
        \bottomrule
    \end{tabularx}
    }
    \caption{Summary of initial data.}
    \label{table:ingredients 4}
\end{table}

\subsubsection*{Optimal control} As previously mentioned, the dynamics of guard ships will be driven by an optimal control. To define the cost, we consider a bounded and globally Lipschitz function $\Hd \colon \R^2 \to \R$. If the quantity $\Hd(X_n(t) - Y_m(t))$ is significantly different from zero when $Y_m(t)$ is close to $X_n(t)$ and is small when $Y_m(t)$ is far from $X_n(t)$ (\eg, when $\Hd$ is compactly supported), this function can be used to count contacts between commercial and pirate ships (the superscript $\mathrm{d}$ stands for ``danger''). Hence we consider the cost functional $\J_{N,M} \colon L^\infty([0,T];\U) \to \R$ defined for every control $u \in L^\infty([0,T];\U)$ by 
\begin{equation} \label{def:expected cost}
    \J_{N,M}(u) :=  \frac{1}{2}  \int_0^T    |u(t)|^2 \d t + \E \Big(\int_0^T \frac{1}{N} \frac{1}{M} \sum_{n=1}^N \sum_{m=1}^M \Hd(X_n(t) - Y_m(t))  \, \d t  \Big) \, ,
\end{equation}
where the stochastic processes $(X(t))_{t \in [0,T]} = (X_1(t) ,\dots, X_N(t)))_{t \in [0,T]}$ and $(Y(t))_{t \in [0,T]} = (Y_1(t), \dots, Y_M(t)))_{t \in [0,T]}$ are given by the unique strong solutions to~\eqref{eq:full model} corresponding to the control~$u$ obtained in Proposition~\ref{prop:solution to ODE/SDE}. 

The objective is to minimize the cost $\J_{N,M}$.

\begin{table}[H]
    \resizebox{0.8\textwidth}{!}{%
    \def\arraystretch{1.5}
    \begin{tabularx}{\linewidth}{ >{\hsize=.25\hsize}X  >{\hsize=.375\hsize}X   >{\hsize=.375\hsize}X }
        \toprule
        {\bfseries Item} & {\bfseries Meaning} & {\bfseries Comment} \\
        \toprule
        $\U \subset \R^{2 \x L}$ & set of admissible controls & compact \\
        \midrule 
        $\Hd$ & kernel for dangerous contacts in cost functional & bounded and Lipschitz continuous \\
        \midrule 
        $\J_{N,M}$ & cost functional associated to~\eqref{eq:full model}, for fixed $N,M$ & defined in~\eqref{def:expected cost} \\
        \midrule 
        $\J_N$ & cost functional defined in~\eqref{def:JN}, for fixed $N$ & obtained as the $\Gamma$-limit of $\J_{N,M}$ as $M \to +\infty$ in Theorem~\ref{thm:Gamma convergence for M} \\
        \midrule 
        $\J$ & cost functional defined in~\eqref{def:J} & obtained as the $\Gamma$-limit of $\J_N$ as $N \to +\infty$ in Theorem~\ref{thm:Gamma convergence for N} \\
        \bottomrule
    \end{tabularx}
    }
    \caption{Summary of the items regarding control.}
    \label{table:ingredients 5}
\end{table}

\section{Well-posedness of the ODE/SDE/ODE model} \label{sec:well-posedness}

\subsection{Well-posedness of the ODE/SDE/ODE model for a fixed control} \label{subsec:well-posedness for original model} 

In this section we prove well-posedness for the model presented in~\eqref{eq:full model}.

We remark that the solutions depend on $N$ and $M$. Not to overburden the notation, in this section we drop the dependence on $N$ and $M$, as we will not consider limits as $N \to + \infty$ or $M \to + \infty$.

\begin{proposition}\label{prop:solution to ODE/SDE}
    Assume the following:
    \begin{itemize}
        \item Let $(W_1(t))_{t \in [0,T]}, \dots, (W_M(t))_{t \in [0,T]}$ be independent Brownian motions;
        \item Let $X^0 = (X^0_1,\dots,X^0_N) \in \R^{2\x N}$;
        \item Let $Y^0_1,\dots,Y^0_M$ be $\R^2$-valued random variables, with $|Y^0_m| < +\infty$ a.s.\ for $m=1,\dots,M$; 
        \item Let $Z^0 = (Z^0_1,\dots,Z^0_L) \in \R^{2 \x L}$;
        \item Let $u \in L^\infty([0,T];\U)$. 
    \end{itemize}
    Then there exists a unique strong solution to~\eqref{eq:full model}, $(X(t))_{t \in [0,T]} = (X_1(t),\dots,X_N(t))_{t \in [0,T]}$, $(Y(t))_{t \in [0,T]} = (Y_1(t),\dots, Y_M(t))_{t \in [0,T]}$, and $Z = (Z_1,\dots,Z_L)$. Moreover, if $\E(|Y^0_m|) < + \infty$ for $m=1,\dots,M$, then $\E(\max_m \|Y_m\|_\infty) < +\infty$.
\end{proposition}
\begin{proof}

    We start by noticing that the ODEs involving the variables $Z_\ell$ are decoupled from the equations involving $X_n$ and $Y_m$. Given a control $u = (u_1,\dots,u_L) \in L^\infty([0,T];\U)$, we solve 
    \[
        \left\{
            \begin{aligned}
                \frac{\d Z_\ell}{\d t}(t) & =   \frac{1}{L}\sum_{\ell'=1}^L \Kgg( Z_\ell(t) - Z_{\ell'}(t) )  + u_\ell(t)  \, ,  \\ 
                Z_\ell(0) & =  Z_\ell^0 \, , \quad  \ell = 1,\dots, L \, .
            \end{aligned}
        \right.
    \]
    We observe that there exists a unique solution for all times $t \in [0,T]$ to the previous ODE system. Too see this, we introduce the function $f = f_u = (f_{u,1},\dots,f_{u,L}) \colon [0,T] \x \R^{2 \x L} \to \R^{2\x L}$ (we drop the dependence on $u$ for ease of notation) defined by 
    \[
    f_{\ell}(t,Z) :=  \frac{1}{L}\sum_{\ell'=1}^L \Kgg(Z_\ell - Z_{\ell'} )  + u_\ell(t) \, , \quad \text{for } \ell = 1, \dots, L  
    \] 
    and we notice that the system reads
    \begin{equation} \label{eq:ODE for Z}
        \left\{
            \begin{aligned}
                \frac{\d Z}{\d t}(t) & =   f(t,Z(t))  \, ,  \\ 
                Z(0) & =  Z^0 \, , 
            \end{aligned}
            \right.
        \end{equation}
    where $Z = (Z_1,\dots,Z_L)$. The right-hand side $f(t,Z)$ is a Carath\'eodory function, globally Lipschitz-continuous in the~$Z$ variable (with Lipschitz constant independent of $t$). 
    These properties are sufficient for the well-posedness of the ODE.\footnote{The result is classical: one considers the Picard operator  $\S \colon C^0([0,T];\R^{2 \x L}) \to C^0([0,T];\R^{2 \x L})$ defined by $\S(Z)(t) := Z^0 + \int_0^t f(s,Z(s)) \, \d s$, which is a contraction with respect to the norm (equivalent to the uniform norm) $\mynorm \varphi \mynorm_{\alpha} := \sup_{t \in [0,T]} \big(e^{-\alpha t} |\varphi(t)|\big)$ for a suitable $\alpha > 0$ (depending on the Lipschitz constant of~$f$). 
    }
    We remark that solutions to~\eqref{eq:ODE for Z} are bounded. Indeed,
    \begin{equation} \label{eq:f has linear growth}
        |f(t,Z)| \leq |f(t,0)| + |f(t,Z) - f(t,0)| \leq \|\Kgg\|_\infty + \|u\|_\infty + C |Z| \leq C(1 + |Z|) \, , 
    \end{equation}
    hence 
    \[
    |Z(t)| \leq |Z^0| + \int_0^t C (1 + |Z(s)|) \, \d s \leq  |Z^0| + CT + \int_0^t |Z(s)| \, \d s
    \]
    and, by Gr\"onwall's inequality, for $t \in [0,T]$
    \begin{equation} \label{eq:Z is bounded}
        |Z(t)|  \leq (|Z^0| + CT) e^{C t}  \leq (|Z^0| + CT) e^{C T} \, ,
    \end{equation}
    where the constant $C$ depends on $\Kgg$ and $\U$ (compact).
 
 We exploit the solution $Z(t)$ to solve the ODE/SDE/ODE system, which now we  write in a more compact way. Let us introduce the $\R^{2 \x (M+N)}$-valued stochastic process $(S(t))_{t\in[0,T]}$ defined~by
\[
S(t) := (Y_1(t), \dots, Y_M(t), X_1(t), \dots, X_N(t))  
\]
(we put the components $Y_1(t), \dots, Y_M(t)$ in the first block for consistency later). We consider the drift vector $b_{Z} = b = (b_1, \dots, b_{M+N}) \colon [0,T] \x \R^{2 \x (M+N)} \to \R^{2 \x (M+N)}$ (we drop the dependence on $Z$ for the ease of notation) defined for every $S = (S_1,\dots, S_{M+N}) \in \R^{2 \x (M+N)}$~by 
\begin{equation} \label{eq:bi 1,M}
    b_i(t,S) :=   \frac{1}{L} \sum_{\ell=1}^L \Kpg(S_i - Z_\ell(t))  - \frac{1}{N} \sum_{j=M+1}^{M+N} \Kpc(S_i - S_j) \, ,  
\end{equation}
for $i = 1, \dots, M$, and
\begin{equation}  \label{eq:bi M+1,M+N}
    b_i(t,S) :=   v^N_{i-M}\big((S_j)_{j={M+1}}^{M+N}\big)\Big(  \r(S_i) + \frac{1}{M} \sum_{j=1}^{M} \Kcp(S_i - S_j)    \Big) \, ,
\end{equation}
for $i = M+1, \dots, M+N$. Moreover, let $\sigma \in  \R^{(2 \x 2) \x   (M+N)}$ be the constant dispersion tensor given by the collection $\sigma = (\sigma_1,\dots,\sigma_{M+N})$ of the matrices $\sigma_i \colon \R^{2 \x (M+N)} \to \R^{2 \x 2}$ defined by 
\[
\sigma_i := \sqrt{2 \kappa}\, \Id_2 \quad \text{for } i = 1,\dots, M \, ,    
\]
and $\sigma_i := 0$ for $i = M+1,\dots,M+N$. For $W = (W_1,\dots,W_{M+N})\in \R^{2\x(M+N)}$ we adopt the short-hand notation $\sigma W$ to denote the element in $\R^{2\x (M+N)}$ with columns $(\sigma W)_1, \dots, (\sigma W)_{M+N} \in \R^2$ given by $(\sigma W)_i = \sigma_i W_i$. 

By setting $S^0 := (Y_1^0, \dots, Y_M^0, X_1^0, \dots, X_N^0)$, the system reads 
\begin{equation} \label{eq:SDE compact}
    \left\{
        \begin{aligned}
            \d S(t) & = b(t,S(t)) \d t + \sigma  \, \d W(t) \, , \\ 
            S(0) & = S^0 \ \text{a.s.,}
        \end{aligned}
    \right.
\end{equation}
where $(W(t))_{t \in [0,T]}$ is a $\R^{2\x (M+N)}$-valued  Brownian motion. Note that $W_1(t), \dots, W_M(t)$ correspond to the $M$ independent $\R^2$-valued Brownian motions already introduced for~\eqref{eq:full model}. This is the reason why we chose to put the $Y_m$'s in place of the $X_n$'s in the first block of~$S$. 

    We are now left to check that the conditions for the existence and uniqueness stated in Proposition~\ref{prop:SDE simple} are satisfied by~\eqref{eq:SDE compact}. By the continuity of $Z(t)$,  the function $t \mapsto b(t,S)$ is continuous for every $S$. Let $i \in \{1,\dots,M\}$, so that $b_i$ is given by~\eqref{eq:bi 1,M}. By the Lipschitz continuity of $\Kpg$, we have that 
    \begin{equation} \label{eq:Lipschitz implies linear growth}
        |\Kpg(z_1 - z_2)|  \leq  |\Kpg(z_1 - z_2) - \Kpg(0)| + |\Kpg(0)|  \leq C |z_1 - z_2| + |\Kpg(0)| \leq C(1 + |z_1| + |z_2|) \, .
    \end{equation} 
    Reasoning analogously for $\Kpc$, it follows that  
    \begin{equation} \label{eq:b has linear growth I}
        \begin{split}    
            |b_i(t, S)| & \leq \frac{1}{L} \sum_{\ell=1}^L | \Kpg( S_i - Z_\ell(t))|   + \frac{1}{N} \sum_{j=M+1}^{M+N} | \Kpc( S_i - S_j )| \\
            &  \leq \frac{1}{L}\sum_{\ell = 1}^L C (1 + |S_i| + |Z_\ell(t)| ) + \frac{1}{N}\sum_{j = M+1}^{M+N} C (1 + |S_i| + |S_j|) \leq C (1 + \max_{ h }|S_h|) \, ,
        \end{split}
    \end{equation} 
    where we used the continuity, and thus boundedness, of $Z_\ell(t)$ for $t \in [0,T]$. Let us check the Lipschitz continuity condition. By the Lipschitz continuity of $\Kpg$ and $\Kpc$, we have that  
    \begin{equation} \label{eq:b is globally Lipschitz}
        \begin{split} 
            |b_i(t,S) - b_i(t, S')| &  \leq  \frac{1}{L} \sum_{\ell=1}^L | \Kpg(S_i  - Z_\ell(t)) - \Kpg(S'_i - Z_\ell(t))|  \\
            & \quad + \frac{1}{N} \sum_{j=M+1}^{M+N} | \Kpc(S_i - S_j) - \Kpc(S'_i  - S'_j)| \\
            & \leq \frac{1}{L} \sum_{\ell=1}^L C | S_i - S'_i|  + \frac{1}{N} \sum_{j=M+1}^{M+N} C |  S_i - S_j  - S'_i + S'_j|  \\
            & \leq \frac{1}{L} \sum_{\ell=1}^L C  |S_i   - S'_i | + \frac{1}{N} \sum_{j=M+1}^{M+N} C \big(|  S_i  -  S'_i| + |S_j  - S'_j|  \big) \\
            & \leq C \max_h |S_h - S'_h| \, ,
        \end{split}  
    \end{equation}  
    where the constant $C$ depends on $\Kpg$, $\Kpc$. (In fact, $b_i$ is even globally Lipschitz continuous for $i \in \{1,\dots,M\}$).
    
    Let now $i \in \{M+1,\dots,M+N\}$, so that $b_i$ is given by~\eqref{eq:bi M+1,M+N}. By the boundedness of~$v^N$, by the bound $\r(x) \leq C(1+|x|)$, and reasoning for $\Kcp$ as in~\eqref{eq:Lipschitz implies linear growth}, we have that  
    \begin{equation} \label{eq:b has linear growth II}
        |b_i(t,S)| \leq \|v^N\|_\infty \Big(  C(1+|S_i|) + \frac{1}{M} \sum_{j=1}^{M} |\Kcp(S_i  - S_j)|   \Big)  \leq C (1 + \max_h |S_h|) \, .
    \end{equation} 
    To check the local Lipschitz-continuity of $b_i$, let us fix $R > 0$. For $t \in [0,T]$ and $\max_h |S_h| \leq R$, $\max_h|S'_h| \leq R$, by the boundedness and the Lipschitz property of $v^N$ (recall that it has a Lipschitz constant independent of $N$), and by the Lipschitz continuity of $\r$ and $\Kcp$, we have that 
    \begin{equation} \label{eq:b is locally Lipschitz}  
        \begin{split}
            & | b_i(t, S) -  b_i(t, S')| \\
            &  \leq  \Big|  v^N_{i-M}\big((S_j)_{j={M+1}}^{M+N}\big) - v^N_{i-M}\big((S'_j)_{j={M+1}}^{M+N}\big) \Big| \cdot  \Big|    \r(S_i) + \frac{1}{M} \sum_{j=1}^{M} \Kcp(S_i - S_j)     \Big| \\
            & \quad  + \Big| v^N_{i-M}\big((S'_j)_{j={M+1}}^{M+N}\big)  \Big|  \cdot \Big|   \r(S_i) + \frac{1}{M} \sum_{j=1}^{M} \Kcp(S_i - S_j)   -   \r(S'_i) - \frac{1}{M} \sum_{j=1}^{M} \Kcp(S'_i - S'_j)  \Big| \\
            & \quad \leq  \max_h \Big( C |S_h-S'_h| (1+|S_h|)  +  C |S_h - S'_h|  \Big) \\
            & \quad \leq  \max_h \Big( C |S_h-S'_h|(1 + |S_h|) \Big) \leq C \max_h |S_h-S'_h|(1+R)\, ,
        \end{split}
    \end{equation} 
        where the constant $C$ depends on $v^N$, $\r$, and $\Kcp$ (independent of $N$). Choosing $C_R = C(1+R)$ we get the desired inequality.
         
    Applying Proposition~\ref{prop:SDE simple}, we conclude the proof of existence and uniqueness. Moreover, we also get $\E(\max_h \|S_h\|_\infty) < +\infty$ and, in particular, $\E(\max_m \|Y_m\|_\infty) < +\infty$.
\end{proof}

\subsection{Existence of an optimal control for the ODE/SDE/ODE model} Let $\J_{N,M}$ be the cost defined in~\eqref{def:expected cost}. We have the following result concerning existence of optimal controls. 

\begin{proposition} \label{prop:existence of optimal control for ODE/SDE}
    Under the assumptions of Proposition~\ref{prop:solution to ODE/SDE}, there exists an optimal control $u^* \in L^\infty([0,T];\U)$, \ie,  
    \[
      \J_{N,M}(u^*) = \min_{u \in L^\infty([0,T];\U)} \J_{N,M}(u)\, .  
    \] 
\end{proposition}
\begin{proof} The result is obtained via the direct method in the Calculus of Variations. We divide the proof in steps for the sake of presentation.

    \vspace{1em}

\step{1} (Preliminary steps) Let $u^j \in L^\infty([0,T];\U)$ be a minimizing sequence, \ie, $\J_{N,M}(u^j) \to \min \J_{N,M}$ as $j \to +\infty$. Since $u^j$ is bounded in $L^\infty([0,T];\U)$, there exists $u^* \in L^\infty([0,T];\U)$ and a subsequence (not relabeled) such that $u^j \wstar u^*$ weakly-* in $L^\infty([0,T];\U)$. We claim that $u^*$ is an optimal control. 
    
    To prove the claim, let us fix $(X^j(t))_{t \in [0,T]} = (X^j_1(t),\dots,X^j_N(t))_{t \in [0,T]}$, $(Y^j(t))_{t \in [0,T]} = (Y^j_1(t),\dots,Y^j_M(t))_{t \in [0,T]}$, and $Z^j = (Z^j_1,\dots,Z^j_L)$, the strong solutions to~\eqref{eq:full model} corresponding to the controls $u^j$ obtained in Proposition~\ref{prop:solution to ODE/SDE}. We adopt the notation of the proof of Proposition~\ref{prop:solution to ODE/SDE} and let $S = (Y_1, \dots, Y_M, X_1, \dots, X_N)$. In this way, for every $j$ we have that 
    \[
        \left\{
            \begin{aligned}
                \frac{\d Z^j}{\d t}(t) & =   f_{u^j}(t,Z^j(t))  \, ,  \\ 
                Z^j(0) & =  Z^0 \, , 
            \end{aligned}
        \right.
    \]
    (we stress the dependence of $f_{u^j}$ on the controls $u^j$) and 
    \[
        \left\{
            \begin{aligned}
                \d S^j(t) & = b_{Z^j}(t,S^j(t)) \d t + \sigma  \, \d W(t) \, , \\ 
                S^j(0) & = S^0 \ \text{a.s.,}  \\
            \end{aligned}
        \right.
    \]
    (we stress the dependence of the drift vector $\R^{2\x L}$ on the trajectories $Z^j$). 

    \vspace{1em}

    \step{2} (Identifying the limit of $Z^j$) We remark that~\eqref{eq:Z is bounded} yields $\|Z^j\|_\infty \leq C$ for every $j$, where $C$ depends on $Z^0$, $T$, $\Kgg$, and $\U$. Let us check that the $Z^j$'s are also equicontinuous. By~\eqref{eq:f has linear growth}, for every $j$ and for $s \leq t$ we have that 
    \[
        \begin{split}
            |Z^j(t) - Z^j(s)| & \leq \int_s^t |f_{u^j}(r, Z^j(r))| \, \d r  \leq \int_s^t (\|\Kgg\|_\infty + \|u^j\|_\infty + C \|Z^j\|_\infty) \, \d r \\
            & \leq (\|\Kgg\|_\infty + \|u^j\|_\infty + C \|Z^j\|_\infty) |t - s| \leq C |t - s| \, ,
        \end{split}
    \]
    where $C$ depends on $Z^0$, $T$, $\Kgg$, and $\U$ (compact). By Arzel\`a-Ascoli's theorem we obtain $Z^* \in C^0([0,T];\R^2)$ such that $\| Z^j - Z^*\|_\infty \to 0$, up to a subsequence, that we do not relabel. This, together with the convergence $u^j \wstar u^*$ and
    \[
     Z^j_\ell(t) = Z^0 + \int_0^t   \Big( \frac{1}{L}\sum_{\ell'=1}^L \Kgg(Z^j_\ell(s) - Z^j_{\ell'}(s) )  + u^j_\ell(s) \Big)  \, \d s  
    \]
    yields, letting $j \to +\infty$, 
    \[
        Z^*_\ell(t) = Z^0 + \int_0^t   \Big( \frac{1}{L}\sum_{\ell'=1}^L \Kgg(Z^*_\ell(s) - Z^*_{\ell'}(s) )  + u^*_\ell(s) \Big)  \, \d s \, , 
    \] 
    \ie, $Z^*$ is the solution to
    \[
        \left\{
            \begin{aligned}
                \frac{\d Z^*}{\d t}(t) & =   f_{u^*}(t,Z^*(t))  \, ,  \\ 
                Z^*(0) & =  Z^0 \, .
            \end{aligned}
        \right.
    \]

    \vspace{1em}

    \step{3} (Identifying the limit of $S^j$) We let $(S^*(t))_{t \in [0,T]}$ be the $\R^2$-valued stochastic process obtained as the strong solution to 
        \[
            \left\{
                \begin{aligned}
                    \d S^*(t) & = b_{Z^*}(t,S^*(t)) \d t + \sigma  \, \d W(t) \, , \\ 
                    S^*(0) & = S^0 \ \text{a.s.,} \\
                \end{aligned}
            \right.
        \] 
    We claim that a.s.\ $\max_h \|S^j_h - S^*_h\|_\infty \to 0$ as $j \to +\infty$. We start by observing that a.s.\  for $0 \leq s \leq t \leq T$ and $i=1,\dots,M+N$
    \begin{equation} \label{eq:0111221903}
        \begin{split}
            & |S^j_i(s) - S^*_i(s)| \leq \int_0^s |b_{i, Z^j}(r,S^j(r)) - b_{i, Z^*}(r,S^*(r))| \, \d r \\
            & \quad  \leq \int_0^s |b_{Z^j,i}(r,S^j(r)) - b_{Z^j,i}(r,S^*(r))| \, \d r  +  \int_0^s |b_{Z^j,i}(r,S^*(r)) - b_{Z^*,i}(r,S^*(r))| \, \d r \, .
        \end{split}
    \end{equation}
    We estimate the former integrand by exploiting the Lipschitz property of $b_{Z^j,i}$ obtained in~\eqref{eq:b is globally Lipschitz} and~\eqref{eq:b is locally Lipschitz}  
    \begin{equation} \label{eq:0111221904}        
        \begin{split}
            |b_{Z^j,i}(r,S^j(r)) - b_{Z^j,i}(r,S^*(r))| & \leq  \max_h \Big( C(1 + |S^*_h(r)|) |S^j_h(r) - S^*_h(r)| \Big) \\
            & \leq \max_h \Big( C (1 + \|S^*_h\|_\infty)  \sup_{0 \leq r \leq s} |S^j_h(r) - S^*_h(r)| \Big) \quad \text{a.s.}\, ,
        \end{split}
    \end{equation}  
    where the constant $C$ depends on $\Kpg$, $\Kpc$, $\Kcp$, $v^N$, and $\r$ (independent of $N$). To estimate the latter integrand in~\eqref{eq:0111221903}, we resort to the definition of $b_{Z}$. By~\eqref{eq:bi 1,M}, for $i=1,\dots,M$ we get that 
    \begin{equation} \label{eq:0111221905}       
        \begin{split}
            |b_{Z^j, i}(r,S^*(r)) - b_{Z^*, i}(r,S^*(r))| & \leq \frac{1}{L} \sum_{\ell=1}^L | \Kpg(S^*_i - Z^j_\ell(r))  - \Kpg(S^*_i - Z^*_\ell(r))| \\
            & \leq  \frac{1}{L} \sum_{\ell=1}^L C | Z^j_\ell(r) - Z^*_\ell(r)| \leq C \|Z^j - Z^*\|_\infty \quad \text{a.s.} \, ,
        \end{split}
    \end{equation}
where the constant $C$ depends on $\Kpg$.  For $i=M+1,\dots,M+N$, by~\eqref{eq:bi M+1,M+N} we have instead that $|b_{Z^j, i}(r,S^*(r)) - b_{Z^*, i}(r,S^*(r))| = 0$. We observe that Proposition~\ref{prop:SDE simple} also gives us that 
$\E(\max_h \|S^*_h \|_\infty) < C(1+\E(\max_h |S^0_h|))$, thus a.s.\ $\max_h \|S^*_h\|_\infty < +\infty$.

We are now in a position to prove that a.s.\ $\max_h \|S^j_h - S^*_h\|_\infty \to 0$. For $k \geq 1$, let us  consider the events 
\[
A_k := \{\omega \in \Omega \ : \ \max_h \|S^*_h(\cdot, \omega)\|_\infty \leq k \} \, .    
\]
We remark that $\P\big( \bigcup_k A_k\big) = 1$, since a.s.\ $\max_h \|S^*_h\|_\infty < +\infty$. Let us fix $\omega \in A_k$ and such that~\eqref{eq:0111221903}--\eqref{eq:0111221905} hold true. Then we have that   
\[
    \begin{split}
        & \max_h \sup_{0 \leq s \leq t} |S^j_h (s,\omega) - S^*_h (s,\omega)| \\
        & \quad \leq  C (1 + \max_k \|S^*_k (\cdot, \omega)\|_\infty) \int_0^t \max_h  \sup_{0 \leq r \leq s} |S^j_h (r,\omega) - S^*_h (r,\omega)| \, \d s + CT \|Z^j - Z^*\|_\infty  \, .
    \end{split} 
\] 
Integrating on $A_k$, we get that 
\[
    \begin{split}
        & \int_{A_k} \max_h \sup_{0 \leq s \leq t} |S^j_h (s,\omega) - S^*_h(s,\omega)|   \, \d \P(\omega) \\
        & \quad \leq CT \|Z^j - Z^*\|_\infty + C (1 + k) \int_0^t \int_{A_k} \max_h \sup_{0 \leq r \leq s}|S^j_h (r,\omega) - S^*_h (r,\omega)| \, \d \P(\omega) \, \d s \, .
    \end{split}
\]
By Gr\"onwall's inequality we deduce that 
\[
    \int_{A_k} \max_h \sup_{0 \leq s \leq t} |S^j_h (s,\omega) - S^*_h (s,\omega)|   \, \d \P(\omega) \leq CT \|Z^j - Z^*\|_\infty e^{C(1+k) t} \, ,
\]
and, in particular, 
\[
    \int_{A_k} \max_h \|S^j_h (\cdot  ,\omega) - S^*_h (\cdot,\omega)\|_\infty   \, \d \P(\omega) \leq CT \|Z^j - Z^*\|_\infty e^{C(1+k) T} .
\]
By {\itshape Step~2} we have that $\| Z^j - Z^*\|_\infty \to 0$ as $j \to +\infty$ and thus $\max_h \| S^j_h (\cdot  ,\omega) - S^*_h (\cdot,\omega)\|_\infty \to 0$ for a.e.\ $\omega \in A_k$. Since $\P\big( \bigcup_k A_k\big) = 1$, we conclude that a.s.\ $\max_h \|S^j_h - S^*_h\|_\infty \to 0$. \EEE

\vspace{1em}

\step{4} (Limit of the cost) Let us show that 
\[
\J_{N,M}(u^*) \leq \liminf_{j \to +\infty} \J_{N,M}(u^j) \, .    
\]
Since $u^j$ is a minimizing sequence, this will be sufficient to conclude that $\J_{N,M}(u^*) = \min_{u} \J_{N,M}(u)$. 

By sequential weak semicontinuity of the $L^2$-norm we get that 
\[
\frac{1}{2} \int_0^T  |u^*(t)| \, \d t \leq  \liminf_{j \to + \infty} \frac{1}{2} \int_0^T  |u^j(t)| \, \d t \,  .
\]
From {\itshape Step 3} we have that a.s.\ $\max_h \|S^j_h - S^*_h \|_\infty \to 0$, thus a.s.\ $\max_m \|Y^j_m - Y^*_m\|_\infty \to 0$ and $\max_n \|X^j_n - X^*_n\|_\infty \to 0$ (recall that $S = (Y_1,\dots,Y_M, X_1,\dots,X_N)$). Then, using the fact that $H^\d$ is bounded, by the Dominated Convergence Theorem  
\[
\E\Big( \int_0^T \frac{1}{N} \frac{1}{M} \sum_{n,m}  \Hd(X^j_n(t) - Y^j_m(t))  \, \d t\Big)  \to \E\Big( \int_0^T \frac{1}{N} \frac{1}{M} \sum_{n,m}  \Hd(X^*_n(t) - Y^*_m(t))  \, \d t\Big) 
\]
as $j \to +\infty$. By the superadditivity of the $\liminf$ we conclude the proof. 
\end{proof}

\section{An averaged ODE/SDE/ODE system}  \label{sec:averaged model}

\subsection{Introducing the averaged ODE/SDE/ODE system}

To study the mean-field limit of \eqref{eq:full model} as $M \to +\infty$, we consider an  averaged ODE/SDE/ODE system, where the trajectories $Y_m(t)$ are replaced by a single trajectory $\bar Y(t)$, interacting with the other agents via its probability distribution. More precisely, let  $(W(t))_{t \in [0,T]}$ be a $\R^2$-valued Brownian motion and consider the problem
\begin{equation} \label{eq:ODE Y average}
    \left\{
        \begin{aligned}
            & \frac{\d \bar X_n}{\d t}(t) = v^N_n(\bar X(t))\Big(  \r(\bar X_n(t)) + \Kcp*\mubp(t)(\bar X_n(t))    \Big) \, ,      \\
            & \d \bar Y(t) = \Big( \frac{1}{L} \sum_{\ell=1}^L \Kpg(\bar Y(t) - Z_\ell(t))  - \frac{1}{N} \sum_{n=1}^N \Kpc(\bar Y(t) - \bar X_n(t)) \Big) \d t + \sqrt{2\kappa} \, \d W(t) \, ,  \\
            & \frac{\d Z_\ell}{\d t}(t) =   \frac{1}{L}\sum_{\ell'=1}^L \Kgg(Z_\ell(t) - Z_{\ell'}(t))  + u_\ell(t)   \, ,  \\ 
            & \bar X_n(0) =  X_n^0 \, , \  \ Z_\ell(0) =  Z_\ell^0 \, , \quad n = 1,\dots,N \, , \ \ell = 1,\dots, L \, , \\
            & \bar Y(0) = \bar Y^0 \ \text{a.s.,}  \quad  \mubp = \Law(\bar Y) \, .
        \end{aligned}
    \right.
\end{equation} 

We start by giving a precise definition for the notion of solutions to the previous system.

\begin{definition}
    A \emph{strong solution} to~\eqref{eq:ODE Y average} is given by a curve $\bar X = (\bar X_1, \dots, \bar X_N) \in C^0([0,T]; \R^{2 \x N})$, an $\R^2$-valued stochastic process $(\bar Y(t))_{t \in [0, T]}$ a.s.\ with continuous paths, and a curve $Z = (Z_1, \dots, Z_L) \in C^0([0,T]; \R^{2\x L})$ such that 
    \begin{enumerate}
        \item a.s.\ for all $t \in [0,T]$
        \[
                \bar Y(t) = \bar Y^0  + \int_0^t \Big(\frac{1}{L} \sum_{\ell=1}^L \Kpg(\bar Y(s) - Z_\ell(s))  - \frac{1}{N} \sum_{n=1}^N \Kpc(\bar Y(s) - \bar X_n(s)) \Big) \d t + \sqrt{2\kappa} \, W(t)
        \]
        \item setting $\mubp := \Law(\bar Y) \in \Pcal\big(C^0([0,T];\R^2)\big)$, the curves $\bar X$ and $Z$ satisfy 
        \[
                \bar X_n(t)  = X_n^0 + \int_0^t v^N_n(\bar X(s))\Big(  \r(\bar X_n(s)) + \Kcp*\mubp(t)(\bar X_n(s))    \Big)\d s  
        \]
        and
        \[
            Z_\ell(t) =  Z_\ell^0  + \int_0^t \Big( \frac{1}{L}\sum_{\ell'=1}^L \Kgg(Z_\ell(s) - Z_{\ell'}(s))  + u_\ell(s)  \Big) \d s
        \]
        for all $t \in [0,T]$. 
    \end{enumerate}  
\end{definition}

\subsection{Well-posedness of the averaged ODE/SDE/ODE system}

Let us prove the following well-posedness result.

\begin{proposition} \label{prop:solution to averaged ODE/SDE}
    Assume the following:
    \begin{itemize}
        \item Let $(W(t))_{t \in [0,T]}$ be a Brownian motion;
        \item Let $X^0 = (X^0_1,\dots,X^0_N) \in \R^{2\x N}$;
        \item Let $\bar Y^0$ be a random variable, with $\E(|\bar Y^0|) < +\infty$;
        \item Let $Z^0 = (Z^0_1,\dots,Z^0_L) \in \R^{2 \x L}$;
        \item Let $u \in L^\infty([0,T];\U)$. 
    \end{itemize}
    Then there exists a unique strong solution to~\eqref{eq:ODE Y average}. Moreover, $\E(\|\bar Y\|_\infty) < +\infty$ and $\mubp \in \Pcal_1\big(C^0([0,T];\R^2)\big)$.  
\end{proposition}
\begin{proof}
    As  recalled in the proof of Proposition~\ref{prop:solution to ODE/SDE}, for every control $u = (u_1,\dots, u_L)\in L^\infty([0,T];\U)$, there exists a unique continuous solution to 
    \begin{equation} \label{eq:bar Z}
        Z_\ell(t) =   Z_\ell^0  + \int_0^t \Big(\frac{1}{L}\sum_{\ell'=1}^L \Kgg(Z_\ell(s) - Z_{\ell'}(s))  + u_\ell(s)  \Big) \d s  \, , \quad t \in [0,T ]\, ,
    \end{equation}
    hence $Z_\ell(t)$ will be treated as fixed in the following.  
    
    The proof now mainly follows the lines of~\cite[Theorem~3.1]{AscCasSol}. For the sake of brevity, we let $C^0 = C^0([0,T];\R^2)$.

    \vspace{1em}

    \step{1} (Decoupling the system) Let us fix $\mu \in \Pcal_1(C^0)$ ($\mu$ plays the role of $\mubp$ in the equation and is used to apply a fixed-point argument). Let us consider the decoupled system
    \begin{equation} \label{eq:decoupled SDE ODE - I}
        \left\{
            \begin{aligned}
                \frac{\d \tilde X_n}{\d t}(t) & = v^N_n(\tilde X(t))\Big(  \r(\tilde X_n(t)) + \Kcp*\mu(t)(\tilde X_n(t))    \Big)  \, ,      \\
                \tilde X_n(0) & =  X_n^0  \, , \quad n = 1,\dots,N  \, , \\
            \end{aligned}
            \right.
    \end{equation}
    \begin{equation} \label{eq:decoupled SDE ODE - II}
        \left\{
            \begin{aligned}
                \d \tilde Y(t) & = \Big( \frac{1}{L} \sum_{\ell=1}^L \Kpg(\tilde Y(t) - Z_\ell(t))  - \frac{1}{N} \sum_{n=1}^N \Kpc(\tilde Y(t) - \tilde X_n(t)) \Big) \d t + \sqrt{2\kappa} \, \d W(t) \, ,  \\
                \tilde Y(0) & = \bar Y^0 \ \text{a.s.,}   \\
            \end{aligned}
            \right.
    \end{equation}
    where the $Z_\ell(t)$ are obtained in~\eqref{eq:bar Z}.  

    \substep{1}{1} We start by commenting about the existence (and uniqueness) of continuous curves $\tilde X = (\tilde X_1, \dots, \tilde X_N) \in C^0([0,T];\R^{2\x N})$ solutions to~\eqref{eq:decoupled SDE ODE - I}. For this, we need to check the conditions for well-posedness of ODE systems. Let us consider the function $g_\mu = (g_{\mu,1},\dots,g_{\mu,N}) \colon [0,T] \x \R^{2 \x N} \to \R^{2 \x N}$ defined by 
    \begin{equation} \label{eq:def of gmu} 
        g_{\mu,n}(t, X) :=  v^N_n(X)\Big(  \r(X_n) + \Kcp*\mu(t)(X_n)    \Big)    
    \end{equation}
    for $n = 1, \dots, N$. The system then reads 
    \begin{equation} \label{eq:compact ODE tilde X}
        \left\{
        \begin{aligned}
            \frac{\d \tilde X}{\d t}(t) & = g_\mu(t,\tilde X(t))  \, , \\
            \tilde X_n(0) & =  X^0_n \, , \quad n = 1,\dots,N  \, .
        \end{aligned}
        \right.
    \end{equation}
    The dependence of $g_\mu$ on the time variable $t$ is only due to the terms
    \[
        \begin{split}
            & \Kcp*\mu(t)(X_n)  = \int_{\R^2} \Kcp(X_n - x) \d \mu(t)(x) =  \int_{\R^2} \Kcp(X_n - x) \d ((\ev_t)_\# \mu)(x) \\
            & \quad = \int_{C^0} \Kcp(X_n - \ev_t(\varphi)) \d \mu(\varphi) = \int_{C^0} \Kcp(X_n - \varphi(t)) \d \mu(\varphi) \, ,
        \end{split}
    \]
    which are continuous in $t$. This follows from, \eg, the Dominated Convergence Theorem by observing that the Lipschitz continuity of $\Kcp$ yields
    \[
        |\Kcp(X_n - \varphi(t)) | \leq |\Kcp(0)| + C |X_n - \varphi(t)| \leq C (1 + |X_n| + \|\varphi\|_\infty)
    \]
    and $\int_{C^0} \|\varphi\|_\infty \d \mu(\varphi)    < +\infty$,  since $\mu \in \Pcal_1(C^0)$. The functions $g_{\mu,n}$ are locally Lipschitz in $X$, \ie, given $R > 0$, there exists $C_R > 0$ such that for $t \in [0,T]$ and $\max_n |X_n| \leq R$, $\max_n |X'_n| \leq R$ it holds that 
    \begin{equation} \label{eq:gmu is locally Lipschitz}
        \max_n |g_{\mu,n}(t,X) - g_{\mu,n}(t,X')| \leq C_R  \max_n |X_n - X'_n|\,.   
    \end{equation}
    The computations are analogous to those in~\eqref{eq:b is locally Lipschitz}, the only difference being in the term 
    \[
        \begin{split}
            |\Kcp*\mu(t)(X_n) - \Kcp*\mu(t)(X'_n)|  & \leq \int_{\R^2} |\Kcp(X_n - x) - \Kcp(X'_n - x)| \d \mu(t)(x) \\
            & \leq \int_{\R^2} C |X_n  - X'_n | \d \mu(t)(x) \leq C|X_n - X'_n| \, .
        \end{split}
    \]
    In conclusion, $g_\mu(t,X)$ is continuous in $t$ and locally Lipschitz in $X$ with respect to the $\max$ norm. By Picard-Lindelh\"of's theorem, the ODE system~\eqref{eq:compact ODE tilde X} admits a unique solution for small times. For existence for all times, with computation analogous to those in~\eqref{eq:b has linear growth II} we observe that we have linear growth for $g_\mu$, \ie,
    \[
        \max_n |g_{\mu,n}(t, X)| \leq C(1+\max_n |X_n|) \, ,
    \]
    the constant $C$ above depending on $\|v^N\|_\infty$, $\r$, and $\Kcp$. This upper bound allows for a Gr\"onwall inequality. Indeed,  
    \[
      \begin{split}
        & |\tilde X_n(t)| \leq |X^0_n| + \int_0^t   \Big|\frac{\d \tilde X_n}{\d t}(s)\Big|    \, \d s  = |X^0_n| + \int_0^t  |g_{\mu,n}(s,\tilde X(s))|  \, \d s \\
        & \quad \leq \max_{n'} |X^0_{n'}| + \int_0^t C(1+\max_{n'} |X_{n'}|)\, \d s  =  \max_{n'} |X^0_{n'}| + CT + C \int_0^t  \max_{n'} |X_{n'}|\, \d s  \, ,
    \end{split}
    \]
    which yields 
    \begin{equation} \label{eq:boundedness of tildeX} 
        \max_n |\tilde X_n(t)| \leq C(\max_n |X^0_n|+T) e^{Ct}  \, , \quad \text{for all } t \in [0,T] \, ,
    \end{equation}     \EEE
    and, in particular, boundedness of solutions in terms of the initial datum $X^0$ and final time~$T$ (in addition to $\|v^N\|_\infty$, $\r$, and $\Kcp$). This is enough to deduce global existence in time. 

    \substep{1}{2} Given the continuous curves $\tilde X$ and $Z$ obtained previously, we consider the~SDE~\eqref{eq:decoupled SDE ODE - II}. We rewrite this SDE by introducing the drift vector $b_{\tilde X} \colon [0,T] \x \R^2 \to \R^2$ (depending on~$\tilde X$)  
    \begin{equation} \label{eq:def of btildeX}
        b_{\tilde X}(t,Y) :=  \frac{1}{L} \sum_{\ell=1}^L \Kpg(Y - Z_\ell(t))  - \frac{1}{N} \sum_{n=1}^N \Kpc(Y - \tilde X_n(t)) 
    \end{equation}
    and by considering the constant dispersion matrix $\sigma = \sqrt{2 \kappa} \, \Id_2$, so that the SDE reads 
    \begin{equation} \label{eq:tildeY integral solution}
        \left\{
            \begin{aligned}
                \d \tilde Y(t) & = b_{\tilde X}(t,\tilde Y(t)) \d t + \sigma   \, \d W(t) \, ,  \\
                \tilde Y(0) & = \bar Y^0 \ \text{a.s.} \\
            \end{aligned}
            \right.
        \end{equation}
    For the existence and uniqueness of a strong solution to this SDE, we check that the assumptions of Proposition~\ref{prop:SDE simple} are satisfied. The drift $b_{\tilde X}$ is continuous in $t$: it follows from the continuity of the curves $\tilde X_n$ and $Z_\ell$. The drift $b_{\tilde X}$ is  globally Lipschitz continuos in $Y$. Indeed, we have that  
        \begin{equation} \label{eq:btildeX is Lipschitz}
            \begin{split} 
                |b_{\tilde X}(t,Y) - b_{\tilde X}(t, Y')| &  \leq  \frac{1}{L} \sum_{\ell=1}^L | \Kpg(Y - Z_\ell(t)) - \Kpg(Y' - Z_\ell(t))|  \\
                & \quad + \frac{1}{N} \sum_{n=1}^{N} | \Kpc(Y - \tilde X_n(t)) - \Kpc(Y' - \tilde X_n(t))| \leq C |Y - Y'| \, ,
            \end{split} 
        \end{equation} 
        the constant $C$ only depending on the Lipschitz constants of $\Kpg$ and $\Kpc$. Finally, $b_{\tilde X}$ satisfies the linear growth condition. This follows from~\eqref{eq:Lipschitz implies linear growth} and the analogous condition for $\Kpc$, which yield  
        \[
            \begin{split}    
                |b_{\tilde X}(t,Y)| & \leq  \frac{1}{L} \sum_{\ell=1}^L |\Kpg(Y - Z_\ell(t))|  + \frac{1}{N} \sum_{n=1}^N |\Kpc(Y - \tilde X_n(t)) | \\
                & \leq |\Kpg(0)| + \frac{1}{L} \sum_{\ell=1}^L C |Y - Z_\ell(t)| + |\Kpc(0)| + \frac{1}{N} \sum_{n=1}^N C |Y -  \tilde X_n(t)| \\
                & \leq |\Kpg(0)| + |\Kpc(0)| + \|Z\|_{\infty} + \max_n \|\tilde X_n\|_\infty + |Y|  \leq C(1 + |Y|) \, , 
            \end{split}
        \]  
        where the constant $C$ depends on $\Kpg$, $\Kpc$, $\|Z\|_{\infty}$, and $\max_n \|\tilde X_n\|_{\infty}$ and we used the boundedness of $\tilde X$ obtained in~\eqref{eq:boundedness of tildeX}.

    We are in a position to apply Proposition~\ref{prop:SDE simple}, which also gives us that  
    \begin{equation} \label{eq:bound on tildeY}
        \E( \|\tilde Y \|_\infty) \leq C(1 + \E(|\bar Y^0 |)) \, .
    \end{equation}   
    This implies that $\Law(\tilde Y) \in \Pcal_1(C^0)$. Indeed,  
    \[
        \begin{split}
            \int_{C^0} \|\varphi\|_\infty \d \Law(\tilde Y)(\varphi) & =  \int_{C^0} \|\varphi\|_\infty \d (\tilde Y_\# \P)(\varphi) = \int_\Omega \|\tilde Y(\cdot,\omega)\|_\infty \, \d \P(\omega)  = \E(\|\tilde Y\|_\infty) < + \infty \, .
        \end{split}
    \]

    \vspace{1em}

\step{2} (Fixed-point argument) Let us implement the machinery to carry out a fixed-point argument.

\substep{2}{1} (Definition of Picard operator) We consider the functional $\L \colon \Pcal_1(C^0) \to \Pcal_1(C^0)$ defined as follows: given $\mu \in \Pcal_1(C^0)$, we let $\tilde X =(\tilde X_1,\dots,\tilde X_N)$ and $(\tilde Y(t))_{t \in [0,T]}$ be the unique solution to~\eqref{eq:decoupled SDE ODE - I}--\eqref{eq:decoupled SDE ODE - II} obtained as explained in the previous step. Then we set $\L(\mu) := \Law(\tilde Y)$, which belongs to $\Pcal_1(C^0)$ as explained in the previous step. We shall show that $\L$ is a contraction with respect to a suitable auxiliary distance on $\Pcal_1(C^0)$, to deduce the existence of a fixed point. 

\substep{2}{2} (Definition of equivalent Wasserstein distance) The auxiliary distance we consider on $\Pcal_1(C^0)$ is defined as follows. We let $\alpha > 0$ (its choice is made precise later in~\eqref{eq:choice of alpha}) and we define on $C^0$ the norm  
\begin{equation} \label{eq:alpha norm}
    \mynorm \varphi \mynorm_\alpha := \sup_{t \in [0,T]} \Big( e^{-\alpha t} |\varphi(t)| \Big) \, .  
\end{equation}  
Then we define the auxiliary distance on $\Pcal_1(C^0)$ by 
\[
  \W_{1,\alpha}\big( \mu_1, \mu_2 \big) :=   \inf_{\gamma} \int_{C^0 \x C^0} \mynorm \varphi- \psi \mynorm_\alpha \, \d \gamma(\varphi,\psi)  \, ,    
  \] 
  where the infimum is taken over all transport plans $\gamma \in \Pcal(C^0 \x C^0)$ with marginals $\pi^1_\# \gamma = \mu_1$ and $\pi^2_\# \gamma = \mu_2$, where $\pi^i$ is the projection on the $i$-th component. Since the norm $\mynorm \cdot \mynorm_\alpha$ is equivalent to the usual uniform norm $\|\cdot \|_{\infty}$ on $C^0$, the distance $\W_{1,\alpha}$ is equivalent to the usual $1$-Wasserstein distance $\W_1$ on $\Pcal_1(C^0)$.

\substep{2}{3} (Start of proof of contraction property) Given $\mu, \mu' \in \Pcal_1(C^0)$, let us estimate $\W_{1,\alpha}\big( \L(\mu), \L(\mu')  \big)$. Let $\tilde X = (\tilde X_1, \dots, \tilde X_N)$, $\tilde Y$ and $\tilde X' = (\tilde X'_1, \dots, \tilde X'_N)$, $\tilde Y'$ be solutions obtained in {\itshape Step~1} corresponding to $\mu$ and $\mu'$, respectively. By Kantorovich's duality, there exists a functional $\Psi \colon C^0 \to C^0$ Lipschitz continuous with respect to $\mynorm \cdot \mynorm_\alpha$ with Lipschitz constant 1 such that, using the fact that $\L(\mu) = \Law(\tilde Y)$ and $\L(\mu') = \Law(\tilde Y')$, 
    \begin{equation} \label{eq:first estimate for contraction}
        \begin{split}
            \W_{1,\alpha}\big( \L(\mu), \L(\mu')  \big) & = \int_{C^0} \Psi(\varphi) \, \L(\mu)(\varphi) - \int_{C^0} \Psi(\varphi') \, \L(\mu')(\varphi') \\
            & = \E(\Psi(\tilde Y) - \Psi(\tilde Y')) \leq \E(\mynorm\tilde Y - \tilde Y'\mynorm_\alpha) \, .
        \end{split}
    \end{equation}
    The following substeps show how to estimate the term $\E \big( \mynorm \tilde Y - \tilde Y' \mynorm_{\alpha}  \big)$.   
    
    \substep{2}{4} (Estimate of $|\tilde Y(t) - \tilde Y'(t) |$) We start by observing that from~\eqref{eq:tildeY integral solution}, from the definition of $b_{\tilde X}$ in~\eqref{eq:def of btildeX}, by the Lipschitz continuity of $\Kpc$, and by the Lipschitz continuity of $b_{\tilde X'}$ obtained in~\eqref{eq:btildeX is Lipschitz}, we have that a.s.\
    \begin{equation} \label{eq:first estimate of tildeY}
        \begin{split}
            & | \tilde Y(t) - \tilde Y'(t) |  = \Big|    \int_0^t b_{\tilde X}(s,\tilde Y(s)) \d s   -  \int_0^t \tilde b_{\tilde X'}(s,\tilde Y'(s)) \d s \Big| \\
            & \quad \leq \int_0^t \Big( |b_{\tilde X}(s,\tilde Y(s)) - b_{\tilde X'}(s,\tilde Y(s))|   +   |b_{\tilde X'}(s,\tilde Y(s)) - b_{\tilde X'}(s,\tilde Y'(s))| \Big) \, \d s \\
            & \quad \leq \int_0^t  \Big(  \frac{1}{N} \sum_{n=1}^N | \Kpc(\tilde Y(s) - \tilde X_n(s)) - \Kpc( \tilde Y(s) - \tilde X'_n(s))| + C|\tilde Y(s) - \tilde Y'(s)| \Big) \, \d s \\
            & \quad \leq \int_0^t C \Big(  \max_n  |\tilde X_n(s) - \tilde X'_n(s)| +  |\tilde Y(s) - \tilde Y'(s)| \Big) \, \d s \, , 
        \end{split}
    \end{equation} 
    the constant $C$ depending only on the Lipschitz constants of $\Kpg$ and $\Kpc$. 

    \substep{2}{5} (Estimate of $|\tilde X_n(s) - \tilde X'_n(s) |$) The curves $\tilde X$ and $\tilde X'$ are solutions to~\eqref{eq:compact ODE tilde X}. As obtained in~\eqref{eq:boundedness of tildeX}, they are bounded by a constant $R > 0$ depending on the initial datum $X^0$, the final time $T$, and parameters of the problem ($\|v\|_\infty$, $\r$, and $\Kcp$), \ie, $\max_n \|\tilde X_n\|_\infty \leq R$, $\max_n \|\tilde X'_n\|_\infty \leq R$. We recall that $g_{\mu}$ and $g_{\mu'}$ are locally Lipschitz, hence there exists $C > 0$ (depending on $R$) such that~\eqref{eq:gmu is locally Lipschitz} is satisfied. It follows that for $n=1,\dots,N$
    \begin{equation} \label{eq:3110221829}
        \begin{split}
            & |\tilde X_n(s) - \tilde X'_n(s) |  \leq \int_0^s |g_{\mu,n}(r,\tilde X(r)) - g_{\mu',n}(r,\tilde X'(r))| \, \d r \\
            & \quad \leq \int_0^s \Big( |g_{\mu,n}(r,\tilde X(r)) - g_{\mu,n}(r,\tilde X'(r))|  + |g_{\mu,n}(r,\tilde X'(r)) - g_{\mu',n}(r,\tilde X'(r))| \Big)  \, \d r \\
            & \quad \leq  \int_0^s \Big( C \max_{n'}|\tilde X_{n'}(r) - \tilde X'_{n'}(r)|  + |g_{\mu,n}(r,\tilde X'(r)) - g_{\mu',n}(r,\tilde X'(r))| \Big)  \, \d r \, .
        \end{split}
    \end{equation}
   Let us now apply the definition of $g_\mu$ and $g_{\mu'}$ in~\eqref{eq:def of gmu} to estimate for $n = 1, \dots, N$ and $r \in [0,s]$
    \begin{equation} \label{eq:3110221830}
        \begin{split}
             & |g_{\mu,n}(r,\tilde X'(r)) - g_{\mu',n}(r,\tilde X'(r))| \leq \|v\|_\infty \big|\Kcp*\mu(r)(\tilde X'_n(r))  - \Kcp*\mu'(r)(\tilde X'_n(r)) \big|  \\
            &\quad \leq C  \Big| \int_{\R^2} \Kcp( \tilde X'_n(r) - x) \, \d \Big( \mu(r) - \mu'(r) \Big)(x) \Big|  \, ,
        \end{split}
    \end{equation}
    where the constant $C$ depends on $\|v\|_\infty$. We observe that by the Lipschitz continuity of $x \mapsto \Kcp(\tilde X'_n(r) - x)$ and by Kantorovich's duality, 
    \begin{equation} \label{eq:3110221831}
        \Big| \int_{\R^2} \Kcp( \tilde X'_n(r) - x) \, \d \Big( \mu(r) - \mu'(r) \Big)(x) \Big|     \leq  C \W_1(\mu(r), \mu'(r)) \, ,
    \end{equation}
    where $C$ is the Lipschitz constant of $\Kcp$. To bound this term, let us fix an optimal plan $\gamma \in \Pcal(C^0 \x C^0)$ with marginals $\pi^1_\# \gamma = \mu$, $\pi^2_\# \gamma = \mu'$ and satisfying 
    \[
    \W_{1,\alpha}(\mu, \mu') =    \int_{C^0 \x C^0} \mynorm \varphi - \psi \mynorm_\alpha \, \d \gamma(\varphi,\psi)   \, . 
    \]
    We remark that $\gamma(r) = (\ev_r)_\# \gamma \in \Pcal(\R^2 \x \R^2)$ has marginals $\pi^1_\# (\ev_r)_\# \gamma = (\ev_r)_\# \pi^1_\# \gamma = \mu(r)$ and $\pi^2_\# (\ev_r)_\# \gamma = (\ev_r)_\#  \pi^2_\# \gamma =  \mu'(r)$, hence, by optimality of $\W_1$ and by the definition of $\mynorm \cdot \mynorm_\alpha$ in~\eqref{eq:alpha norm}, we obtain for $r \in [0,s]$ 
    \begin{equation*} 
        \begin{split}
           &  \W_1(\mu(r), \mu'(r)) \leq  \int_{\R^2 \x \R^2} |x - x'|\,  \d \gamma(r)(x,x')  = \int_{\R^2 \x \R^2} |x - x'|\,  \d (\ev_r)_\#\gamma(x,x') \\
            & = \int_{C^0 \x C^0} |\varphi(r) - \psi(r)|\,  \d \gamma(\varphi,\psi)  \leq e^{ \alpha r} \int_{C^0 \x C^0}   e^{- \alpha r}|\varphi(r) - \psi(r)|  \,  \d \gamma(\varphi,\psi)   \\
            &  \leq e^{\alpha r} \int_{C^0 \x C^0} \mynorm \varphi  - \psi  \mynorm_\alpha \,  \d \gamma(\varphi,\psi)    = e^{\alpha r} \W_{1,\alpha}(\mu, \mu') \, .
        \end{split} 
    \end{equation*}
    Integrating in $r$, we get that 
    \begin{equation} \label{eq:3110221832}
        \int_0^s    \W_1(\mu(r), \mu'(r)) \, \d r \leq \frac{e^{\alpha s}-1}{\alpha} \W_{1,\alpha}(\mu, \mu')  \leq \frac{e^{\alpha s}}{\alpha} \W_{1,\alpha}(\mu, \mu') \, .
    \end{equation}
    Putting together \eqref{eq:3110221829}--\eqref{eq:3110221832} we conclude that 
    \[
        \max_n |\tilde X_n (s) - \tilde X'_n(s) | \leq C \Big( \int_0^s  \max_n |\tilde X_n (r) - \tilde X'_n (r)|  \, \d r + \frac{e^{\alpha s}}{\alpha} \W_{1,\alpha}(\mu, \mu') \Big) \, \d r\, .
    \]
    By Gr\"onwall's inequality we conclude that 
    \begin{equation} \label{eq:estimate of tildeX}
        \begin{split}
            \max_n |\tilde X_n (s) - \tilde X'_n (s) | & \leq C \frac{e^{\alpha s}}{\alpha} e^{C s} \W_{1,\alpha}(\mu, \mu') \leq C e^{CT}\frac{ e^{\alpha s}}{\alpha} \W_{1,\alpha}(\mu, \mu') \\
            & \leq C \frac{ e^{\alpha s}}{\alpha} \W_{1,\alpha}(\mu, \mu').
        \end{split}
    \end{equation}
    To sum up, the constant $C$ in the previous formula depends on: $X^0$, $T$, $\|v\|_\infty$, $\r$, and~$\Kcp$.
    
    \substep{2}{6} (Concluding estimate of $|\tilde Y(t) - \tilde Y'(t) |$) Substituting~\eqref{eq:estimate of tildeX} in~\eqref{eq:first estimate of tildeY} we obtain that 
    \[
        \begin{split}
            | \tilde Y(s) - \tilde Y'(s) | & \leq C \int_0^s  \Big(  \frac{e^{\alpha r}}{\alpha}   \W_{1,\alpha}(\mu, \mu') +   |\tilde Y(r) - \tilde Y'(r)| \Big) \, \d r \\
             & \leq C \Big( \frac{e^{\alpha s} - 1}{\alpha^2}  \W_{1,\alpha}(\mu, \mu') +  \int_0^s     |\tilde Y(r) - \tilde Y'(r)|  \, \d r \Big) \\
             & \leq C \Big( \frac{e^{\alpha s}}{\alpha^2}  \W_{1,\alpha}(\mu, \mu') +  \int_0^s     |\tilde Y(r) - \tilde Y'(r)|  \, \d r \Big) \, .
        \end{split}
    \]
    Multiplying both sides by $e^{-\alpha s}$ and using that $e^{- \alpha s} \leq e^{-\alpha r}$ we get that a.s.\ for $s \in [0,t]$
    \[
        \begin{split}
             e^{- \alpha s} | \tilde Y(s) - \tilde Y'(s) |    & \leq C \Big( \frac{1}{\alpha^2}  \W_{1,\alpha}(\mu, \mu') +  \int_0^s  e^{- \alpha r}   |\tilde Y(r) - \tilde Y'(r)|  \, \d r \Big) \\
            & \leq C \Big( \frac{1}{\alpha^2}  \W_{1,\alpha}(\mu, \mu') +  \int_0^{t} \sup_{0 \leq r \leq s} e^{- \alpha r}   |\tilde Y(r) - \tilde Y'(r)|   \, \d s \Big) \, .
        \end{split}
    \]
    Taking the supremum for $s \in [0,t]$ and the expectation, we deduce that 
    \[
        \begin{split}
            & \E\Big(\sup_{0 \leq s \leq t} e^{- \alpha s} | \tilde Y(s) - \tilde Y'(s) |   \Big) \\
            & \quad \leq C \Big( \frac{1}{\alpha^2}  \W_{1,\alpha}(\mu, \mu') +  \int_0^{t} \E\Big( \sup_{0 \leq r \leq s}  e^{- \alpha r}   |\tilde Y(r) - \tilde Y'(r)|  \Big) \, \d s \Big)
        \end{split}
            \]
    and thus, by Gr\"onwall's inequality, 
    \[
        \E\Big(\sup_{0 \leq s \leq t} e^{- \alpha s} | \tilde Y(s) - \tilde Y'(s) |  \Big) \leq \frac{C}{\alpha^2} \W_{1,\alpha}(\mu, \mu') e^{C t}
    \]
    which for $t = T$ yields 
    \begin{equation} \label{eq:final estimate of tildeY}
        \E\big(\mynorm \tilde Y - \tilde Y' \mynorm_\alpha \big) \leq \frac{C}{\alpha^2} \W_{1,\alpha}(\mu, \mu') e^{CT} \leq \frac{C}{\alpha^2} \W_{1,\alpha}(\mu, \mu') \, .
    \end{equation}
    Keeping track of the constant $C$, it depends on: $X^0$, $T$, $\|v\|_\infty$, $\r$, $\Kcp$, $\Kpg$, and $\Kpc$.
    
    \substep{2}{7} (Choice of $\alpha$ and end of proof of contraction property) We choose $\alpha > 0$ in such a way that 
    \begin{equation} \label{eq:choice of alpha}
        C_\alpha :=   \frac{C}{\alpha^2}  < 1 \, ,  
    \end{equation}
    where $C$ is the constant obtained in~\eqref{eq:final estimate of tildeY}. In this way, by~\eqref{eq:first estimate for contraction} and~\eqref{eq:final estimate of tildeY} we conclude that 
    \[
    \W_{1, \alpha}(\L(\mu),\L(\mu'))\leq  C_\alpha \W_{1, \alpha}(\mu,\mu') \, ,
    \]
    \ie, $\L \colon \Pcal_1(C^0) \to \Pcal_1(C^0)$ is a contraction with respect to the equivalent Wasserstein distance $\W_{1,\alpha}$. As such, it admits a unique fixed point $\mubp \in \Pcal_1(C^0)$.

    \vspace{1em}

    \step{3} Given the fixed point $\mubp \in \Pcal_1(C^0)$ of $\L$, we define $\bar X = (\bar X_1, \dots, \bar X_N)$ as the solution to~\eqref{eq:decoupled SDE ODE - I} corresponding to $\mubp$, and then we let $\bar Y$ be the solution to~\eqref{eq:decoupled SDE ODE - II} corresponding to $\bar X$. Since $\mubp$ is a fixed point, we have that $\L(\mubp) = \mubp$, \ie, $\Law(\bar Y) = \mubp$. Hence we found the unique strong solution to the coupled system. This concludes the proof.
\end{proof}

\begin{remark} \label{rmk:bar X is bounded}
    By~\eqref{eq:boundedness of tildeX}, it follows that $\max_n \|\bar X_n\|_\infty$ is bounded by a constant depending on the initial datum $X^0$, the final time $T$, $\|v^N\|_\infty$, $\r$, and $\Kcp$. 

    By~\eqref{eq:bound on tildeY}, it follows that $\E(\|\bar Y\|_\infty) \leq C(1+\E(|\bar Y^0|))$, where the constant $C$ depends on $\Kpg$, $\Kpc$, $\|Z\|_{\infty}$, $\max_n \|\bar X_n\|_{\infty}$, $T$, and $W$. 
\end{remark}

\section{Propagation of chaos}
 
\begin{proposition} \label{prop:identically distributed} Assume the following:
    \begin{itemize}
        \item Let $(W_1(t))_{t \in [0,T]}$ and $(W_2(t))_{t \in [0,T]}$ be two $\R^2$-valued Brownian motions. 
        \item Let $X^0 = (X_1^0,\dots,X_N^0) \in \R^{2\x N}$;
        \item Let $Y_1^0, Y_2^0$ be identically distributed $\R^2$-valued random variables with $\E(|Y_m^0|) < +\infty$;
        \item Let  $Z^0 = (Z_1^0,\dots,Z_L^0) \in \R^{2\x L}$;
        \item Let $u \in L^\infty([0,T];\U)$. 
    \end{itemize}
    For $m = 1, 2$, let $\bar X_m = (\bar X_{m,1},\dots,\bar X_{m,N})$, $(\bar Y_m(t))_{t \in [0,T]}$, $Z = (Z_1, \dots, Z_L)$ be the unique strong solution to\footnote{This corresponds to the averaged ODE/SDE/ODE system~\eqref{eq:ODE Y average} with initial data $X^0$, $Y_m^0$, $Z^0$, with Brownian motion~$W_m$, and with control $u$ provided by Proposition~\ref{prop:solution to averaged ODE/SDE}.}
    \begin{equation} \label{eq:m m ODE Y average}
        \left\{
            \begin{aligned}
                & \frac{\d \bar X_{m,n}}{\d t}(t) = v^N_n(\bar X_m(t))\Big(  \r(\bar X_{m,n}(t)) + \Kcp*\mubp_m(t)(\bar X_{m,n}(t))    \Big) \, ,      \\
                & \d \bar Y_m(t)  = \Big(\frac{1}{L} \sum_{\ell=1}^L \Kpg(\bar Y_m(t) - Z_\ell(t))  - \frac{1}{N} \sum_{n=1}^N \Kpc( \bar Y_m(t) - \bar X_{m,n}(t)) \Big) \d t   + \sqrt{2\kappa} \, \d W_m(t) \, ,   \\
                & \frac{\d Z_\ell}{\d t}(t)  =  \frac{1}{L}\sum_{\ell'=1}^L \Kgg(Z_\ell(t) - Z_{\ell'}(t))  + u_\ell(t)   \, ,  \\ 
                & \bar X_{m,n}(0) =  X_n^0 \, , \  \ Z_\ell(0) =  Z_\ell^0 \, , \quad n = 1,\dots,N \, , \ \ell = 1,\dots, L \, , \\
                & \bar Y_m(0)  = Y_m^0 \  \text{a.s.,} \quad \mubp_m  = \Law(\bar Y_m) \, .
            \end{aligned}
        \right.
    \end{equation}
    Then the stochastic processes $(\bar Y_1(t))_{t \in [0,T]}$ and $(\bar Y_2(t))_{t \in [0,T]}$ are identically distributed and $\bar X_{1}(t) = \bar X_{2}(t)$ for $t \in [0,T]$. 
\end{proposition}
\begin{proof}
    We fix $Z = (Z_1,\dots,Z_L)$ as the solution to 
    \[
        \left\{
            \begin{aligned}
        \frac{\d Z_\ell}{\d t}(t) & =   \frac{1}{L}\sum_{\ell'=1}^L \Kgg(Z_\ell(t) - Z_{\ell'}(t))  + u_\ell(t)    \, ,  \\ 
        Z_\ell(0) & =  Z_\ell^0 \, , \quad \ell = 1,\dots, L \, ,
    \end{aligned}
    \right.    
    \]
    which is independent of $m$ since it is decoupled from the first two sets of equations.

    Let $m \in \{1,2\}$.  We resort to some tools already considered in {\itshape Step~2} in the proof of Proposition~\ref{prop:solution to averaged ODE/SDE}. As in that proof, we set $C^0 := C^0([0,T];\R^2)$.

    \vspace{1em}

    \step{1} (Exploiting the decoupled system) Given $\mu \in \Pcal_1(C^0)$ we let $\tilde X = (\tilde X_1,\dots,\tilde X_N)$ and $(\tilde Y_m(t))_{t \in [0,T]}$ be the unique solution to the decoupled system
    \begin{equation} \label{eq:decoupled SDE ODE m I}
        \left\{
            \begin{aligned}
                \frac{\d \tilde X_n}{\d t}(t) &= v^N_n(\tilde X(t))\Big(  \r(\tilde X_n(t)) + \Kcp*\mu(t)(\tilde X_n(t))    \Big) \, ,      \\
                 \tilde X_n(0) &=  X_n^0 \, , \quad n = 1,\dots,N \, ,\\
            \end{aligned}
        \right.
    \end{equation}
   \begin{equation} \label{eq:decoupled SDE ODE m II}
       \left\{
           \begin{aligned}
                   \d \tilde Y_m(t) & = \Big(\frac{1}{L} \sum_{\ell=1}^L \Kpg(\tilde Y_m(t) - Z_\ell(t))  - \frac{1}{N} \sum_{n=1}^N \Kpc(\tilde Y_m(t) - \tilde X_n(t)) \Big) \d t  + \sqrt{2\kappa} \, \d W_m(t) \, ,  \\
                \tilde Y_m(0)  & =  Y_m^0 \ \text{a.s.,}
           \end{aligned}
       \right.
   \end{equation}
   obtained as explained in {\itshape Step 1} in the proof of Proposition~\ref{prop:solution to averaged ODE/SDE}. We claim that 
   \begin{equation} \label{eq:Law tilde Ym}
     \Law(\tilde Y_1) = \Law(\tilde Y_2) \, .
   \end{equation}
   Using the short-hand notation introduced in~\eqref{eq:def of btildeX}, we have a.s.\ for $t \in [0,T]$ 
   \[
        \tilde Y_m(t) = Y_m^0 + \int_0^t b_{\tilde X}(s,\tilde Y_m(s)) \, \d s + \sqrt{2 \kappa} \, W_m(t) \, .
   \]
   \substep{1}{1} (Proof of claim~\eqref{eq:Law tilde Ym} for Picard iterations $\tilde Y_m^j$) We consider the Picard iterations (used in the proof of Proposition~\ref{prop:SDE simple}) constructed as follows for $m=1,2$: for $\omega \in \Omega$ 
   \begin{align}
        \tilde Y_m^0(t,\omega) & := Y_m^0(\omega) \, , \  \text{for } t \in [0,T]  \, , \label{eq:Picard base}\\
        \tilde Y_m^j(t,\omega) & := Y_m^0(\omega) + \int_0^t b_{\tilde X}(s,\tilde Y^{j-1}_m(s,\omega)) \, \d s + \sqrt{2 \kappa} \, W_m(t,\omega)\, , \ \text{for } t \in [0,T]\, , \ j \geq 1 \, .  \label{eq:Picard iteration}
    \end{align}
    We observe that $\Law(\tilde Y_1^0) = \Law(\tilde Y_2^0)$, as by~\eqref{eq:Picard base} they coincide with the common law of the identically distributed random variables given by the initial data $Y_1^0$, $Y_2^0$. This is the base step of an induction argument. Let $j \geq 1$ and assume $\Law(\tilde Y_1^{j-1}) = \Law(\tilde Y_2^{j-1})$. Let $\Psi \colon \R^2 \x C^0 \x C^0 \to C^0$ be the continuous map defined by
    \[
    \Psi(\xi,\varphi,w)(t) := \xi + \int_0^t b_{\tilde X}(s,\varphi(s)) \, \d s + \sqrt{2 \kappa} \, w(t) \,. 
    \]
    With this notation, \eqref{eq:Picard iteration} reads $Y_m^j(\cdot,\omega) = \Psi(Y_m^0(\omega), \tilde Y^{j-1}_m(\cdot, \omega), W_m(\cdot,\omega))$ for $\omega \in \Omega$ such that $W_m(\cdot,\omega)$ is a continuous path (this occurs a.s.). Then we have that 
    \[
        \begin{split}
            \Law(\tilde Y_m^{j}) & = (\tilde Y_m^{j})_\# \P = \Psi_\#(Y_m^0, \tilde Y^{j-1}_m, W_m)_\#\P 
        \end{split}
    \]
    Since $Y^0_1$, $Y^0_2$ are identically distributed, by the induction assumption, and since $\Law(W_1) = \Law(W_2)$ (it is the Wiener measure), we have that $(Y_1^0, \tilde Y^{j-1}_1, W_1)_\#\P = (Y_2^0, \tilde Y^{j-1}_2, W_2)_\#\P$. Thus, repeating backward the same computations for $\tilde Y_2^{j}$, we conclude that $\Law(\tilde Y_1^{j}) = \Law(\tilde Y_2^{j})$.
    
     \substep{1}{2} (Convergence of Picard iterations to $\tilde Y_m$) By Remark~\ref{rmk:Picard itearations} we have that $\E(\| \tilde Y^j_m  - \tilde Y_m\|_\infty) \to 0$. (Note that $b_{\tilde X}$ is globally Lipschitz continuous, as proven in~\eqref{eq:btildeX is Lipschitz}.)

    \substep{1}{3} (Proof of claim~\eqref{eq:Law tilde Ym}) The convergence $\E(\|\tilde Y_m^j - \tilde Y_m\|_\infty) \to 0$ implies that $\tilde Y_m^j \to \tilde Y_m$ in law, hence $\Law(\tilde Y_1) = \Law(\tilde Y_2)$, that is our claim~\eqref{eq:Law tilde Ym}.

    \vspace{1em}

    \step{2} (Exploiting the fixed point) For $m=1,2$ we consider the functionals $\L_m = \L_{Y^0_m, W_m}\colon \Pcal_1(C^0) \to \Pcal_1(C^0)$ defined as in {\itshape Step 2} in the proof of Proposition~\ref{prop:solution to averaged ODE/SDE} (we stress here the dependence on $m$ to keep track of the dependence on the initial datum $Y^0_m$ and the Brownian motion $W_m$). Given $\mu \in \Pcal_1(C^0)$, we let $\tilde X = (\tilde X_1,\dots,\tilde X_N)$ and $(\tilde Y_m(t))_{t \in [0,T]}$ be the unique solution to the decoupled system~\eqref{eq:decoupled SDE ODE m I}--\eqref{eq:decoupled SDE ODE m II}. Then we set $\L_m(\mu) := \Law(\tilde Y_m)$. By the discussion in {\itshape Step~1} we have that $\L_1(\mu) = \L_2(\mu)$.

        Let us now fix an initial guess for the law $\mu$, \eg, $\mu = \delta_0 \in \Pcal_1(C^0)$ (it is enough that it satisfies $\W_1(\mu,\mubp_m) < +\infty$). We apply iteratively $\L_m^0(\mu) = \mu$, $\L_m^j(\mu) = \L_m(\L^{j-1}_m(\mu))$. Since $\L_m$ is a contraction with respect to the modified $1$-Wasserstein distance $\W_{1,\alpha}$, $\L_m^j(\mu) \to \mubp_m$ as $j \to +\infty$, where $\mubp_m$ is the unique fixed point $\mubp_m = \L_m(\mubp_m)$. Since $\L_1(\mu) = \L_2(\mu)$, we conclude that $\mubp_1 = \mubp_2$, \ie, the law given by the solution $\bar Y_m$ to~\eqref{eq:m m ODE Y average} does not depend on~$m$. In conclusion, $\bar Y_1, \bar Y_2$ are  identically distributed. We let $\mubp$ denote their common law.

    The solution $\bar X_m = (\bar X_{m,1},\dots,\bar X_{m,N})$ is then obtained as the solution to~\eqref{eq:decoupled SDE ODE m I} corresponding to $\mubp$. Thus it does not depend on $m$, yielding $\bar X_1 = \bar X_2$.
\end{proof}

\begin{proposition} \label{prop:propagation of chaos} Assume the following: 
    \begin{itemize}
        \item Let $(W_m(t))_{t \in [0,T]}$, $m = 1, \dots,M$ be $M$ independent $\R^2$-valued Brownian motions;
        \item Let $X^0 = (X_1^0,\dots,X_N^0) \in \R^{2\x N}$;
        \item Let $Y_1^0, \dots, Y_M^0$ be i.i.d.\ $\R^2$-valued random variables with $\E(|Y_m^0|) < +\infty$ and independent from the Brownian motions $(W_m(t))_{t \in [0,T]}$;
        \item Let $Z^0 = (Z_1^0,\dots,Z_L^0) \in \R^{2 \x L}$;
        \item Let $u \in L^\infty([0,T];\U)$. 
    \end{itemize}
    For every $m = 1, \dots, M$, let $\bar X = (\bar X_1,\dots,\bar X_N)$, $(\bar Y_m(t))_{t \in [0,T]}$, $Z = (Z_1, \dots, Z_L)$ be the unique strong solution to\footnote{This corresponds to the averaged ODE/SDE/ODE system~\eqref{eq:ODE Y average} with initial data $X^0$, $Y_m^0$, $Z^0$, with Brownian motion~$W_m$, and with control $u$. The solution is provided by Proposition~\ref{prop:solution to averaged ODE/SDE}. Note that we applied Proposition~\ref{prop:identically distributed} to deduce that $(\bar Y_1(t))_{t \in [0,T]}, \dots, (\bar Y_M(t))_{t \in [0,T]}$ are identically distributed with common law $\mubp$ and the curve $\bar X$ is independent of $m$.}
    \begin{equation} \label{eq:m ODE Y average}
        \left\{
            \begin{aligned}
                & \frac{\d \bar X_n}{\d t}(t) = v^N_n(\bar X(t))\Big(  \r(\bar X_n(t)) + \Kcp*\mubp(t)(\bar X_n(t))    \Big) \, ,      \\
                & \d \bar Y_m(t) = \Big(\frac{1}{L} \sum_{\ell=1}^L \Kpg(\bar Y_m(t) - Z_\ell(t))  - \frac{1}{N} \sum_{n=1}^N \Kpc( \bar Y_m(t) - \bar X_n(t)) \Big) \d t  + \sqrt{2\kappa} \, \d W_m(t)  \\
                & \frac{\d Z_\ell}{\d t}(t)  =  \frac{1}{L}\sum_{\ell'=1}^L \Kgg(Z_\ell(t) - Z_{\ell'}(t))  + u_\ell(t)    \, ,  \\ 
                & \bar X_n(0) =  X_n^0 \, , \  \ Z_\ell(0) =  Z_\ell^0 \, , \quad n = 1,\dots,N \, , \ \ell = 1,\dots, L \, , \\
                & \bar Y_m(0)  = Y_m^0 \ \text{a.s.,} \quad  \mubp  = \Law(\bar Y_m) \, .
            \end{aligned}
        \right.
    \end{equation}
    Then the stochastic processes $(\bar Y_1(t))_{t \in [0,T]}, \dots, (\bar Y_M(t))_{t \in [0,T]}$ are independent.
\end{proposition}

\begin{proof}
    The leading idea of the proof is to to write $\bar Y_m$ in terms of the initial datum $Y_m^0$ and the Brownian motion $W_m$.

    We consider the solution operator $\S \colon \R^2 \x C^0 \to C^0$ defined by $\S(\xi,w) := \varphi$, where $\varphi$ is the unique solution to the integral equation
    \[
    \varphi(t) = \xi + \int_0^t b_{\bar X} (s, \varphi(s)) \, \d s + \sqrt{2 \kappa} \, w(t) \, , \quad t \in [0,T] \, .    
    \]
    The fact that there exists a unique solution to the previous problem follows from the fact that the operator $\Psi \colon \R^2 \x C^0 \x C^0 \to C^0$ defined by
    \[
        \Psi(\xi,\varphi,w)(t) := \xi + \int_0^t b_{\bar X} (s, \varphi(s)) \, \d s + \sqrt{2 \kappa} \, w(t) \, , \quad  \text{for } t \in [0,T] \, .
    \]
    is such that $\Psi(\xi, \cdot, w) \colon C^0 \to C^0$ is a contraction with respect to the auxiliary norm $\mynorm \varphi \mynorm_\alpha := \sup_{t \in [0,T]} (e^{-\alpha t} |\varphi(t)|)$ for a suitable $\alpha > 0$. Indeed, by the Lipschitz continuity of $b_{\bar X}$,
    \[
        \begin{split}
            & e^{-\alpha t} | \Psi(\xi,\varphi_1,w)(t) - \Psi(\xi,\varphi_2,w)(t)|  \leq e^{-\alpha t}\int_0^t |b_{\bar X} (s, \varphi_1(s)) - b_{\bar X} (s, \varphi_2(s))| \, \d s \\
            & \quad \leq C e^{-\alpha t} \int_0^t e^{\alpha s}  e^{-\alpha s}|\varphi_1(s) - \varphi_2(s)| \, \d s  \leq C e^{-\alpha t} \mynorm \varphi_1  - \varphi_2 \mynorm_\alpha \frac{e^{\alpha t} - 1}{\alpha} \leq \frac{C}{\alpha} \mynorm \varphi_1  - \varphi_2 \mynorm_\alpha \, ,
        \end{split}
    \]
    hence, choosing $\alpha > 0$ such that $C_\alpha = \frac{C}{\alpha} < 1$, 
    \[
    \mynorm     \Psi(\xi,\varphi_1,w) - \Psi(\xi,\varphi_2,w) \mynorm_\alpha \leq C_\alpha  \mynorm \varphi_1  - \varphi_2 \mynorm_\alpha \, ,
    \]
    and thus it has a unique fixed point.

    We now observe that the solution operator $\S \colon \R^2 \x C^0 \to C^0$ is continuous. Indeed, it is Lipschitz with respect to both variables. Letting $\varphi_1 = \S(\xi_1,w)$ and $\varphi_2 = \S(\xi_2, w)$, by the Lipschitz continuity of $b_{\bar X}$, we have that 
    \[
        \begin{split}
            |\varphi_1(t) - \varphi_2(t)| & \leq |\xi_1 - \xi_2| + \int_0^t |b_{\bar X}(s, \varphi_1(s)) - b_{\bar X}(s, \varphi_2(s))| \, \d s \\
            & \leq |\xi_1 - \xi_2| + C \int_0^t | \varphi_1(s) - \varphi_2(s)| \, \d s \, ,
        \end{split}
    \] 
    thus, by Gr\"onwall's inequality, 
    \[
        |\varphi_1(t) - \varphi_2(t)| \leq |\xi_1 - \xi_2| e^{Ct} \implies \|\S(\xi_1,w) - \S(\xi_2,w) \|_\infty \leq |\xi_1 - \xi_2| e^{CT} \, .
    \]
    Analogously, letting $\varphi_1 = \S(\xi,w_1)$ and $\varphi_2 = \S(\xi, w_2)$, by the Lipschitz continuity of $b_{\bar X}$, we have that 
    \[
        \begin{split}
            |\varphi_1(t) - \varphi_2(t)| & \leq  \int_0^t |b_{\bar X}(s, \varphi_1(s)) - b_{\bar X}(s, \varphi_2(s))| \, \d s + |w_1(t) - w_2(t)| \\
            & \leq  C \int_0^t | \varphi_1(s) - \varphi_2(s)| \, \d s  + \|w_1 - w_2\|_\infty\, ,
        \end{split}
    \] 
    thus, by Gr\"onwall's inequality, 
    \[
        |\varphi_1(t) - \varphi_2(t)| \leq \|w_1 - w_2\|_\infty e^{Ct} \implies \| \S(\xi,w_1) -  \S(\xi,w_2) \|_\infty \leq \|w_1 - w_2\|_\infty e^{CT} \, .
    \]

    We are now in a position to write the stochastic processes $(\bar Y_m(t))_{t \in [0,T]}$ as $Y_m(\cdot, \omega) = \S(Y_m^0(\omega), W_m(\cdot, \omega))$ for a.e.\ $\omega \in \Omega$. Note that $Y_1^0, \dots, Y_M^0 \colon \Omega \to \R^2$ and $W_1, \dots, W_M \colon \Omega \to C^0$ are independent random variables. It follows that $(\bar Y_1(t))_{t \in [0,T]}, \dots, (\bar Y_M(t))_{t \in [0,T]}$ are independent stochastic processes. This concludes the proof.

    \end{proof}
    
    \section{Mean-field limit for a large number of pirate ships} \label{sec:M to infty}

    In this section we study the limit of the problem as $M \to +\infty$. For this reason we will stress the dependence of initial data and solutions on $M$. Still, we do not stress dependence on $N$, not to overburden the notation.

    \subsection{Mean-field ODE/SDE/ODE limit model as $M\to +\infty$} 

    In the following theorem we shall describe convergence of solutions in terms of empirical measures. Given stochastic processes $(S_1(t))_{t\in [0,T]}, \dots, (S_M(t))_{t\in [0,T]}$ a.s.\ with continuous paths, we associate the empirical measure\footnote{The measurability of these random variables is proven with an argument analogous to the one in Footnote~\ref{footnote:random measures}, keeping in mind the separability of $C^0([0,T];\R^2)$.} $\nu_M \colon \Omega \to \Pcal(C^0([0,T];\R^2))$ defined for a.e.\ $\omega \in \Omega$ by 
        \[
            \nu_M(\cdot, \omega)  := \frac{1}{M} \sum_{m=1}^M \delta_{S_m(\cdot, \omega)} \, .
        \]
        (The first placeholder is kept free for the time variable.) If $\max_m\E(\|S_m\|_\infty) < +\infty$  then a.s.\ $\nu_M \in \Pcal_1(C^0([0,T];\R^2))$. Indeed, 
        \begin{equation} \label{eq:2212061303}
            \begin{split}
                \E\Big( \int_{C^0([0,T];\R^2)} \|\varphi\|_\infty \, \d  \nu_M(\cdot, \cdot)(\varphi) \Big) & = \frac{1}{M} \sum_{m=1}^M \E \Big( \int_{C^0([0,T];\R^2)}  \|\varphi\|_\infty \, \d  \delta_{S_m} \Big) \\
                & = \frac{1}{M} \sum_{m=1}^M   \E\big(\|S_m\|_\infty\big)  < + \infty \, .
            \end{split}
        \end{equation}
        We set $\nu_M(t,\omega) := (\ev_t)_\# \nu_M(\cdot,\omega)$ for all $\omega \in \Omega$ and $t \in [0,T]$. With a slight abuse of notation, we let $\nu_M(t)$ denote the random measure $\nu_M(t) \colon \Omega \to \Pcal(\R^2)$. 

    \begin{theorem} \label{thm:limit to averaged system}
        Assume the following:
        \begin{itemize}
            \item Let $(W_m(t))_{t \in [0,T]}$, $m \geq 1$, be a sequence of independent $\R^2$-valued Brownian motions;
            \item Let $X^0 = (X_1^0,\dots,X_N^0) \in \R^{2\x N}$;
            \item Let $Y^0 =(Y_1^0, \dots, Y_M^0)$, where $Y_1^0, \dots, Y_M^0$ are i.i.d.\ $\R^2$-valued random variables with $\E(|Y^0_m|) < +\infty$ and independent from the Brownian motions $(W_m(t))_{t \in [0,T]}$;
            \item Let $Z^0 = (Z_1^0,\dots,Z_L^0) \in \R^{2 \x L}$;
            \item Let $(W(t))_{t \in [0,T]}$ be a Brownian motion;
            \item Let $\bar Y^0$ be a $\R^2$-valued random variable identically distributed to $Y_1^0, \dots, Y_M^0$.
        \end{itemize}
        Let $u^M, u \in L^\infty([0,T];\U)$ be such that $u^M \wstar u$ weakly* in $L^\infty([0,T];\U)$.\footnote{In fact, by the boundedness of $\U$, this is equivalent to requiring that $u^M \weak u$ weakly in $L^1([0,T];\U)$.}  Let $(X^M(t))_{t \in [0,T]} = (X^M_1(t),\dots,X^M_N(t))_{t \in [0,T]}$, $(Y^M(t))_{t \in [0,T]} = (Y^M_1(t), \dots, Y^M_M(t))_{t \in [0,T]}$, and $Z^M = (Z^M_1,\dots,Z^M_L)$ be the unique strong solution to\footnote{This corresponds to the original ODE/SDE/ODE system~\eqref{eq:full model} with initial data $X^0$, $Y^0$, $Z^0$ and with control $u^M$. The solution is provided by Proposition~\ref{prop:solution to ODE/SDE}. We stressed the dependence on $M$ since we are interested in the limit as $M \to +\infty$.}
        \begin{equation} \label{eq:ODE random M}
            \left\{
                \begin{aligned}
                    & \d X^M_n(t)  = v^N_n(X^M(t))\Big(  \r(X^M_n(t)) + \frac{1}{M} \sum_{m=1}^M \Kcp(X^M_n(t) - Y^M_m(t))    \Big)\d t \, ,      \\
                    & \d Y^M_m(t) = \Big(\frac{1}{L} \sum_{\ell=1}^L \Kpg(Y^M_m(t) - Z^M_\ell(t))  - \frac{1}{N} \sum_{n=1}^N \Kpc(Y^M_m(t) - X^M_n(t)) \Big) \d t   + \sqrt{2\kappa} \, \d W_m(t) \, ,  \\
                    & \frac{\d Z^M_\ell}{\d t}(t) =  \frac{1}{L}\sum_{\ell'=1}^L \Kgg(Z^M_\ell(t) - Z^M_{\ell'}(t) )  + u^M_\ell(t)    \, ,  \\ 
                    & X^M_n(0)  = X_n^0 \ \text{a.s.,} \ Y^M_m(0) = Y_m^0 \ \text{a.s.,}  \ Z^M_\ell(0) =  Z_\ell^0 \, , \\
                    &  n = 1,\dots,N \, , \ m = 1,\dots,M \, , \ \ell = 1,\dots, L \, .
                \end{aligned}
            \right.
        \end{equation}
        Let $\nup_M$ be the empirical measures associated to $(Y^M_1(t))_{t \in [0,T]}, \dots, (Y^M_M(t))_{t \in [0,T]}$. Then there exist $\bar X = (\bar X_1,\dots, \bar X_N)$, $(\bar Y(t))_{t \in [0,T]}$, and $Z = (Z_1,\dots,Z_L)$ such that  
        \begin{equation} \label{eq:needed convergence}
            \E\Big( \max_n \|X^M_n - \bar X_n\|_\infty \Big) +   \int_0^T \E \big( \W_1(\nup_M(t),\mubp(t)) \big) \, \d t  + \|Z^M - Z\|_\infty \to 0 \ \text{as } M \to +\infty \, .
        \end{equation} 
        Moreover, $\bar X = (\bar X_1,\dots, \bar X_N)$, $(\bar Y(t))_{t \in [0,T]}$, and $Z = (Z_1,\dots,Z_L)$ are the unique strong solution to~\eqref{eq:ODE Y average}.\footnote{Corresponding to the initial data $X^0$, $\bar Y^0$, $Z^0$, with Brownian motion $W$, and control $u$. We recall that the solution is provided by Proposition~\ref{prop:solution to averaged ODE/SDE}.}
    \end{theorem}

    \begin{proof} 
        To prove the result, we need to exploit an intermediate problem. For $m=1,\dots, M$, let $\bar X = (\bar X_1,\dots, \bar X_N)$, $(\bar Y^M_m(t))_{t \in [0,T]}$, and $Z = (Z_1,\dots,Z_L)$ be the unique strong solution~to\footnote{This corresponds to the averaged ODE/SDE/ODE system~\eqref{eq:ODE Y average} with initial data $X^0$, $Y^0_m$, $Z^0$, with Brownian motion $W_m$, and control $u$. The solution is provided by Proposition~\ref{prop:solution to averaged ODE/SDE}. Note that we applied Proposition~\ref{prop:identically distributed} to deduce that $(\bar Y^M_1(t))_{t \in [0,T]}, \dots, (\bar Y^M_M(t))_{t \in [0,T]}$ are identically distributed with common law $\mubp$ and the curve $\bar X$ does not depend on $m$ and $M$.}
        \begin{equation} \label{eq:m ODE Y average M}
            \left\{
                \begin{aligned}
                    & \frac{\d \bar X_{n}}{\d t}(t) = v^N_n(\bar X(t))\Big(  \r(\bar X_{n}(t)) + \Kcp*\mubp(t)(\bar X_{n}(t))    \Big)  \, ,      \\
                    & \d \bar Y^M_m(t) = \Big( \frac{1}{L} \sum_{\ell=1}^L \Kpg( \bar Y^M_m(t) - Z_\ell(t))  - \frac{1}{N} \sum_{n=1}^N \Kpc(\bar Y^M_m(t) - \bar X_{n}(t) ) \Big) \d t + \sqrt{2\kappa} \, \d W_m(t) \, ,  \\
                    & \frac{\d Z_\ell}{\d t}(t)  =   \frac{1}{L}\sum_{\ell'=1}^L \Kgg(Z_\ell(t) - Z_{\ell'}(t) )  + u_\ell(t)    \, ,  \\ 
                    & \bar X_{n}(0) =  X_n^0 \, , \  \ Z_\ell(0) =  Z_\ell^0 \, , \quad n = 1,\dots,N \, , \ \ell = 1,\dots, L \, , \\
                    & \bar Y^M_m(0)  = Y_m^0 \ \text{a.s.,} \quad \mubp  = \Law(\bar Y^M_m) \, .
                \end{aligned}
            \right.
        \end{equation}
        Our first task is to prove that  
        \begin{equation} \label{eq:intermediate convergence}
            \E\Big( \max_n \|X^M_n - \bar X_n\|_\infty\Big) + \E\Big(\max_m \|Y^M_m - \bar Y^M_m \|_\infty \Big) + \|Z^M - Z\|_\infty \to 0 \quad \text{as } M \to +\infty \, .
        \end{equation}  
        from which~\eqref{eq:needed convergence} will follow as shown in {\itshape Step 5} below.

        As in the previous proofs, let $C^0 := C^0([0,T];\R^2)$. Let us also consider the empirical measures\footnote{The random measure $\nubp_M \colon \Omega \to \Pcal(C^0)$ (empirical measure of $\bar Y^M_1, \dots, \bar Y^M_M$) must not be confused with $\nup_M \colon \Omega \to \Pcal(C^0)$ (empirical measure of $Y^M_1, \dots, Y^M_M$) or $\mubp \in \Pcal(C^0)$ (common law of the stochastic processes $(\bar Y^M_1(t))_{t \in [0,T]}, \dots, (\bar Y^M_M(t))_{t \in [0,T]}$ and $(\bar Y(t))_{t \in [0,T]}$).} $\nubp_M \colon \Omega \to \Pcal(C^0)$ associated to $(\bar Y^M_1(t))_{t \in [0,T]}, \dots, (\bar Y^M_M(t))_{t \in [0,T]}$. To be precise, we have that for a.e.\ $\omega \in \Omega$ 
        \[
            \nup_M(\cdot, \omega)  := \frac{1}{M} \sum_{m=1}^M \delta_{Y^M_m(\cdot, \omega)} \, , \quad \nubp_M(\cdot, \omega) := \frac{1}{M} \sum_{m=1}^M \delta_{\bar Y^M_m(\cdot,\omega)} \, .
        \]
        (The first placeholder is kept free for the time variable.) Notice that, in fact, a.s.\ $\nup_M \in \Pcal_1(C^0)$ and $\nubp_M \in \Pcal_1(C^0)$ by~\eqref{eq:2212061303} and since by Proposition~\ref{prop:solution to ODE/SDE} and Proposition~\ref{prop:solution to averaged ODE/SDE} we have that $\E(\max_m \|Y^M_m\|_\infty) < +\infty$ and $\E(\max_{m} \|\bar Y^M_m\|_\infty) < +\infty$, respectively.

        \step{1}  (Estimate of $|Y^M_m - \bar Y^M_m|$)   Using the fact that $Y^M_m$ and $\bar Y^M_m$ are strong solutions to~\eqref{eq:ODE random M} and~\eqref{eq:m ODE Y average M}, respectively, and by the Lipschitz continuity of $\Kpg$ and $\Kpc$ we have that a.s.\ for $0\leq s \leq t$ and $m=1,\dots,M$ 
        \[ 
            \begin{split}
                & |Y^M_m(s)-\bar Y^M_m(s)| \\
                &  \leq \int_0^s \Big( \frac{1}{L} \sum_{\ell=1}^L | \Kpg(Y^M_m(r) - Z^M_\ell(r) ) -  \Kpg(\bar Y_m(r) - Z_\ell(r)) | + \\
                & \hspace{2cm} + \frac{1}{N} \sum_{n=1}^N | \Kpc(Y^M_m(r) - X^M_n(r)) - \Kpc(\bar Y^M_m(r) - \bar X_n(r)) | \Big) \, \d r \\
                & \leq  \int_0^s  C \Big(  |Y^M_m(r) - \bar Y^M_m(r)| + \frac{1}{L} \sum_{\ell=1}^L  | Z^M_\ell(r) - Z_\ell(r) |+  \frac{1}{N} \sum_{n=1}^N |X^M_n(r) - \bar X_n(r)| \Big) \, \d r \\
                & \leq \int_0^t  C \Big( \max_n \sup_{0\leq r \leq s} |X^M_n(r) - \bar X_n(r)| + \max_{m'} \sup_{0\leq r \leq s}|Y^M_{m'}(r) - \bar Y^M_{m'}(r)| \Big) \, \d s + CT \|Z^M - Z\|_\infty \, ,
            \end{split}
        \]  
        the constant $C$ depending on $\Kpg$ and $\Kpc$. Taking the supremum in $s \in [0,t]$, the maximum in $m$ and then the expectation, we obtain that  for every $t \in [0,T]$ 
        \begin{equation} \label{eq:first bound on Y - barY}
            \begin{split}
                & \E \Big( \max_m \sup_{0\leq s \leq t}|Y^M_m(s) - \bar Y^M_m(s)| \Big) \\
                & \leq C \int_0^t     \E \Big( \max_m \sup_{0\leq r \leq s}|Y^M_m(r) - \bar Y^M_m(r)| + \max_n \sup_{0\leq r \leq s} |X^M_n(r) - \bar X_n(r)| \Big)  \, \d s \\
                & \quad  + CT \|Z^M - Z\|_\infty \, .
            \end{split}
        \end{equation} 
        
        \step{2} (Estimate of $|X^M_n - \bar X_n|$) 
        To estimate $|X^M_n(s) - \bar X_n(s)|$,  we rewrite 
        \begin{equation} \label{eq:Kcp in convolution form}
            \frac{1}{M} \sum_{m=1}^M \Kcp(X^M_n(t) - Y^M_m(t))  =  \int_{\R^2} \Kcp(X^M_n(t) - y) \, \d \nup_M(t)(y)  =  \Kcp*\nup_M(t)(X^M_n(t)) \, .
        \end{equation}
        Then, we exploit the properties of $v^N$, $\r$, and $\Kcp$ and~\eqref{eq:Kcp in convolution form} to get from~\eqref{eq:ODE random M} that a.s.\ for  $0\leq s \leq t$ and $n=1,\dots,N$
        \begin{equation} \label{eq:0311221610}
            \begin{split}
                & |X^M_n(s) - \bar X_n(s)| \leq \\
                & \leq \int_0^s \Big|  v^N_n(X^M(r))\Big(  \r(X^M_n(r)) + \frac{1}{M} \sum_{m=1}^M \Kcp(X^M_n(t) - Y^M_m(r))    \Big) + \\
                & \hspace{2cm} - v^N_n(\bar X(r))\Big(  \r(\bar X_n(r)) + \Kcp*\mubp(r)(\bar X_n(r))    \Big) \Big| \, \d r \\
                & \leq \int_0^s \|v^N\|_\infty \Big(  |\r(X^M_n(r)) -  \r(\bar X^M_n(r))| + |\Kcp*\nup_M(r)(X_n(r))  -  \Kcp*\mubp(r)(\bar X_n(r))| \Big) + \\
                & \hspace{2cm} + | v^N_n(X^M(r)) - v^N_n(\bar X(r))| \big|  \r(\bar X_n(r)) - \Kcp*\mubp(r)(\bar X_n(r))    \big| \, \d r  \, .
            \end{split}
        \end{equation}
        To estimate the term involving $|\Kcp*\nup_M(r)(X^M_n(r))  -  \Kcp*\mubp(r)(\bar X_n(r))|$ in~\eqref{eq:0311221610}, we exploit Kantorovich's duality and the Lipschitz continuity of $\Kcp$ to get that a.s.\ 
        \begin{equation} \label{eq:bound of KmuM - Kmu}  
            \begin{split}
                & |\Kcp*\nup_M(r)(X^M_n(r))  -  \Kcp*\mubp(r)(\bar X_n(r))| \\
                & = \Big| \int_{\R^2} \Kcp(X^M_n(r) - y) \, \d \nup_M(r)(y) -  \int_{\R^2} \Kcp(\bar X_n(r) - y) \, \d \mubp(r)(y) \Big| \\
                & \leq \Big| \int_{\R^2} \Kcp(X^M_n(r) - y) \, \d \Big( \nup_M(r) - \mubp(r) \Big) (y) \Big| \\
                & \quad + \Big| \int_{\R^2} \big(  \Kcp(X^M_n(r) - y) - \Kcp(\bar X_n(r) - y) \big) \, \d \mubp(t)(y) \Big| \\ 
                & \leq C \W_1(\nup_M(r), \mubp(r))  + C |X^M_n(r) - \bar X_n(r)| \\
                & \leq C \max_{n'}|X^M_{n'}(r) - \bar X_{n'}(r)| + C \W_1(\nup_M(r), \nubp_M(r)) +  C \W_1(\nubp_M(r), \mubp(r))  \, .
            \end{split}
        \end{equation}  
        We bound $\W_1(\nup_M(r), \nubp_M(r))$ using for a.e.\ $\omega \in \Omega$ as an admissible transport plan the diagonal transport $\gamma(\omega) = \frac{1}{N} \sum_{n=1}^N \delta_{(Y^M_n(r,\omega), \bar Y^M_n(r, \omega))}$ to obtain that a.s.\  
        \begin{equation} \label{eq:bound of W1} 
            \begin{split}
              \W_1(\nup_M(r), \nubp_M(r)) & \leq \int_{\R^2 \x \R^2} |y - y'| \, \d \gamma (y,y') = \frac{1}{M} \sum_{m=1}^M|Y^M_m(r) - \bar Y^M_m(r)| \\
              & \leq \max_m |Y^M_m(r) - \bar Y^M_m(r)| \,.
            \end{split}
        \end{equation}  

        To estimate the term involving $\Kcp*\mubp(r)(\bar X_n(r))$ in~\eqref{eq:0311221610}, we use the fact that $|\Kcp(z)| \leq |\Kcp(0)| + C|z|$ to get that  
        \begin{equation*} 
            \begin{split}
                & |\Kcp*\mubp(r)(\bar X_n(r))| \leq \int_{\R^2} | \Kcp(\bar X_n(r) - y) | \, \d \mubp(r)(y) \\
                & \leq \int_{\R^2} C (1 + |\bar X_n(r)| + |y|) \, \d \mubp(r)(y)   \leq C \Big( 1 + \max_{n'} \|\bar X_{n'} \|_\infty + \sup_{0 \leq r \leq T} \Big( \int_{\R^2} |y| \, \d \mubp(r)(y) \Big) \Big) \\
                & \leq C \Big( 1 + \max_{n'}\|\bar X_{n'}\|_\infty + \int_{\C^0} \| \varphi \|_\infty \, \d  \mubp(\varphi)  \Big)  \, ,
            \end{split}
        \end{equation*}  
        where in the last inequality we used that 
        \begin{equation*}  
            \begin{split}
                \int_{\R^2} |y| \, \d \mubp(r)(y) & = \int_{\R^2} |y| \, \d ((\ev_r)_\# \mubp)(y) = \int_{\C^0} |\ev_r(\varphi)| \, \d  \mubp(\varphi)   \leq \int_{\C^0} \| \varphi \|_\infty \, \d  \mubp(\varphi)  \, ,
            \end{split}
        \end{equation*} 
        which is finite, since $\mubp \in \Pcal_1(C^0)$ by Proposition~\ref{prop:solution to averaged ODE/SDE}. By Remark~\ref{rmk:bar X is bounded}   we recall that $\max_n \|\bar X_n \|_\infty$ is bounded by a constant depending on $\max_n \|X^0_n \|_\infty$, $T$, $\|v^N\|_\infty$, $\r$, and $\Kcp$. Hence  
        \begin{equation}\label{eq:bound of Kmu}
            |\Kcp*\mubp(r)(\bar X_n(r))| \leq C \, .
        \end{equation}

        Then we can proceed with the estimate in~\eqref{eq:0311221610}: By~\eqref{eq:bound of KmuM - Kmu}--\eqref{eq:bound of Kmu} and by exploiting also the Lipschitz continuity of $\r$ and $v^N$, we obtain that 
        \[
            \begin{split}
                & |X^M_n(s) - \bar X_n(s)|  \\
                &  \leq C  \int_0^s \max_{n'}|X^M_{n'}(r) - \bar X_{n'}(r)| \, \d r + C  \int_0^s \max_m |Y^M_m (r) - \bar Y^M_m (r)| \, \d r \\
                & \hspace{2cm} + C \int_0^s \W_1(\nubp_M(r), \mubp(r)) \, \d r \\
                & \leq C  \int_0^t \max_{n'}\sup_{0 \leq r \leq s} |X^M_{n'}(r) - \bar X_{n'}(r)| \, \d s + C  \int_0^t \max_{m}\sup_{0 \leq r \leq s} |Y^M_m(r) - \bar Y^M_m(r)| \, \d s \\
                & \hspace{2cm} + C \int_0^T \W_1(\nubp_M(s), \mubp(s)) \, \d s   \, .
            \end{split}
        \]  
         Taking the supremum in $s$, the maximum in $n$, and then the expectation, we  obtain that for every $t \in [0,T]$ 
        \begin{equation} \label{eq:first bound on X - barX} 
            \begin{split}
                & \E \Big( \max_n \sup_{0 \leq s \leq t} |X^M_n(s) - \bar X_n(s)| \Big) \leq  C  \int_0^t \E \Big( \max_n \sup_{0 \leq r \leq s} |X^M_n(r) - \bar X_n(r)| \Big)  \, \d s \\
                & \quad + C  \int_0^t \E \Big( \max_m \sup_{0 \leq r \leq s} |Y^M_m(r) - \bar Y^M_m(r)| \Big)  \, \d s  +  C \int_0^T \E \big( \W_1(\nubp_M(s), \mubp(s)) \big) \, \d s   \, .
            \end{split}
        \end{equation}  

        \step{3} (Gr\"onwall's inequality) Putting together \eqref{eq:first bound on Y - barY} and \eqref{eq:first bound on X - barX} we have that for every $t \in [0,T]$ 
        \[  
            \begin{split}
                & \E \Big( \max_n \sup_{0 \leq s \leq t} |X^M_n (s) - \bar X_n (s)|  + \max_m \sup_{0\leq s \leq t}|Y^M_m (s) - \bar Y^M_m (s)| \Big)  \\
                & \leq C \int_0^t    \E \Big( \max_n \sup_{0\leq r \leq s} |X^M_n (r) - \bar X_n (r)|+ \max_m \sup_{0\leq r \leq s}|Y^M_m (r) - \bar Y^M_m (r)| \Big)  \, \d s  \\
                & \quad + CT \| Z^M - Z \|_\infty +   C \int_0^T \E \big( \W_1(\nubp_M(s), \mubp(s)) \big) \, \d s \, .
            \end{split}
        \]  
        By Gr\"onwall's inequality, we deduce that for every $t \in [0,T]$ 
        \[ 
            \begin{split}
                & \E \Big( \max_n \sup_{0 \leq s \leq t} |X^M_n (s) - \bar X_n (s)| + \max_m \sup_{0\leq s \leq t}|Y^M_m (s) - \bar Y^M_m (s)| \Big) \\
                & \quad  \leq C e^{Ct} \Big( CT \| Z^M - Z \|_\infty + \int_0^T \E \big( \W_1(\nubp_M(s), \mubp(s)) \big) \, \d s  \Big) \, .
            \end{split}
        \]  
        In particular, 
        \[
          \begin{split}
            & \E \Big( \max_n \|X^M_n - \bar X^M_n \|_\infty + \max_m \|Y^M_m  - \bar Y^M_m \|_\infty \Big) \\
            &  \leq C \Big( \| Z^M - Z \|_\infty + \int_0^T \E \big( \W_1(\nubp_M(s), \mubp(s)) \big) \, \d s \Big) =: \alpha(M) \, ,
          \end{split}
        \]  
        where the constant depends additionally on $T$. 

        \step{4} (Convergence to zero of $\alpha(M)$)
        To conclude the proof, we show that $\alpha(M)  \to 0$ as $M\to +\infty$. 
        
        \substep{4}{1} Let us show that $\|Z^M - Z\|_\infty \to 0$ as $M \to +\infty$. We start by observing that by \eqref{eq:ODE random M} and \eqref{eq:m ODE Y average M}
        \[
            \begin{split}
                & | Z^M_\ell(t) - Z_\ell(t) | \\
                &  \leq \frac{1}{L} \sum_{\ell' =1}^L \int_0^t  |\Kgg(Z^M_{\ell}(s) - Z^M_{\ell'}(s) ) - \Kgg(Z_{\ell}(s) - Z_{\ell'}(s)) |  \,   \d s + \Big| \int_0^t ( u^M_{\ell}(s) - u_\ell(s)  )  \, \d s \Big| \\
                & \leq \int_0^t C |Z^M(s) - Z(s)| \, d s + \Big| \int_0^t ( u^M(s) - u(s)  )  \, \d s \Big| \, ,
            \end{split}
        \]
        where the constant $C$ depends on $\Kgg$. By Gr\"onwall's inequality, it follows that 
        \[
            |Z^M(t) - Z(t)| \leq R_M(t)  + \int_0^t R_M(s) C e^{C(t-s)} \, \d s \leq  R_M(t)  + C e^{CT} \int_0^T R_M(s)  \, \d s \, ,
        \]
        where $ R_M(t) = \Big| \int_0^t ( u^M(s) - u(s)  )  \, \d s \Big|$, hence
        \[
        \|Z^M - Z\|_\infty \leq C \|R_M\|_\infty \, .  
        \]
        Since $u^M \wstar u$ weakly* in $L^\infty([0,T];\U)$, we have that $R_M(t) \to 0$ for every $t \in [0,T]$. Moreover, by the boundedness of $\U$, $R_M(t)$ are equibounded and equi-Lipschitz. It follows that $\|R_M\|_\infty \to 0$, thus $\|Z^M - Z\|_\infty \to 0$.

        \substep{4}{2} Let us show that $\int_0^T \E \big( \W_1(\nubp_M(s), \mubp(s)) \big) \, \d s \to 0$ as $M \to +\infty$. 

        To show this, we apply the discussion in Subsection~\ref{subsec:empirical measure} about the approximation of a law (here played by $\mubp(s)$) with empirical measures on independent samples of the law (here played by $\nubp_M(s)$). Let us check that all the assumptions hold true. For every $s \in [0,T]$ we have that $\mubp(s) \in \Pcal_1(\R^d)$. This follows from the fact that, by Proposition~\ref{prop:identically distributed}, $\mubp = \Law(\bar Y^M_1) = \dots = \Law(\bar Y^M_M) = \Law(\bar Y)$, thus
        \begin{equation} \label{eq:mubps is in Pcal1}
            \begin{split}
                \int_{\R^2} |y| \, \d \mubp(s)(y) & = \int_{\R^2} |y| \, \d ((\ev_s)_\# (\bar Y)_\# \P) (y) = \int_{\Omega} |\bar Y(s,\omega)| \, \d \P(\omega) \\
                & = \E(|\bar Y(s)|) \leq \E(\|\bar Y\|_\infty) < + \infty \, ,
            \end{split}
        \end{equation}
        where the finiteness of $\E(\|\bar Y\|_\infty)$ follows from Proposition~\ref{prop:solution to averaged ODE/SDE}. Moreover, the random variables $\bar Y^M_1(s), \dots, \bar Y^M_M(s)$ are i.d.\ with law $(\bar Y^M_m(s,\cdot))_\# \P = (\ev_s)_\# (\bar Y^M_m)_\# \P = (\ev_s)_\# \mubp = \mubp(s)$. Finally, by Proposition~\ref{prop:propagation of chaos} we have that $(\bar Y^M_1(t))_{t \in [0,T]}, \dots, (\bar Y^M_M(t))_{t \in [0,T]}$ are independent stochastic processes, thus, in particular, $\bar Y^M_1(s), \dots, \bar Y^M_M(s)$ are independent random variables. By~\cite[Lemma~4.7.1]{PanZem} we conclude that 
        \[
            \E\big(\W_1(\nubp_M(s), \mubp(s))\big) \to 0 \quad \text{ for every $s \in [0,T]$}\, , 
        \]
        as $M \to +\infty$. Let us now show that $s \mapsto \E\big(\W_1(\nubp_M(s), \mubp(s))\big)$ is dominated.  Indeed, since $\bar Y_1(s), \dots, \bar Y_M(s)$ are identically distributed and by~\eqref{eq:mubps is in Pcal1}, for every $s \in [0,T]$ we have that 
        \[
            \begin{split}
                & \E \big( \W_1(\nubp_M(s), \mubp(s)) \big)  \leq \E\big( \W_1(\nubp_M(s), \delta_0) \big) +   \W_1(\mubp(s), \delta_0)  \\
                & \leq \E\Big( \int_{\R^2} |y| \, \d \nubp_M(s)(y)  \Big) +   \int_{\R^2} |y| \, \d \mubp(s)(y) \leq  \frac{1}{M} \sum_{m=1}^M \E( |\bar Y^M_m(s)|) +   \int_{\R^2} |y| \, \d \mubp(s)(y) \\
                & \leq   \E( |\bar Y(s)|) +   \int_{\R^2} |y| \, \d \mubp(s)(y) \leq 2 \E( \|\bar Y\|_\infty)   < +\infty \, ,
            \end{split}
        \]
        where the finiteness of the last term follows from Proposition~\ref{prop:solution to averaged ODE/SDE}. We conclude that 
        \begin{equation} \label{eq:int of W1 to zero}
            \int_0^T \E \big( \W_1(\nubp_M(s), \mubp(s)) \big) \, \d s \to 0 
        \end{equation}
        as $M \to +\infty$. This concludes the proof of \eqref{eq:intermediate convergence}. 

        \step{5} (Conclusion with proof of~\eqref{eq:needed convergence}) By~\eqref{eq:bound of W1}, we have that 
        \[
         \int_0^T \E\big( \W_1(\nup_M(s), \nubp_M(s))  \big)    \, \d s \leq T \E\Big( \max_m \|Y^M_m - \bar Y^M_m\|_\infty \Big) \, .  
        \]
        Combining this with~\eqref{eq:int of W1 to zero} and~\eqref{eq:intermediate convergence}, we obtain~\eqref{eq:needed convergence} and we conclude the proof.

    \end{proof}

\begin{proposition} \label{prop:PDE formulation of ODE/SDE/ODE}
    Under the assumptions of Theorem~\ref{thm:limit to averaged system}, the curve $\bar X = (\bar X_1,\dots,\bar X_N)$, the law $\mubp \in \Pcal_1(C^0([0,T];\R^2))$, and the curve $Z = (Z_1,\dots,Z_L)$ from~\eqref{eq:m ODE Y average M} are solutions to the ODE/PDE/ODE system: 
    \begin{equation}  \label{eq:ODE/PDE N}
        \left\{
            \begin{aligned}
                & \frac{\d \bar X_n}{\d t}(t) = v^N(\bar X(t))\Big(  \r(\bar X_n(t)) + \Kcp*\mubp(t)(\bar X_n(t))    \Big)  \, ,      \\
                & \de_t \mubp - \kappa \Delta_y \mubp + \div_y\Big( \Big( \frac{1}{L} \sum_{\ell=1}^L \Kpg(\cdot - Z_\ell(t)) - \frac{1}{N} \sum_{n=1}^N \Kpc(\cdot - \bar X_n(t))  \Big)\mubp \Big)  = 0 \, ,\\
                & \frac{\d Z_\ell}{\d t}(t)  = \Big( \frac{1}{L}\sum_{\ell'=1}^L \Kgg(Z_\ell(t) - Z_{\ell'}(t) )  + u_\ell(t)  \Big) \d t \, ,  \\ 
                & \bar X_n(0) =  X_n^0 \, , \  \ Z_\ell(0) =  Z_\ell^0 \, , \quad n = 1,\dots,N \, , \ \ell = 1,\dots, L \, , \\
                & \mubp(0)  = \Law(\bar Y^0) \, ,
            \end{aligned}
        \right.
    \end{equation}
    where the parabolic PDE is understood in the sense of distributions.\footnote{To be precise, we regard $\mubp \in \Pcal_1(C^0([0,T];\R^2))$ as the distribution defined by the duality
        \[
          \int_{-\infty}^0 \int_{\R^2} \xi(t,y) \, \d \mubp(0)(y) \, \d t + \int_0^T \int_{\R^2} \xi(t,y) \, \d \mubp(t)(y) \, \d t  \quad \text{ for every } \xi \in C^\infty_c((-\infty,T) \x \R^2)) \, .
        \]
        Note that the function $t \mapsto \int_{\R^2} \xi(t,y) \, \d \mubp(t)(y) = \int_{C^0([0,T];\R^2)} \xi(t,\varphi(t)) \, \d \mu(\varphi)$ is continuous in $t$, \eg, by the Dominated Convergence Theorem. A solution to the PDE in the sense of distributions satisfies 
        \[
            \begin{split}
                & \int_{\R^2} \xi(0,y) \, \d \Law(\bar Y^0_1)(y) + \int_0^T \int_{\R^2} \Big[  \de_t \xi(t, y)  + \kappa \Delta_y \xi(t,y) \\
                & \quad \quad +  \Big( \frac{1}{L} \sum_{\ell=1}^L \Kpg( y - Z_\ell(t) )  - \frac{1}{N} \sum_{n=1}^N \Kpc(y - \bar X_n(t)) \Big) \cdot \nabla_y \xi(t, y) \Big] \, \d \mubp(t)(y) \, \d t = 0 
            \end{split}
        \]
        for every $\xi \in C^\infty_c((-\infty,T) \x \R^2))$. \label{footnote:distributional solution}}
\end{proposition}
\begin{proof}
    We exploit the fact that $\mubp$ is the law of the stochastic processes $(\bar Y(t))_{t \in [0,T]}$, where $(\bar Y(t))_{t \in [0,T]}$ solves the SDE 
    \[
            \left\{
                \begin{aligned}
                    & \d \bar Y(t) = \Big( \frac{1}{L} \sum_{\ell=1}^L \Kpg(\bar Y(t) - Z_\ell(t))  - \frac{1}{N} \sum_{n=1}^N \Kpc(\bar Y(t) - \bar X_n(t) ) \Big) \d t + \sqrt{2\kappa} \, \d W(t) \, ,  \\
                    & \bar Y(0)  = \bar Y^0 \ \text{a.s.}
                \end{aligned}
            \right.
    \]   
    Let us fix a test function $\xi \in C^\infty_c((-\infty,T) \x \R^2)$. By It\^{o}'s formula \cite[Theorem~6.4]{Mao}, we have that $(\xi(t,\bar Y(t)))_{t \in [0,T]}$ is an It\^{o} process solving the SDE
    \[
        \begin{split}
            & \d \big (\xi(t, \bar Y(t))  \big) = \de_t \xi(t, \bar Y(t)) \, \d t + \kappa \Delta_y \xi(t,\bar Y(t)) \, \d t \\ 
            & \quad +  \Big( \frac{1}{L} \sum_{\ell=1}^L \Kpg(\bar Y(t) - Z_\ell(t))  - \frac{1}{N} \sum_{n=1}^N \Kpc(\bar Y(t) - \bar X_n(t) ) \Big) \cdot \nabla_y \xi(t, \bar Y(t)) \, \d t \\
            & \quad     + \nabla_y \xi(t,\bar Y(t)) \cdot  \d W(t)
        \end{split}
    \]
    with initial datum $\xi(0,\bar Y^0)$. This means that a.s.\ for every $t \in [0,T]$ 
    \[
        \begin{split}
            &  \xi(t, \bar Y(t)) = \xi(0,\bar Y^0) +\int_0^t \Big[  \de_t \xi(s, \bar Y(s))  + \kappa \Delta_y \xi(s,\bar Y(s))    \\ 
            & \quad +   \Big(\frac{1}{L} \sum_{\ell=1}^L \Kpg(\bar Y(s) - Z_\ell(s) )  - \frac{1}{N} \sum_{n=1}^N \Kpc( \bar Y(s) - \bar X_n(s) ) \Big) \cdot \nabla_y \xi(s, \bar Y(s)) \Big] \, \d s \\
            & \quad + \int_0^t  \nabla_y \xi(s,\bar Y(s)) \cdot \d W(s) \, .
        \end{split}
    \]
    By \cite[Lemma~5.4]{Mao} we have that 
    \[
    \E \Big( \int_0^t  \nabla_y \xi(s,\bar Y(s)) \cdot \d W(s) \Big)  = 0 \, ,
    \]
    thus, taking the expectation, we obtain in particular that
    \[
        \begin{split}
            &  \E\big(\xi(T, \bar Y(T)) \big) = \E\big(\xi(0,\bar Y^0)\big) + \int_0^T \E \Big[  \de_t \xi(t, \bar Y(t))  + \kappa \Delta_y \xi(t,\bar Y(t))    \\ 
            & \quad +   \Big( \frac{1}{L} \sum_{\ell=1}^L \Kpg( \bar Y(t) - Z_\ell(t))  - \frac{1}{N} \sum_{n=1}^N \Kpc(\bar Y(t) - \bar X_n(t) ) \Big) \cdot \nabla_y \xi(t, \bar Y(t)) \Big] \, \d t \, .
        \end{split}
    \]
    Using the fact that $\mubp(t) = \Law(\bar Y(t))$ and $\xi(T,\cdot) \equiv 0$, we get that 
    \[
        \begin{split}
            &  0 = \int_{\R^2} \xi(T, y) \, \d \mubp(T)(y)  = \int_{\R^2} \xi(0,y) \, \d \Law(\bar Y^0)(y) + \int_0^T \int_{\R^2} \Big[  \de_t \xi(t, y)  + \kappa \Delta_y \xi(t,y)    \\ 
            & \quad +   \Big( \frac{1}{L} \sum_{\ell=1}^L \Kpg( y - Z_\ell(t) )  + \frac{1}{N} \sum_{n=1}^N \Kpc(y - \bar X_n(t)) \Big) \cdot \nabla_y \xi(t, y) \Big] \, \d \mubp(t)(y) \, \d t \, .
        \end{split}
    \]
    This concludes the proof.
\end{proof}

\subsection{Limit of optimal control problems as $M \to + \infty$} Let us consider the following cost functional for the limit problem obtained in~\eqref{thm:limit to averaged system}. Let $\J_N \colon L^\infty([0,T]; \U) \to \R$ be defined for every $u \in L^\infty([0,T];\U)$ by 
\begin{equation} \label{def:JN}
    \J_N(u) :=   \frac{1}{2}   \int_0^T    |u(t)|^2 \d t + \frac{1}{N}   \sum_{n=1}^N  \int_0^T  \int_{\R^2} \Hd(\bar X_n(t) - y) \, \d \mubp(t)(y) \, \d t  \, , 
\end{equation}
where $\bar X = (\bar X_1, \dots, \bar X_N)$ and $(\bar Y(t))_{t \in [0,T]}$ are the unique strong solutions to~\eqref{eq:ODE Y average} provided by Proposition~\ref{prop:solution to averaged ODE/SDE}.
    Notice that the definition of $\J_N$ does not depend on the particular initial random datum $\bar Y^0$, but only on its law, since this is also the case for $\mubp$ by Proposition~\ref{prop:identically distributed}.

\begin{theorem} \label{thm:Gamma convergence for M}
    Let us fix $N \geq 1$. Under the assumptions of Theorem~\ref{thm:limit to averaged system}, the sequence of functionals $(\J_{N,M})_{M \geq 1}$ $\Gamma$-converges to $\J_N$ as $M \to +\infty$ with respect to the weak* topology in $L^\infty([0,T];\U)$.\footnote{Note that the weak* convergence in $L^\infty([0,T];\U)$ is metrizable, since $\U$ is bounded, hence we can use the sequential characterization of $\Gamma$-limits, \cf~\cite[Proposition~8.1]{DM}. \label{footnote:metrizability}}
\end{theorem}
\begin{proof}
    \step{1} (Asymptotic lower bound). Let us fix a sequence of controls $(u^M)_{M \geq 1}$, $u^M \in L^\infty([0,T];\U)$ such that $u^M \wstar u$ weakly* in $L^\infty([0,T];\U)$ as $M \to +\infty$ . Let us show that 
    \begin{equation}  \label{eq:liminf inequality M}
        \J_N(u) \leq \liminf_{M \to +\infty} \J_{N,M}(u^M) \, .    
    \end{equation}
    On the one hand, by definition~\eqref{def:expected cost}, we have that 
    \begin{equation*}
        \J_{N,M}(u^M) =  \frac{1}{2}  \int_0^T    |u^M(t)|^2 \d t + \E \Big(\int_0^T \frac{1}{N} \frac{1}{M} \sum_{n=1}^N \sum_{m=1}^M \Hd(X^M_n(t) - Y^M_m(t) )  \, \d t  \Big) \, ,
    \end{equation*}
    where the stochastic processes $(X^M(t))_{t \in [0,T]} = (X^M_1(t),\dots,X^M_N(t))_{t \in [0,T]}$, $(Y^M(t))_{t \in [0,T]} = (Y^M_1(t), \dots, Y^M_M(t))_{t \in [0,T]}$ (and the curve $Z^M = (Z^M_1,\dots,Z^M_L)$) are the unique strong solution to~\eqref{eq:ODE random M}. On the other hand, we have that 
    \[
    \J_N(u) :=   \frac{1}{2}   \int_0^T    |u(t)|^2 \d t + \frac{1}{N}   \sum_{n=1}^N  \int_0^T  \int_{\R^2} \Hd(\bar X_n(t) - y) \, \d \mubp(t)(y) \, \d t  \, , 
    \]
    where the curve $\bar X = (\bar X_1, \dots, \bar X_N)$, the stochastic process $(\bar Y(t))_{t \in [0,T]}$ with law $\mubp$ (and the curve $Z = (Z_1,\dots,Z_L)$) are the unique strong solution  to~\eqref{eq:ODE Y average}.

    By the weak sequential lower semicontinuity of the $L^2$-norm, we have that 
    \[
    \int_0^T |u(t)|^2 \, \d t \leq     \liminf_{M \to +\infty} \int_0^T |u^M(t)|^2 \, \d t \, .
    \]
    Let us prove the convergence 
    \begin{equation} \label{eq:2211231748}
        \E \Big(\int_0^T \frac{1}{N} \frac{1}{M} \sum_{n=1}^N \sum_{m=1}^M \Hd(X^M_n(t) - Y^M_m(t) )  \, \d t  \Big) \to \frac{1}{N}   \sum_{n=1}^N  \int_0^T  \int_{\R^2} \Hd(\bar X_n(t) - y) \, \d \mubp(t)(y) \, \d t\,,
    \end{equation}
    as $M \to +\infty$. This will conclude the proof of~\eqref{eq:liminf inequality M}.  

    We exploit the equality 
    \[
        \frac{1}{M} \sum_{m=1}^M \Hd(X^M_n(t) - Y^M_m(t)) =  \Hd * \nup_M(t)(X^M_n(t))  
    \]
    to deduce that 
    \begin{equation} \label{eq:2212041255}
        \begin{split}
            & \Big| \E \Big(\int_0^T \! \frac{1}{N}\frac{1}{M} \sum_{n=1}^N\sum_{m=1}^M \Hd(X^M_n(t) - Y^M_m(t))  \, \d t  \Big) -   \frac{1}{N} \sum_{n=1}^N \int_0^T \! \int_{\R^2} \Hd(\bar X_n(t) - y) \, \d \mubp(t)(y) \, \d t  \Big| \\
            & \quad  \leq  \frac{1}{N}\sum_{n=1}^N \int_0^T \Big|    \E \Big( \Hd * \nup_M(t)(X^M_n(t))   \Big)   -   \Hd*\mubp(t)(\bar X_n(t))    \Big|\, \d t \\
            & \quad  \leq  \frac{1}{N}\sum_{n=1}^N \int_0^T     \E \Big( \big| \Hd * \nup_M(t)(X^M_n(t))    -   \Hd*\nup_M(t)(\bar X_n(t))  \big| \Big)  \, \d t \\
            & \quad \quad +  \frac{1}{N}\sum_{n=1}^N \int_0^T     \E \Big( \big| \Hd * \nup_M(t)(\bar X_n(t))    -   \Hd*\mubp(t)(\bar X_n(t))  \big| \Big)  \, \d t \, .
        \end{split}
    \end{equation}
    We estimate the first term on the right-hand side of~\eqref{eq:2212041255} by using the fact that, by the Lipschitz continuity of $\Hd$, a.s.\ for every $t \in [0,T]$
    \[  
        \begin{split}
            & \big| \Hd * \nup_M(t)(X^M_n(t))    -   \Hd*\nup_M(t)(\bar X_n(t))  \big| \\
             & \quad \leq   \int_{\R^2} \big| \Hd(X^M_n(t) - y) - \Hd(\bar X_n(t) - y) \big| \, \d \nup_M(t)(y)\\
            & \quad \leq C |X^M_n(t) - \bar X_n(t)| \leq C \max_{n'} \| X^M_{n'} - \bar X_{n'} \|_\infty \, .
        \end{split}
    \]  
    We estimate the second term on the right-hand side of~\eqref{eq:2212041255} by Kantorovich's duality, which by the Lipschitz continuity of $\Hd(\bar X_n(t) - \cdot)$ yields a.s.\ for every $t \in [0,T]$
    \[
        \begin{split}
            \big| \Hd * \nup_M(t)(\bar X_n(t))    -   \Hd*\mubp(t)(\bar X_n(t))  \big| & = \Big| \int_{\R^2} \Hd(\bar X_n(t)-y) \, \d \Big( \nup_M(t) - \mubp(t)\Big)(y) \Big|  \\
            &  \leq C \W_1(\nup_M(t),\mubp(t)) \, .
        \end{split}
    \]
        Putting together the previous inequalities, we conclude that
        \[  
            \begin{split}
                & \Big| \E \Big(\int_0^T  \frac{1}{N}\frac{1}{M} \sum_{n=1}^N\sum_{m=1}^M \Hd(X^M_n(t) - Y^M_m(t))  \, \d t  \Big) -   \frac{1}{N} \sum_{n=1}^N \int_0^T  \int_{\R^2} \Hd(\bar X_n(t) - y) \, \d \mubp(t)(y) \, \d t  \Big| \\
                & \quad \leq C \E\Big( \max_n \| X^M_n - \bar X_n \|_\infty \Big) + C \int_0^T \E \Big( \W_1(\nup_M(t),\mubp(t))\Big) \, \d t \, ,
            \end{split}
        \]  
        whence~\eqref{eq:2211231748} by Theorem~\ref{thm:limit to averaged system}.

    \step{2} (Asymptotic upper bound). Let us fix $u \in L^\infty([0,T];\U)$. For every $M \geq 1$, let us set $u^M = u$. As in {\itshape Step 1}, we have that 
    \begin{equation*}
        \J_{N,M}(u^M) =  \frac{1}{2}  \int_0^T    |u(t)|^2 \d t + \E \Big(\int_0^T \frac{1}{N} \frac{1}{M} \sum_{n=1}^N \sum_{m=1}^M \Hd(X^M_n(t) - Y^M_m(t))  \, \d t  \Big) \, ,
    \end{equation*}
    where the stochastic processes $(X^M(t))_{t \in [0,T]}$, $(Y^M(t))_{t \in [0,T]}$ (and the curve $Z^M$) are the unique strong solution to~\eqref{eq:ODE random M} corresponding to the control $u^M = u$ and 
    \[
    \J_N(u) :=   \frac{1}{2}   \int_0^T    |u(t)|^2 \d t + \frac{1}{N}   \sum_{n=1}^N  \int_0^T  \int_{\R^2} \Hd(\bar X_n(t) - y) \, \d \mubp(t)(y) \, \d t  \, , 
    \]
    where the curve $\bar X$, the stochastic process $(\bar Y(t))_{t \in [0,T]}$ with law $\mubp$ (and the curve $Z$) are the unique strong solution  to~\eqref{eq:ODE Y average}. Trivially, we have $u^M \wstar u$, hence we deduce \eqref{eq:2211231748} once again and, in particular, the asymptotic upper bound 
    \[
     \lim_{M \to +\infty} \J_{N,M}(u) = \J_N(u) \, .    
    \]
    This concludes the proof. 
\end{proof}

As a byproduct, we obtain the following result. 

\begin{proposition} \label{prop:existence of optimal control N}
    Under the assumptions of Proposition~\ref{prop:solution to averaged ODE/SDE}, there exists an optimal control $u^* \in L^\infty([0,T];\U)$, \ie,  
    \[
      \J_{N}(u^*) = \min_{u \in L^\infty([0,T];\U)} \J_{N}(u)\, .  
    \] 
\end{proposition}
\begin{proof}
    The proof is standard in the theory of $\Gamma$-convergence. Let us consider a sequence of independent Brownian motions $(W_m(t))_{t \in [0,T]}$, $m \geq 1$ and $Y^0_1, \dots, Y^0_M$ i.i.d.\ random variable with the same law of $\bar Y^0$. Let $(u^M)_{M \geq 1}$ be a sequence such that $\J_{N,M}(u^M) = \inf \J_{N,M}$. Since $(u^M)_{M \geq 1}$ is bounded in $L^\infty([0,T];\U)$, there exists $u^*$ and a subsequence (not relabeled) such that $u^M \wstar u^*$ weakly-* in $L^\infty([0,T];\U)$. By Theorem~\ref{thm:Gamma convergence for M} we have that 
    \[
        \begin{split}
            \J_N(u^*) & \leq \liminf_{M \to +\infty} \J_{N,M}(u^M) =  \liminf_{M \to +\infty} \inf \J_{N,M} \leq  \limsup_{M \to +\infty} \inf \J_{N,M} \\
            &  \leq \limsup_{M \to +\infty} \J_{N,M} (u^*)  = \J_N(u^*)\, .
        \end{split}
    \]
    (Here we used that the recovery sequence for $u^*$ is the constant sequence given by $u^*$, see the proof of Theorem~\ref{thm:Gamma convergence for M}.)
\end{proof}



\section{Mean-field limit for a large number of commercial ships} \label{sec:N to infty}

In this section we study the limit of the problem as $N \to +\infty$. For this reason we will stress the dependence of initial data and solutions on $N$.

\subsection{Mean-field limit as $N \to +\infty$}

In this section, we will use the explicit formula for the velocity correction
\[
        v^N_n(X)  = v\Big(\frac{1}{N-1} \sum_{n' = 1}^N \eta\big(X_n, X_n - X_{n'}\big)\Big)  = v\Big(\frac{1}{N} \sum_{n' = 1}^N \eta_N\big(X_n, X_n - X_{n'} \big)\Big) \, ,
\]
where we set 
\begin{equation}\label{def:etaN}
    \eta_N = \frac{N}{N-1} \eta \, .
\end{equation} 

In what follows, we shall use the symbol $*_2$ to indicate that the convolution is done with respect to the second variable, \ie, $\eta *_2   \nu(x) = \int_{\R^2} \eta(x,x-x') \, \d \nu(x')$.

\begin{theorem} \label{thm:N to infty}
Assume the following:
    \begin{itemize}
        \item Let $(W(t))_{t \in [0,T]}$ be a $\R^2$-valued Brownian motion;
        \item Let $X^{N,0} = (X_1^0,\dots,X_N^0) \in \R^{2\x N}$ and assume that $\max_n \|X^{N,0}_n\|_\infty \leq R_0$ with $R_0$ independent from $N$;
        \item Let $\bar Y^0$ be a $\R^2$-valued random variable with $\E(|\bar Y^0|) < + \infty$;
        \item Let $Z^0 = (Z_1^0,\dots,Z_L^0) \in \R^{2 \x L}$;
        \item Let $\muc_0 \in \Pcal_1(\R^2)$ with $\supp(\muc_0) \subset \bar B_{R_0}$ be such that $\W_1(\frac{1}{N} \sum_{n=1}^N \delta_{X_n^0}, \muc_0) \to 0$ as $N \to +\infty$;
    \end{itemize}
    Let $u^N, u \in L^\infty([0,T];\U)$ be such that $u^N \wstar u$ weakly* in $L^\infty([0,T];\U)$.\footnote{In fact, by the boundedness of $\U$, this is equivalent to requiring that $u^N \weak u$ weakly in $L^1([0,T];\U)$.}  Let $\bar X^N = (\bar X^N_1,\dots, \bar X^N_N)$, $(\bar Y^N(t))_{t \in [0,T]}$, and $Z^N = (Z^N_1,\dots, Z^N_L)$ be the unique strong solution to\footnote{This corresponds to the averaged ODE/SDE/ODE system~\eqref{eq:ODE Y average} with initial data $X^{N,0}$, $\bar Y^0$, $Z^0$, with Brownian motion $W$ and control $u^N$. The solution is provided by Proposition~\ref{prop:solution to averaged ODE/SDE}.}
    \begin{equation}  \label{eq:ODE/SDE N}
    \left\{
        \begin{aligned}
            & \frac{\d \bar X^N_{n}}{\d t}(t) = v^N_n(\bar X^N(t))\Big(  \r(\bar X^N_{n}(t)) + \Kcp*\mubp_N(t)(\bar X^N_{n}(t))    \Big)  \, ,      \\
            & \d \bar Y^N(t) = \Big(\frac{1}{L} \sum_{\ell=1}^L \Kpg( \bar Y^N(t) - Z^N_\ell(t))  - \frac{1}{N} \sum_{n=1}^N \Kpc(\bar Y^N(t) - \bar X^N_{n}(t)) \Big) \d t  + \sqrt{2\kappa} \, \d W(t) \, ,   \\
            & \frac{\d Z^N_\ell}{\d t}(t)  =  \frac{1}{L}\sum_{\ell'=1}^L \Kgg(Z^N_\ell(t) - Z^N_{\ell'}(t))  + u^N_\ell(t)   \, ,  \\ 
            & \bar X^N_{n}(0) =  X_n^0 \, , \  \ Z^N_\ell(0) =  Z_\ell^0 \, , \quad n = 1,\dots,N \, , \ \ell = 1,\dots, L \, , \\
            & \bar Y^N(0)  = \bar Y^0 \ \text{a.s.,} \quad \mubp_N = \Law(\bar Y^N) \, .
        \end{aligned}
        \right.    
    \end{equation}
    Let us consider the measures 
    \begin{equation} \label{eq:nucN}
        \nuc_N(t) := \frac{1}{N} \sum_{n=1}^N \delta_{\bar X^N_n(t)} \, .
    \end{equation}
    Then there exists $\muc \in C^0([0,T];\Pcal_1(\R^2))$, $(\bar Y(t))_{t \in [0,T]}$, and $Z=(Z_1,\dots,Z_L)$ such that  
    \[
      \sup_{t \in [0,T]}\W_1( \nuc_N(t),\muc(t)) + \E\Big( \|\bar Y^N - \bar Y\|_\infty \Big) + \|Z^N - Z\|_\infty  \to 0 \quad \text{as } N \to +\infty \, . 
    \]
    Moreover,  $\muc \in C^0([0,T];\Pcal_1(\R^2))$, $(\bar Y(t))_{t \in [0,T]}$, and $Z=(Z_1,\dots,Z_L)$ provide the unique solution to 
    \begin{equation} \label{eq:limit PDE/SDE/ODE}
        \left\{
        \begin{aligned}
            & \de_t \muc + \div_x \Big( v\big( \eta *_2 \muc \big) \big(  \r + \Kcp*\mup   \big) \muc \Big) = 0 \, ,\\
            & \d \bar Y(t) = \Big(\frac{1}{L} \sum_{\ell=1}^L \Kpg( \bar Y(t) - Z_\ell(t))  -   \Kpc*\muc(t)(\bar Y(t)) \Big) \d t  + \sqrt{2\kappa} \, \d W(t) \, ,   \\
            & \frac{\d Z_\ell}{\d t}(t)  =   \frac{1}{L}\sum_{\ell'=1}^L \Kgg(Z_\ell(t) - Z_{\ell'}(t))  + u_\ell(t)   \, ,  \\ 
            & \muc(0) = \muc_0 \, , \\
            & \bar Y(0)  = \bar Y^0 \ \text{a.s.,} \quad \mup = \Law(\bar Y) \, ,\\
            & Z_\ell(0) =  Z_\ell^0 \, , \ \ell = 1,\dots, L \, .
        \end{aligned}
        \right.    
    \end{equation}
\end{theorem}

\begin{proof}
\step{1} (PDE solved by the empirical measures)
In terms of $\nuc_N(t)$, $v^N_n(\bar X^N(t))$ reads
\begin{equation} \label{eq:alternative v}
    \begin{split}
        v^N_n(\bar X^N(t)) & = v\Big(\frac{1}{N} \sum_{n' = 1}^N \eta_N\big(\bar X^N_n(t),\bar X^N_n(t) - \bar X^N_{n'}(t)  \big)\Big) \\
        & = v\Big(\int_{\R^2} \eta_N(\bar X^N_n(t), \bar X^N_n(t) - x' ) \, \d \nuc_N(t)(x')\Big) = v\big( \eta_N *_2 \nuc_N(t) (\bar X^N_n(t)) \big) \, .
    \end{split}
\end{equation}

Let us derive the PDE solved by $\nuc_N(t)$ in the sense of distributions.\footnote{We use here the duality introduced in Footnote~\ref{footnote:distributional solution}.} Let us fix $\xi \in C^\infty_c((-\infty,T) \x \R^2)$. By~\eqref{eq:ODE/SDE N} and~\eqref{eq:alternative v} we have that 
\begin{equation} \label{eq:nucN distributional solution}
    \begin{split}
        0 & = \frac{\d}{\d t} \Big( \int_{-\infty}^0 \int_{\R^2} \xi(t,x) \, \d \nuc_N(0)(x) \, \d t + \int_0^T \int_{\R^2} \xi(t,x) \, \d \nuc_N(t)(x) \, \d t \Big) \\
        &   = \frac{1}{N} \sum_{n=1}^N \frac{\d}{\d t} \Big(  \int_{-\infty}^0  \xi(t, X^0_n)   \, \d t +  \int_0^T  \xi(t,\bar X^N_n(t))   \, \d t \Big) \\
        &  = \frac{1}{N} \sum_{n=1}^N  \Big(  \int_{-\infty}^0  \de_t \xi(t, X^0_n)   \, \d t +  \int_0^T  \Big( \de_t \xi(t,\bar X^N_n(t)) + \frac{\d \bar X^N_n}{\d t}(t) \cdot \nabla_x \xi(t,\bar X^N_n(t))  \, \d t \Big) \\
        & = \frac{1}{N} \sum_{n=1}^N \Big( \xi(0,X^0_n) + \int_0^T  \Big( \de_t \xi(t,\bar X^N_n(t))\  + \\
        & \hspace{1cm} + v\big( \eta_N *_2 \nuc_N(t) (\bar X^N_n(t)) \big)  \big(  \r(\bar X^N_{n}(t)) + \Kcp*\mubp_N(t)(\bar X^N_{n}(t))    \big) \cdot \nabla_x \xi(t,\bar X^N_n(t))  \Big) \, \d t \Big) \\
        & = \int_{\R^2} \xi(0,x) \, \d \Big( \frac{1}{N} \sum_{n=1}^N \delta_{X^0_n} \Big)(x) + \int_0^T \!\! \int_{\R^2} \! \Big( \de_t \xi(t, x) \ +  \\
        & \hspace{1cm} +  v\big( \eta_N *_2 \nuc_N(t) (x) \big) \big(  \r(x) + \Kcp*\mubp_N(t)(x)    \big) \cdot \nabla_x \xi(t,x)  \Big) \, \d \nuc_N(t)(x) \, \d t \, .
    \end{split}
\end{equation}
This means that $\nuc_N$ is a distributional solution to 
\begin{equation} \label{eq:PDE nuN}
    \left\{
        \begin{aligned}
            & \de_t \nuc_N + \div_x \Big( v\big( \eta_N *_2 \nuc_N \big) \big(  \r + \Kcp*\mubp_N   \big) \nuc_N \Big) = 0 \, ,\\
            & \nuc_N(0) = \frac{1}{N} \sum_{n=1}^N \delta_{X^0_n} \, .
        \end{aligned}
        \right.
\end{equation}

\step{2} (Convergence of empirical measures $\nuc_N$) To show compactness of the sequence of curves $\nuc_N \in C^0([0,T];\Pcal_1(\R^2))$ we rely on a Arzel\`a-Ascoli Theorem for metric-valued functions. We split the proof in substeps.

\substep{2}{1} (Equiboundedness of supports)   By Remark~\ref{rmk:bar X is bounded}, we have that $\max_n \|\bar X^N_n \|_\infty \leq R$, where the constant $R$ depending on the initial datum $X^0$, the final time $T$, $\|v^N\|_\infty$, $\r$, and $\Kcp$. This implies that $\supp\big(\nuc_N(t)\big)$ are contained in the closed ball $\bar B_R$ for every $t \in [0,T]$.  

\substep{2}{2} (Equicontinuity) Let us prove that $\nuc_N \in C^0([0,T];\Pcal_1(\R^2))$ are equicontinuous.

We observe that he sequence $\|Z^N\|_\infty$ is bounded. Indeed, by~\eqref{eq:ODE/SDE N}, 
\[
    \begin{split}
        |Z^N_\ell(t)| & \leq  |Z^0_\ell| + \int_0^t \Big( \Big| \frac{1}{L}\sum_{\ell'=1}^L \Kgg( Z^N_\ell(s) - Z^N_{\ell'}(s)) \Big|  + |u^N_\ell(s)|  \Big)   \, \d s \\
        & \leq |Z^0| + \int_0^t C (1 + |Z^N(s)|) \, \d s  \leq |Z^0| + CT + \int_0^t C |Z^N(s)| \, \d s \, ,
    \end{split}
\]
the constant $C$ depending on $\Kgg$ and the set of admissible controls $\U$ (bounded). Taking the norm of $Z^N$ and by Gr\"onwall's inequality, we obtain that 
\[
|Z^N(t)| \leq (|Z^0| + CT) e^{Ct} \leq R'\, ,    
\]
where the constant $R'$ depends on $\Kgg$, $\U$, and $T$. 

By Remark~\ref{rmk:bar X is bounded}, for every $r \in [0,T]$ we have that 
\begin{equation} \label{eq:bound on EYN}
    \int_{\R^2} |y| \, \d \mubp_N(r)(y) = \E(|\bar Y^N(r)|) \leq C(1+\E(|\bar Y^0|)) \, ,
\end{equation}
where the constant $C$ depends on $\Kpg$, $\Kpc$, $\|Z^N\|_{\infty}$ (bounded by $R'$), $\max_n \|\bar X^N_n \|_{\infty}$ (bounded by $R$), $T$, and $W$. Then the Lipschitz continuity of $\Kcp$ and~\eqref{eq:bound on EYN} yield 
\begin{equation} \label{eq:bound on Kcp}
    \begin{split}
        |\Kcp*\mubp_N(r)(x)| & \leq  \int_{\R^2} | \Kcp(x-y) | \, \d \mubp_N(r)(y) \leq \int_{\R^2} ( |\Kcp(0)| + C | x | + C | y |  ) \, \d \mubp_N(r)(y) \\ 
        & \leq C(1+\E(|\bar Y^0|)+|x|) \leq C(1+|x|) \, ,
    \end{split}
\end{equation}
where the constant $C$ additionally depends on $\E(|\bar Y^0|)$. 

By~\eqref{eq:ODE/SDE N} and~\eqref{eq:bound on Kcp}, for $s \leq t$ and $n=1,\dots,N$ we have that 
\[
    \begin{split}
        |\bar X^N_n(s) - \bar X^N_n(t)| & \leq \int_s^t |v^N_n(\bar X^N(r))\big(  \r(\bar X^N_{n}(r)) + \Kcp*\mubp_N(r)(\bar X^N_{n}(r))    \big)| \, \d r \\
        & \leq \int_s^t C ( 1 + |\bar X^N_{n}(r)|) \, \d r \leq C |t - s| \, ,
    \end{split}
\]
where the constant $C$ depends on the constant obtained in~\eqref{eq:bound on Kcp} and additionally on $\|v^N\|_\infty$ and $\r$. Using as transport plan between $\nuc_N(s)$ and $\nuc_N(t)$ the measure $\gamma = \frac{1}{N} \sum_{n=1}^N \delta_{(\bar X^N_n(s), \bar X^N_n(t))}$, we obtain that 
\[
\W_1(\nuc_N(s),\nuc_N(t)) \leq  \int_{\R^2 \x \R^2} |x - x'| \, \d \gamma(x,x') = \frac{1}{N} \sum_{n=1}^N |\bar X^N_n(s) - \bar X^N_n(t)| \leq C|t-s| \, .
\]
\ie, the curves $\nuc_N \in C^0([0,T]; \Pcal_1(\bar B_R))$ are equi-Lipschitz with respect to the 1-Wasserstein distance.

\substep{2}{3} (Compactness) Since the ball $\bar B_R$ is compact, the Wasserstein space $\Pcal_1(\bar B_R)$ is compact too~\cite[Remark 6.19]{Vil}.\footnote{In fact, the curves $\nuc_{N} \in C^0([0,T];\Pcal_1(\R^2))$ take values in a compact set of $\Pcal_1(\R^2)$ independent of $N$ even under weaker assumptions. This is the case, \eg, when $q$-moments of $\nuc_{N}(t)$ with $q>1$ are uniformly bounded, \ie, $\sup_N \sup_t \int_{\R^2} |x|^q \, \d \nuc_N(t)(x) < +\infty$ for some $q>1$ (this can be proven basing on~\cite[Theorem~6.9]{Vil}. A uniform bound on the $q$-moments follows from the analogous assumption on the distribution of initial data by a Gr\"onwall inequality. } Hence the Arzel\`a-Ascoli Theorem for continuous functions with values in a metric space guarantees the existence of a curve $\muc \in C^0([0,T]; \Pcal_1(\bar B_R))$ and a subsequence $N_k$ such that
\begin{equation} \label{eq:W1 converges}
    \sup_{t \in [0,T]}\W_1(\nuc_{N_k}(t), \muc(t)) \to 0 \quad \text{as } N_k \to +\infty \, .
\end{equation}

Without loss of generality, we do not relabel this subsequence and denote it simply by $N$. This does not affect the proof, as in Theorem~\ref{thm:uniqueness of PDE/SDE/ODE} we shall prove uniqueness of solutions for the limit problem. 

\step{3} (Convergence of $Z^N$)
We let $Z = (Z_1,\dots, Z_L)$ be the unique solution to 
\[
    \left\{
        \begin{aligned}
            & \frac{\d Z_\ell}{\d t}(t)  =   \frac{1}{L}\sum_{\ell'=1}^L \Kgg( Z_\ell(t) - Z_{\ell'}(t))  + u_\ell(t)   \, ,  \\ 
            & Z_\ell(0) =  Z_\ell^0 \, .
        \end{aligned}
    \right.    
\]
As in {\itshape Substep 4.1} in the proof of Theorem~\ref{thm:limit to averaged system}, we get that 
\begin{equation} \label{eq:ZN converges}
    \|Z^N - Z \|_{\infty} \to 0  \quad \text{as } N \to +\infty\, .
\end{equation}
 
\step{4} (Convergence of $\bar Y^N$) Let us consider the SDE 
\begin{equation} \label{eq:SDE for bar Y}
    \left\{
        \begin{aligned}
            & \d \bar Y(t) = \Big(\frac{1}{L} \sum_{\ell=1}^L \Kpg( \bar Y(t) - Z_\ell(t))  -   \Kpc*\muc(t)(\bar Y(t)) \Big) \d t  + \sqrt{2\kappa} \, \d W(t) \, ,   \\
            & \bar Y(0)  = \bar Y^0 \ \text{a.s.} \\
        \end{aligned}
        \right.   
    \end{equation}
We will show that $\bar Y^N$ converges to $\bar Y$. 

\substep{4}{1} (Well-posedness of~\eqref{eq:SDE for bar Y}) There exists a unique strong solution to~\eqref{eq:SDE for bar Y}. Indeed, let us consider the drift 
\[
b(t,Y) :=  \frac{1}{L} \sum_{\ell=1}^L \Kpg( Y - Z_\ell(t))  -   \Kpc*\muc(t)(Y) 
\]
and the constant dispersion matrix $\sigma = \sqrt{2\kappa} \, \Id_2$, so that 
\[
    \left\{
        \begin{aligned}
            & \d \bar Y(t) = b(t,\bar Y(t)) \d t  + \sigma \, \d W(t) \, ,   \\
            & \bar Y(0)  = \bar Y^0 \ \text{a.s.} \\
        \end{aligned}
        \right.  
\]
Let us observe that $b$ is continuous in $t$ and Lipschitz-continuous in $Y$ (with Lipschitz constant independent of $t$). Indeed, $Z$ is a continuous curve, while by Kantorovich's duality 
\begin{equation} \label{eq:Kpcmuct Lipschitz}
    \begin{split}
        & | \Kpc*\muc(t)(Y) - \Kpc*\muc(s)(Y) |\\
        &   = \Big| \int_{\R^2} \Kpc(Y - x) \, \d \Big( \muc(t) - \muc(s)\Big)(x)  \Big| \leq C \W_1(\muc(t),\muc(s)) \, ,
    \end{split}
\end{equation}
and $t \mapsto \muc(t)$ is a continuous curve in the Wasserstein space $\Pcal_1(\R^2)$. Moreover, the function $Y \mapsto \Kpg( Y - Z_\ell(t))$ is Lipschitz-continuos, and so is $Y \mapsto \Kpc*\muc(t)(Y)$, since 
\[
    \begin{split}
        | \Kpc*\muc(t)(Y) -  \Kpc*\muc(t)(Y')| &  \leq  \int_{\R^2} |\Kpc(Y - x)  -  \Kpc(Y' - x) | \, \d \muc(t)(x) \\
        & \leq  \int_{\R^2} C |Y - Y' | \, \d \muc(t)(x) = C |Y - Y'| \, .
    \end{split}
\]
Moreover, we have that 
\[
|b(t,Y)| \leq |b(t,0)| + |b(t,0) - b(t,Y)| \leq |b(t,0)| + C |Y|
\]
and 
\[
    \begin{split}
        |b(t,0)| & \leq     \frac{1}{L} \sum_{\ell=1}^L |\Kpg(- Z_\ell(t))|  +   |\Kpc*\muc(t)(0)| \\
        & \leq C (1 + |Z(t)|) + \int_{\R^2} C(1+|x|) \, \d \muc(t)(x) \leq C (1 + \|Z\|_\infty) +  C(1+R)  \leq C \, , 
    \end{split}
\]
where the last inequality follows from the fact that $Z$ is bounded and $\muc(t)$ has support in the ball $\bar B_R(0)$ for every $t \in [0,T]$. We conclude that 
\begin{equation} \label{eq:bound on b}
    |b(t,Y)| \leq C(1+|Y|) \, ,
\end{equation}
where the constant $C$ depends on $\Kpg$, $\Kpc$, $\|Z\|_\infty$, $R$. Thus the assumptions of Proposition~\ref{prop:SDE simple} are satisfied. Proposition~\ref{prop:SDE simple} also gives us that
\begin{equation} \label{eq:bound on bar Y}
    \E(\|\bar Y\|_\infty) \leq C \, ,    
\end{equation}
where the constant $C$ depends on $\Kpg$, $\Kpc$, $\|Z\|_\infty$, $R$, $\bar Y^0$, $T$, and $W$.


\substep{4}{2} (Convergence of $\bar Y^N$ to $\bar Y$) Let us prove that 
\begin{equation} \label{eq:YN converges}
    \E\Big(\|\bar Y^N - \bar Y\|_\infty\Big)  \to 0 \quad \text{as } N \to +\infty \, .
\end{equation}
For, we start by noticing that 
\[
    \frac{1}{N} \sum_{n=1}^N \Kpc(\bar Y^N(t) - \bar X^N_{n}(t)) = \Kpc*\nuc_N(t)(\bar Y^N(t)) \, .
\]
Hence, by~\eqref{eq:ODE/SDE N}, \eqref{eq:SDE for bar Y}, \eqref{eq:Kpcmuct Lipschitz}, and by Kantorovich's duality, we have a.s.\ for $0 \leq s \leq t \leq T$
\[
    \begin{split}
     & |\bar Y^N(s) - \bar Y(s)| \leq  \frac{1}{L} \sum_{\ell=1}^L \int_0^s     | \Kpg( \bar Y^N(r) - Z^N_\ell(r))   -    \Kpg( \bar Y(r) - Z_\ell(r)) | \, \d r \\
     & \hspace{4cm} + \int_0^s    |\Kpc*\nuc_N(r)(\bar Y^N(r)) - \Kpc*\muc(r)(\bar Y(r)) |  \, \d r \\
     & \leq  \int_0^s C |\bar Y^N(r) - \bar Y(r)| \, \d r + CT \|Z^N - Z\|_\infty \\
     & \quad + \int_0^s    |\Kpc*\nuc_N(r)(\bar Y^N(r)) - \Kpc*\muc(r)(\bar Y^N(r)) |  \, \d r \\
     & \quad + \int_0^s    |\Kpc*\muc(r)(\bar Y^N(r)) - \Kpc*\muc(r)(\bar Y(r)) |  \, \d r \\
     & \leq \int_0^s C |\bar Y^N(r) - \bar Y(r)| \, \d r + CT \|Z^N - Z\|_\infty  + CT \sup_{r \in [0,T]} \W_1(\nuc_N(r), \muc(r))   \, .
    \end{split}
\] 
Taking the supremum and the expectation, we deduce that 
\[
    \begin{split}
        \E\Big(\sup_{0\leq s \leq t}|\bar Y^N(s) - \bar Y(s)| \Big) & \leq \int_0^t C \E\Big( \sup_{0\leq r \leq s} |\bar Y^N(r) - \bar Y(r)| \Big) \, \d s \\
        & \hspace{1cm} + CT \|Z^N - Z\|_\infty  + CT \sup_{r \in [0,T]} \W_1(\nuc_N(r), \muc(r))
    \end{split}
  \]
  and, by Gr\"onwall's inequality, 
  \[
    \E\Big(\sup_{0\leq s \leq t}|\bar Y^N(s) - \bar Y(s)| \Big) \leq CT \Big(\|Z^N - Z\|_\infty  + \sup_{r \in [0,T]} \W_1(\nuc_N(r), \muc(r)) \Big) e^{Ct} \, .
  \]
    In particular,
    \[
    \E\Big(\|\bar Y^N - \bar Y\|_\infty\Big)  \leq C  \Big(\|Z^N - Z\|_\infty  + \sup_{r \in [0,T]} \W_1(\nuc_N(r), \muc(r)) \Big) \, ,
    \]
    the constant $C$ depending also on $T$. By~\eqref{eq:ZN converges} and~\eqref{eq:W1 converges}, we obtain~\eqref{eq:YN converges}.

    \step{5} (Limit problem)
    With~\eqref{eq:W1 converges}, \eqref{eq:ZN converges}, and~\eqref{eq:YN converges}  at hand, we are in a position to pass to the limit as $N \to +\infty$ in~\eqref{eq:PDE nuN} and prove that $\muc$ is a distributional solution to 
    \begin{equation} \label{eq:PDE muc}
        \left\{
            \begin{aligned}
                & \de_t \muc + \div_x \Big( v\big( \eta *_2 \muc \big) \big(  \r + \Kcp*\mup   \big) \muc\Big) = 0 \, ,\\
                & \muc(0) = \muc_0 \, ,
            \end{aligned}
            \right.
    \end{equation}
    \ie, 
    \begin{equation} \label{eq:muc distributional solution}
        \begin{split}
            0 = & \int_{\R^2} \xi(0,x) \, \d  \muc_0(x)    \\
            & \hspace{0.2cm}+  \int_0^T \!\! \int_{\R^2} \! \Big( \de_t \xi(t, x) +  v\big( \eta *_2 \muc(t) (x) \big) \big(  \r(x) + \Kcp*\mup(t)(x)    \big) \cdot \nabla_x \xi(t,x)  \Big) \, \d \muc(t)(x) \, \d t \, .
        \end{split}
    \end{equation}
        We divide the proof in substeps.


    \substep{5}{1} (Convergence of initial datum term) By the Lipschitz continuity of $x \mapsto \xi(0,x)$ and by Kantorovich's duality, we have that 
    \[
    \Big| \int_{\R^2} \xi(0,x) \, \d \Big(\frac{1}{N} \sum_{n=1}^N \delta_{X^0_n} - \muc_0 \Big)(x)\Big| \leq C \W_1\Big(\frac{1}{N} \sum_{n=1}^N \delta_{X_n^0}, \muc_0\Big) \, .
    \] 
     By the assumption on the initial data, we have that $\W_1(\frac{1}{N} \sum_{n=1}^N \delta_{X_n^0}, \muc_0) \to 0$, hence
    \begin{equation} \label{eq:initial datum term}
        \int_{\R^2} \xi(0,x) \, \d \Big(\frac{1}{N} \sum_{n=1}^N \delta_{X^0_n}\Big)(x) \to \int_{\R^2} \xi(0,x) \, \d \muc_0 \, .
    \end{equation}

    \substep{5}{2} (Convergence of time-derivative term)
    Since $x \mapsto \de_t \xi(t,x)$ is Lipschitz-continuous with a Lipschitz constant independent of $t$, by Kantorovich's duality we have that 
    \[
    \Big| \int_{\R^2} \de_t \xi(t,x) \, \d \Big( \nuc_N(t) - \muc(t) \Big)(x)  \Big| \leq C \W_1(\nuc_N(t), \muc(t)) \, ,
    \]
    for every $t$. By~\eqref{eq:W1 converges} it follows that $\int_{\R^2} \de_t \xi(t,x) \, \d \nuc_N(t)(x) \to \int_{\R^2} \de_t \xi(t,x) \, \d  \muc(t)(x)$ as $N \to +\infty$ uniformly in $t$, thus 
    \begin{equation} \label{eq:time-derivative term}
        \int_0^T \int_{\R^2} \de_t \xi(t,x) \, \d \nuc_N(t)(x) \, \d t \to \int_0^T \int_{\R^2} \de_t \xi(t,x) \, \d  \muc(t)(x) \, \d t \, .
    \end{equation}

    \substep{5}{3} (Convergence of divergence term -- I)
   Let us show that 
    \begin{equation} \label{eq:divergence term - I}
        \begin{split}
            & \int_0^T \int_{\R^2} v\big( \eta *_2 \nuc_N(t) (x) \big)     \r(x) \cdot \nabla_x \xi(t,x)  \, \d \nuc_N(t)(x) \, \d t \\
            & \hspace{2cm} \to \int_0^T \int_{\R^2} v\big( \eta *_2 \muc(t) (x) \big)     \r(x) \cdot \nabla_x \xi(t,x)  \, \d \muc(t)(x) \, \d t    \quad \text{as } N \to +\infty \, . 
        \end{split}
    \end{equation}
    We start by splitting
    \begin{equation} \label{eq:2211281906}
        \begin{split}
            &  \Big|  \int_0^T \int_{\R^2} v\big( \eta *_2 \nuc_N(t) (x) \big)     \r(x) \cdot \nabla_x \xi(t,x)  \, \d \nuc_N(t)(x) \, \d t \\
            & \hspace{1.5cm} - \int_0^T \int_{\R^2} v\big( \eta *_2 \muc(t) (x) \big)     \r(x) \cdot \nabla_x \xi(t,x)  \, \d \muc(t)(x) \, \d t \Big|  \\
            & \leq \int_0^T \int_{\R^2} \Big|v\big( \eta_N *_2 \nuc_N(t) (x) \big) - v\big( \eta *_2 \muc(t) (x) \big) \Big| \Big|\r(x)\nabla_x \xi(t,x)\Big| \, \d \nuc_N(t)(x) \, \d t   \\
            & \hspace{0.5cm} + \Big| \int_0^T \int_{\R^2} v\big( \eta *_2 \muc(t) (x) \big)     \r(x) \cdot \nabla_x \xi(t,x)  \, \d \Big( \nuc_N(t) - \muc(t) \Big)(x) \, \d t \Big| \, .
        \end{split}
        \end{equation}
    By the Lipschitz continuity of $v$, by~\eqref{def:etaN}, by the Lipschitz continuity of $\eta$, and by Kantorovich's duality we have that for every $x \in \R^2$ and $t \in [0,T]$ 
    \begin{equation} 
        \begin{split}
            & \Big|v\big( \eta_N *_2 \nuc_N(t) (x) \big) - v\big( \eta *_2 \muc(t) (x) \big)\Big|  \leq C \Big|\eta_N *_2 \nuc_N(t) (x) -   \eta *_2 \muc(t) (x) \Big| \\
            & \leq C \Big|\eta_N *_2 \nuc_N(t) (x) -   \eta *_2 \nuc_N(t) (x) \Big| +  C \Big|\eta *_2 \nuc_N(t) (x) -   \eta *_2 \muc(t) (x) \Big| \\
            & \leq C   \int_{\R^2}  |  \eta_N(x,x-x')  -  \eta(x,x-x') | \, \d \nuc_N(t)(x')  + C \Big|\int_{\R^2}\eta(x,x-x') \, \d \Big(  \nuc_N(t) - \muc(t)  \Big)(x')   \Big| \\
            & \leq C \Big( \frac{1}{N-1} +  \sup_{s \in [0,T]}\W_1( \nuc_N(s), \muc(s) ) \Big) \, ,
        \end{split}
    \end{equation}  
    where the constant $C$ depends on $v$ and $\eta$. Integrating in time and space and using the fact that $|\r(x)| \leq C(1+|x|)$, thus it is bounded on the compact support of $\xi$, we obtain that 
    \begin{equation} \label{eq:2211281646}
        \begin{split}
            & \int_0^T \int_{\R^2} \Big|v\big( \eta_N *_2 \nuc_N(t) (x) \big) - v\big( \eta *_2 \muc(t) (x) \big)\Big| \Big| \r(x) \nabla_x \xi(t,x) \Big|  \, \d \nuc_N(t)(x) \, \d t \\
            & \hspace{1cm} \leq C \Big( \frac{1}{N-1} +  \sup_{s \in [0,T]}\W_1( \nuc_N(s), \muc(s) ) \Big) \, ,
        \end{split}
    \end{equation}
    where the constant $C$ depends on $v$, $\eta$, $\r$, $\xi$, and $T$.
    
    Moreover, the function $x \mapsto v\big( \eta *_2 \muc(t) (x) \big)     \r(x) \cdot \nabla_x \xi(t,x)$ is Lipschitz-continuous with a Lipschitz constant independent of $t$ and depending on $v$, $\eta$, $\r$, and $\xi$. For, $x \mapsto v\big( \eta *_2 \muc(t) (x) \big)$ satisfies the latter property, since
    \begin{equation} \label{eq:2211291659}        
        \begin{split}
            & | v\big( \eta *_2 \muc(t) (x) \big)  - v\big( \eta *_2 \muc(t) (x') \big)|  \leq C |  \eta *_2 \muc(t) (x) - \eta *_2 \muc(t) (x') |  \\
            & \hspace{0.5cm} \leq C \int_{\R^2} \Big|  \eta(x,x-x'') - \eta(x',x'-x'') \Big| \, \d \muc(t) (x'') \leq C |x - x'| \, ,
        \end{split}
    \end{equation}
    where the constant $C$ depends on $v$ and $\eta$. As above, $\r$ is bounded on the support of $\xi$. By the Lipschitz continuity of $\r$ and $\nabla_x \xi$, we conclude that the product $x \mapsto v\big( \eta *_2 \muc(t) (x) \big)     \r(x) \cdot \nabla_x \xi(t,x)$ is also Lipschitz-continuous. 
    Thus by Kantorovich's duality we obtain that for every $t \in [0,T]$
    \begin{equation} \label{eq:2211281647}
        \begin{split}
            & \Big|  \int_{\R^2} v\big( \eta *_2 \muc(t) (x) \big)     \r(x) \cdot \nabla_x \xi(t,x)  \, \d \Big( \nuc_N(t) - \muc(t)\Big)(x) \Big| \leq C \sup_{s \in [0,T]}\W_1(\nuc_N(s),\muc(s)) \, ,
        \end{split}
    \end{equation}
    where $C$ depends on $v$, $\eta$, $\r$, $\xi$. Combining \eqref{eq:2211281906}, \eqref{eq:2211281646}, and \eqref{eq:2211281647}, by~\eqref{eq:W1 converges} it follows that 
    \[
        \begin{split}
            &  \Big|  \int_0^T \int_{\R^2} v\big( \eta *_2 \nuc_N(t) (x) \big)     \r(x) \cdot \nabla_x \xi(t,x)  \, \d \nuc_N(t)(x) \, \d t \\
            & \hspace{1.5cm} - \int_0^T \int_{\R^2} v\big( \eta *_2 \muc(t) (x) \big)     \r(x) \cdot \nabla_x \xi(t,x)  \, \d \muc(t)(x) \, \d t \Big|  \\
            & \leq \int_0^T \int_{\R^2} C \Big|v\big( \eta_N *_2 \nuc_N(t) (x) \big) - v\big( \eta *_2 \muc(t) (x) \big) \Big|  \, \d \nuc_N(t)(x) \, \d t   \\
            & \hspace{0.5cm} + \Big| \int_0^T \int_{\R^2} v\big( \eta *_2 \muc(t) (x) \big)     \r(x) \cdot \nabla_x \xi(t,x)  \, \d \nuc_N(t)(x) \, \d t\\
            & \hspace{1.5cm} - \int_0^T \int_{\R^2} v\big( \eta *_2 \muc(t) (x') \big)     \r(x') \cdot \nabla_x \xi(t,x')  \, \d \muc(t)(x') \, \d t \Big| \\
            & \leq C \Big( \frac{1}{N-1} + \sup_{t \in [0,T]}\W_1( \nuc_N(t), \muc(t) ) \Big) \to 0 \quad \text{as } N \to +\infty \, ,
        \end{split}
    \]
    where the constant $C$ depends on $v$, $\eta$, $\r$, $\xi$, and $T$.

    \substep{5}{4} (Convergence of divergence term -- II) Let us prove that
    \begin{equation} \label{eq:divergence term - II}
        \begin{split}
            & \int_0^T \int_{\R^2} v\big( \eta_N *_2 \nuc_N(t) (x) \big)  \Kcp*\mubp_N(t)(x)  \cdot \nabla_x \xi(t,x)  \, \d \nuc_N(t)(x) \, \d t  \\
            & \hspace{1cm} \to \int_0^T \int_{\R^2} v\big( \eta *_2 \muc(t) (x) \big)  \Kcp*\mup(t)(x)  \cdot \nabla_x \xi(t,x)  \, \d \muc(t)(x) \, \d t \quad \text{as } N \to +\infty \, .
        \end{split}
    \end{equation}
    We start by splitting 
    \begin{equation} \label{eq:2211281912}
        \begin{split}
            & \Big| \int_0^T \int_{\R^2} v\big( \eta_N *_2 \nuc_N(t) (x) \big)  \Kcp*\mubp_N(t)(x)  \cdot \nabla_x \xi(t,x)  \, \d \nuc_N(t)(x) \, \d t  \\
            & \hspace{0.5cm} - \int_0^T \int_{\R^2} v\big( \eta *_2 \muc(t) (x) \big)  \Kcp*\mup(t)(x)  \cdot \nabla_x \xi(t,x)  \, \d \muc(t)(x) \, \d t \Big| \\
            & \hspace{0.1cm} \leq \int_0^T \int_{\R^2} \Big| v\big( \eta_N *_2 \nuc_N(t) (x) \big) - v\big( \eta *_2 \muc(t) (x) \big) \Big| \Big| \Kcp*\mubp_N(t)(x)  \cdot \nabla_x \xi(t,x) \Big| \, \d \nuc_N(t)(x) \, \d t \\
            & \hspace{0.5cm} + \int_0^T \int_{\R^2} \Big| v\big( \eta *_2 \muc(t) (x) \big) \Big| \Big|\Big( \Kcp*\mubp_N(t)(x)  - \Kcp*\mup(t)(x) \Big) \cdot \nabla_x \xi(t,x) \Big| \, \d \nuc_N(t)(x) \, \d t \\
            & \hspace{0.5cm} + \Big| \int_0^T \int_{\R^2}   v\big( \eta *_2 \muc(t) (x) \big)    \Kcp*\mup(t)(x)  \cdot \nabla_x \xi(t,x)  \, \d \Big( \nuc_N(t) - \muc(t)\Big) (x) \, \d t  \Big|  \, .
        \end{split}
    \end{equation}
    For the first term in the right-hand side of~\eqref{eq:2211281912}, we argue analogously to~\eqref{eq:2211281646} to obtain that 
    \[
        \begin{split}
            & \int_0^T \int_{\R^2} \Big| v\big( \eta_N *_2 \nuc_N(t) (x) \big) - v\big( \eta *_2 \muc(t) (x) \big) \Big| \Big| \Kcp*\mubp_N(t)(x)  \cdot \nabla_x \xi(t,x) \Big| \, \d \nuc_N(t)(x) \, \d t \\
            & \leq C \Big( \frac{1}{N-1} +  \sup_{s \in [0,T]}\W_1( \nuc_N(s), \muc(s) ) \Big) \to 0 \quad \text{as } N \to +\infty \, ,
        \end{split}
    \]
    where the constant $C$ depends on $v$, $\eta$, $\Kcp$, $\xi$, and $T$. The only difference consists in the fact that we have $\Kcp*\mubp_N(t)(x)$ in place of $\r(x)$. For this, we need to observe that
    \begin{equation} \label{eq:2211291704}
        \begin{split}
            |\Kcp*\mubp_N(t)(x)| & = \Big| \int_{\R^2} \Kcp(x - y) \, \d \mubp_N(t)(y) \Big| \leq \int_{\R^2} \Big( |\Kcp(0)| + C|x| + C|y| \Big) \, \d \mubp_N(t)(y) \\
            & \leq C (1 + |x|) \, .
        \end{split}
    \end{equation}
    In the last inequality, we used the fact that, since $\mubp_N(t)$ is the law of $\bar Y^N(t)$,  
    \[
    \int_{\R^2} |y| \, \d \mubp_N(t)(y) \leq \E(\|\bar Y^N\|_\infty) \leq C \, ,
    \]
    where the boundedness follows from the convergence~\eqref{eq:YN converges}.

    For the second term in the right-hand side of~\eqref{eq:2211281912}, we start by observing that $\Kcp$ is Lipschitz, thus we have for every $x \in \R^2$ and $t \in [0,T]$ 
    \[
        \begin{split}
            & \Big| \Kcp*\mubp_N(t)(x)  - \Kcp*\mup(t)(x)\Big|  \\
            & \quad = \Big| \int_{\R^2} \Kcp(x - y) \, \d \mubp_N(t)(y) - \int_{\R^2} \Kcp(x - y') \, \d \mup(t) (y') \Big|  \\
            & \quad = \Big| \E\Big( \Kcp(x - \bar Y^N(t)) - \Kcp(x - \bar Y(t) )  \Big) \Big|  \leq C \E\Big( |  \bar Y^N(t) -  \bar Y(t)  | \Big)  \leq \E\Big(\|\bar Y^N - \bar Y\|_\infty\Big)
        \end{split} 
    \]
    By~\eqref{eq:YN converges}, it follows that 
    \[
        \begin{split}
            & \int_0^T \int_{\R^2} \Big| v\big( \eta *_2 \muc(t) (x) \big) \Big| \Big|  \Kcp*\mubp_N(t)(x)  - \Kcp*\mup(t)(x) \Big| \Big| \nabla_x \xi(t,x) \Big| \, \d \nuc_N(t)(x) \, \d t     \\
            & \hspace{1cm} \leq C T \|v\|_\infty \|\nabla_x \xi\|_\infty \E\Big(\|\bar Y^N - \bar Y\|_\infty\Big) \to 0 \quad \text{as } N \to +\infty  \, .
        \end{split}
    \]

    For the third term in the right-hand side of~\eqref{eq:2211281912}, we observe that the function $x \mapsto v\big( \eta *_2 \muc(t) (x) \big)    \Kcp*\mup(t)(x)  \cdot \nabla_x \xi(t,x)$ is Lipschitz-continuous with a Lipschitz constant independent of $t$ and depending on $v$, $\eta$, $\Kcp$, $\mup$, and $\xi$. This follows from~\eqref{eq:2211291659}, from the fact that $\xi$ is compactly supported, and the inequality 
    \[
        |\Kcp*\mup(t)(x)|  \leq C (1 + |x|)
    \]
    obtained as in~\eqref{eq:2211291704}. By Kantorovich's duality,
    \[
        \begin{split}
            & \Big| \int_0^T \int_{\R^2}   v\big( \eta *_2 \muc(t) (x) \big)    \Kcp*\mup(t)(x)  \cdot \nabla_x \xi(t,x)  \, \d \Big( \nuc_N(t) - \muc(t)\Big) (x) \, \d t  \Big| \\
            & \quad \leq C \W_1(\nuc_N(t),\muc(t))  \leq \sup_{s \in [0,T]} C \W_1( \nuc_N(s),\muc(s) ) \to 0 \quad \text{as } N \to +\infty \, ,
        \end{split}
    \]
    by~\eqref{eq:W1 converges}.
   
    \substep{5}{5} (Conclusion) Combining~\eqref{eq:nucN distributional solution}, \eqref{eq:initial datum term}, \eqref{eq:time-derivative term}, \eqref{eq:divergence term - I}, and~\eqref{eq:divergence term - II}, we conclude the proof of~\eqref{eq:muc distributional solution}.

    We prove the uniqueness of the solution in Theorem~\ref{thm:uniqueness of PDE/SDE/ODE} below.

    \end{proof}

    \begin{theorem} \label{thm:uniqueness of PDE/SDE/ODE}
        Under the assumptions of Theorem~\ref{thm:N to infty}, the solution $\muc \in C^0([0,T];\Pcal_1(\R^2))$, $(\bar Y(t))_{t \in [0,T]}$, and $Z=(Z_1,\dots,Z_L)$ to~\eqref{eq:limit PDE/SDE/ODE} is unique.
    \end{theorem}
    \begin{proof}
        The uniqueness of $Z$ is direct, as the ODE for $Z$ is decoupled from the first two equations. 

        Assume now that $\muc_i \in C^0([0,T];\Pcal_1(\R^2))$, $(\bar Y_i(t))_{t \in [0,T]}$ for $i=1,2$ are solutions to~\eqref{eq:limit PDE/SDE/ODE} with the same initial data, \ie,
        \begin{equation} \label{eq:two problems}
            \left\{
                \begin{aligned}
                    & \de_t \muc_i + \div_x \Big( v\big( \eta *_2 \muc_i \big) \big(  \r + \Kcp*\mup_i   \big) \muc_i \Big) = 0 \, ,\\
                    & \d \bar Y_i(t) = \Big(\frac{1}{L} \sum_{\ell=1}^L \Kpg( \bar Y_i(t) - Z_\ell(t))  -   \Kpc*\muc_i(t)(\bar Y_i(t)) \Big) \d t  + \sqrt{2\kappa} \, \d W(t) \, ,   \\
                    & \frac{\d Z_\ell}{\d t}(t)  = \Big( \frac{1}{L}\sum_{\ell'=1}^L \Kgg(Z_\ell(t) - Z_{\ell'}(t))  + u_\ell(t)  \Big) \d t \, ,  \\ 
                    & \muc_i(0) = \muc_0 \, , \\
                    & \bar Y_i(0)  = \bar Y^0 \ \text{a.s.,} \quad \mup_i = \Law(\bar Y_i)  \, , \\
                    & Z_\ell(0) =  Z_\ell^0 \, , \ \ell = 1,\dots, L \, .
                \end{aligned}
                \right.    
            \end{equation}
        where $\supp(\muc_i(0)) = \supp(\muc_0) \subset \bar B_R$. As customary in uniqueness proofs for evolutionary problems, we will temporary neglect the assumption that the initial data $\bar Y_1(0)$ and $\bar Y_2(0)$ are a.s.\ equal and $\muc_1(0)$ and $\muc_2(0)$ are equal in order to carry out a Gr\"onwall-type argument to deduce stability with respect to initial data.  The objective is to prove the following pair of estimates:
        \begin{gather}
            \E( \|\bar Y_1 - \bar Y_2\|_\infty ) \leq C \Big( \E(|\bar Y_1(0) - \bar Y_2(0)| ) +  \int_0^t \W_1(\muc_1(s),\muc_2(s)) \, \d s \Big) \, ,  \label{eq:aimed Gronwall barY} \\
            \W_1(\muc_1(t), \muc_2(t))  \leq  C \Big( \W_1(\muc_1(0), \muc_2(0))    + \E(\|\bar Y_1 - \bar Y_2\|_\infty) \Big) \, .  \label{eq:aimed Gronwall W1}
        \end{gather}
        These two inequalities provide uniqueness when combined. Indeed, if $\bar Y_1(0) = \bar Y^0 = \bar Y_2(0)$ a.s.\ and $\muc_1(0) = \muc_0 = \muc_2(0)$, then \eqref{eq:aimed Gronwall barY} simply reads 
        \begin{equation*}
            \E( \|\bar Y_1 - \bar Y_2\|_\infty ) \leq C \int_0^t \W_1(\muc_1(s),\muc_2(s)) \, \d s  \, .
        \end{equation*}
        Substituting into~\eqref{eq:aimed Gronwall W1}, we get that 
        \begin{equation*}
            \W_1(\muc_1(t), \muc_2(t))  \leq  C \int_0^t \W_1(\muc_1(s),\muc_2(s)) \, \d s  \, ,
        \end{equation*}
        which by Gr\"onwall's inequality yields $\W_1(\muc_1(t), \muc_2(t))  = 0$ for all $t \in [0,T]$. Then~\eqref{eq:aimed Gronwall barY} gives $\E( \|\bar Y_1 - \bar Y_2\|_\infty ) = 0$.

        We divide the proof of~\eqref{eq:aimed Gronwall barY}--\eqref{eq:aimed Gronwall W1} in several steps.
        
         \step{1} (Estimate of $\E(\|\bar Y_1 - \bar Y_2 \|_\infty)$) 
         By~\eqref{eq:two problems} we have that a.s.\ for $0 \leq s \leq t \leq T$
        \begin{equation} \label{eq:2211301855}
            \begin{split}
                & |\bar Y_1(s) - \bar Y_2(s)| \\
                &  \leq |\bar Y_1(0) - \bar Y_2(0)| + \int_0^s \Big| \Big(\frac{1}{L} \sum_{\ell=1}^L \Kpg( \bar Y_1(r) - Z_\ell(r))  -   \Kpc*\muc_1(r)(\bar Y_1(r)) \Big) \\
                & \hspace{4cm} -  \Big(\frac{1}{L} \sum_{\ell=1}^L \Kpg( \bar Y_2(r) - Z_\ell(r))  -   \Kpc*\muc_2(r)(\bar Y_2(r)) \Big) \Big| \d r \\
                & \leq |\bar Y_1(0) - \bar Y_2(0)| + \frac{1}{L} \sum_{\ell=1}^K \int_0^s |\Kpg( \bar Y_1(r) - Z_\ell(r)) - \Kpg( \bar Y_2(r) - Z_\ell(r))| \, \d r \\
                & \hspace{4cm} + \int_0^s |\Kpc*\muc_1(r)(\bar Y_1(r)) - \Kpc*\muc_1(r)(\bar Y_2(r))| \, \d r \\
                & \hspace{4cm} + \int_0^s |\Kpc*\muc_1(r)(\bar Y_2(r)) - \Kpc*\muc_2(r)(\bar Y_2(r))| \, \d r \, .
            \end{split}
        \end{equation}
        The first integrand in~\eqref{eq:2211301855} is bounded using the Lipschitz continuity of $\Kpg$ by 
        \begin{equation} \label{eq:2211301856}
            |\Kpg( \bar Y_1(r) - Z_\ell(r)) - \Kpg( \bar Y_2(r) - Z_\ell(r))| \leq C |\bar Y_1(r) -  \bar Y_2(r)| \, .
        \end{equation}
        The second integrand in~\eqref{eq:2211301855} is estimated using the Lipschitz continuity of $\Kpc$ as follows 
        \begin{equation} \label{eq:2211301857}
            \begin{split}
                &  |\Kpc*\muc_1(r)(\bar Y_1(r)) - \Kpc*\muc_1(r)(\bar Y_2(r))| \\
                & \quad \leq  \int_{\R^2} | \Kpc(\bar Y_1(r) - x)  -  \Kpc(\bar Y_2(r) - x)  |  \, \d \muc_1(r)(x) \leq C |\bar Y_1(r) - \bar Y_2(r)| \, .
            \end{split}
        \end{equation}
        The third integrand in~\eqref{eq:2211301855} is estimated using the Lipschitz continuity of $\Kpc$ by Kantorovich's duality
        \begin{equation} \label{eq:2211301858}
            \begin{split}
                |\Kpc*\muc_1(r)(\bar Y_2(r)) - \Kpc*\muc_2(r)(\bar Y_2(r))| & = \Big| \int_{\R^2} \Kpc(\bar Y_2(r) - x) \, \d \Big( \muc_1(r) - \muc_2(r) \Big)(x) \Big|   \\
                & \leq C \W_1(\muc_1(r),\muc_2(r)) \, .
            \end{split}
        \end{equation}
        
        Combining~\eqref{eq:2211301855}--\eqref{eq:2211301858} we infer that a.s.\ for $0 \leq s \leq t \leq T$
        \begin{equation*}
            |\bar Y_1(s) - \bar Y_2(s)| \leq |\bar Y_1(0) - \bar Y_2(0)| + \int_0^s |\bar Y_1(r) -  \bar Y_2(r)| \, \d r + C \int_0^s \W_1(\muc_1(r),\muc_2(r)) \, \d r \, .
        \end{equation*}
        Taking the supremum in $s$ and the expectation we deduce that 
        \begin{equation*}
            \begin{split}
                \E\Big( \sup_{0\leq s \leq t}|\bar Y_1(s) - \bar Y_2(s)| \Big) & \leq \E(|\bar Y_1(0) - \bar Y_2(0)| ) + C \int_0^t \W_1(\muc_1(s),\muc_2(s)) \, \d s \\
                & \quad + C \int_0^t \sup_{0 \leq r \leq s} \E\Big( |\bar Y_1(r) -  \bar Y_2(r)| \Big) \, \d r  \, .
            \end{split}
        \end{equation*}
        By Gr\"onwall's inequality 
        \begin{equation*}
            \E\Big( \sup_{0\leq s \leq t}|\bar Y_1(s) - \bar Y_2(s)| \Big)  \leq \Big( \E(|\bar Y_1(0) - \bar Y_2(0)| ) + C \int_0^t \W_1(\muc_1(s),\muc_2(s)) \, \d s \Big) e^{Ct} \, 
        \end{equation*}
        for $t \in [0,T]$ and, in particular, 
        \begin{equation*}
            \E( \|\bar Y_1 - \bar Y_2\|_\infty ) \leq C \Big( \E(|\bar Y_1(0) - \bar Y_2(0)| ) +  \int_0^t \W_1(\muc_1(s),\muc_2(s)) \, \d s \Big)
        \end{equation*}
        for $t \in [0,T]$, where the constant $C$ also depends on $T$.  

        \step{2} (Introducing the flow for the transport equation) Following an idea in~\cite{PicRos, PicRos2}, we prove uniqueness by regarding the solutions of the transport equation from a Lagrangian point of view. Let us consider for every $x \in \supp( \muc_i(0))$ the flow 
        \begin{equation} \label{eq:flow Phi}
            \left\{
                \begin{aligned}
                    & \de_t \Phi_i(t,x) =  v\Big( \eta *_2 \muc_i(t)(\Phi_i(t,x)) \Big) \Big(  \r(\Phi_i(t,x)) + \Kcp*\mup_i(t)(\Phi_i(t,x))   \Big)  \, , \\
                    & \Phi_i(0,x) = x \, .
                \end{aligned}
                \right.
            \end{equation}
        Then $\muc_i(t) = \Phi_i(t,\cdot)_\# \muc_i(0)$, see~\cite[Theorem~5.34]{Vil2}.

        Let us show that the flows $\Phi_i$ are bounded. We notice that 
        \[
            \begin{split}
                & |\Kcp * \mup_i(t)(X)|  \leq \int_{\R^2} \Big( |\Kcp(0)| + |\Kcp(X - y) - \Kcp(0)| \Big) \, \d \mup_i(t)(y) \\
                & \leq \int_{\R^2} C ( 1 + |X| + |y| ) \, \d \mup_i(t)(y) \leq C( 1+|X| + \E(\|\bar Y_i\|_\infty)) \leq C(1+|X|) \, ,
            \end{split}
        \]
        where we used the bound~\eqref{eq:bound on bar Y}. By~\eqref{eq:flow Phi} and by the estimate $|r(X)| \leq C(1+|X|)$ we deduce that for every $x \in \bar B_R$
        \[
        \begin{split}
            |\Phi_i(t,x)| & \leq |x| + \int_0^t   \|v\|_\infty \Big(  |\r(\Phi_i(s,x))| + |\Kcp*\mup_i(s)(\Phi_i(s,x))|   \Big)   \\
            & \leq |x| + \int_0^t C (1 + |\Phi_i(s,x)|) \, \d s = |x| + Ct + \int_0^t C |\Phi_i(s,x)| \, \d s \, .
        \end{split}
        \]
        By Gr\"onwall's inequality and since $x \in \bar B_R$,  we obtain that 
        \begin{equation} \label{eq:flow is bounded}
            |\Phi_i(t,x)| \leq (|x| + Ct) e^{C t} \leq (R + CT) e^{CT} \leq C \quad \text{for } t \in [0,T] \, ,
        \end{equation}
        where the constant $C$ depends on $\|v\|_\infty$, $\r$, $\Kcp$, $R$, and $T$ (in addition to the constant in~\eqref{eq:bound on bar Y}).
 
        In what follows we will show that 
        \begin{equation} \label{eq:2211301645}
            |\Phi_1(t,x) - \Phi_2(t,x')| \leq C |x - x'|  + C  \Big( \int_0^t  \W_1( \mup_1(s),  \mup_2(s)) \, \d s  + \int_0^t \W_1(\muc_1(s),\muc_2(s)) \, \d s \Big)  \, ,
        \end{equation}
        for $x$, $x' \in \bar B_R$ and $t \in [0,T]$. 
        
        We start by observing that 
        \begin{equation} \label{eq:inequality with Phi}
            |\Phi_1(t,x) - \Phi_2(t,x')| \leq |\Phi_1(t,x) - \Phi_2(t,x)| + |\Phi_2(t,x) - \Phi_2(t,x')|
        \end{equation}
        for every $x$, $x' \in \R^2$ and $t \in [0,T]$. 
        
        \step{3} (Estimate of $|\Phi_1(t,x) - \Phi_2(t,x)|$)  We estimate the first term in the right-hand side of~\eqref{eq:inequality with Phi} as follows:
        \begin{equation} \label{eq:Phi1 - Phi2}
            \begin{split}
            & |\Phi_1(t,x) - \Phi_2(t,x)| \\
            & = \Big| \int_0^t v\Big( \eta *_2 \muc_1(s)(\Phi_1(s,x)) \Big) \Big(  \r(\Phi_1(s,x)) + \Kcp*\mup_1(s)(\Phi_1(s,x))\Big) \, \d s  \\
            & \quad - \int_0^t v\Big( \eta *_2 \muc_2(s)(\Phi_2(s,x)) \Big) \Big(  \r(\Phi_2(s,x)) + \Kcp*\mup_2(s)(\Phi_2(s,x))\Big) \, \d s \Big| \\
            & \leq \int_0^t \Big| v\Big( \eta *_2 \muc_1(s)(\Phi_1(s,x)) \Big) - v\Big( \eta *_2 \muc_2(s)(\Phi_2(s,x)) \Big) \Big|  |\r(\Phi_1(s,x))| \, \d s \\
            & \quad + \int_0^t \Big| v\Big( \eta *_2 \muc_1(s)(\Phi_1(s,x)) \Big) - v\Big( \eta *_2 \muc_2(s)(\Phi_2(s,x)) \Big) \Big|  |\Kcp*\mup_1(s)(\Phi_1(s,x))|  \, \d s  \\
            &  \quad + \int_0^t \|v\|_\infty  |\r(\Phi_1(s,x)) - \r(\Phi_2(s,x))| \, \d s \\
            & \quad + \int_0^t \|v\|_\infty |\Kcp*\mup_1(s)(\Phi_1(s,x)) - \Kcp*\mup_1(s)(\Phi_2(s,x))| \, \d s   \\
            & \quad + \int_0^t \|v\|_\infty|\Kcp*\mup_1(s)(\Phi_2(s,x)) - \Kcp*\mup_2(s)(\Phi_2(s,x))|  \, \d s  \, .
            \end{split}
        \end{equation}
        In the following substeps we estimate the five terms on the right-hand side of~\eqref{eq:Phi1 - Phi2}. 

        \substep{3}{1} Let us estimate the first term in the right-hand side of~\eqref{eq:Phi1 - Phi2}. For $x \in \supp( \muc_i(0))$ and $s \in [0,T]$ we split
        \begin{equation} \label{eq:2211301447}
            \begin{split}
                & \Big| v\Big( \eta *_2 \muc_1(s)(\Phi_1(s,x)) \Big) - v\Big( \eta *_2 \muc_2(s)(\Phi_2(s,x)) \Big) \Big| \\
                & \quad  \leq \Big| v\Big( \eta *_2 \muc_1(s)(\Phi_1(s,x)) \Big) - v\Big( \eta *_2 \muc_1(s)(\Phi_2(s,x)) \Big) \Big| \\
                & \quad \quad + \Big| v\Big( \eta *_2 \muc_1(s)(\Phi_2(s,x)) \Big) - v\Big( \eta *_2 \muc_2(s)(\Phi_2(s,x)) \Big) \Big| \, .
            \end{split}
        \end{equation}
        We exploit the Lipschitz continuity of $v$ and $\eta$ to obtain that 
        \begin{equation} \label{eq:2211301448}
            \begin{split}    
                & \Big| v\Big( \eta *_2 \muc_1(s)(\Phi_1(s,x)) \Big) - v\Big( \eta *_2 \muc_1(s)(\Phi_2(s,x)) \Big) \Big| \\
                & \quad  \leq C | \eta *_2 \muc_1(s)(\Phi_1(s,x))  - \eta *_2 \muc_1(s)(\Phi_2(s,x)) |\leq |\Phi_1(s,x) - \Phi_2(s,x)| \, .
            \end{split}
        \end{equation}
        Moreover, we use the Lipschitz continuity of $x' \mapsto \eta(\Phi_2(s,x), \Phi_2(s,x) - x')$ and Kantorovich's duality to deduce that 
        \begin{equation} \label{eq:2211301449}
            \begin{split}
                & \Big| v\Big( \eta *_2 \muc_1(s)(\Phi_2(s,x)) \Big) - v\Big( \eta *_2 \muc_2(s)(\Phi_2(s,x)) \Big) \Big| \\
                & \leq C |\eta *_2 \muc_1(s)(\Phi_2(s,x)) - \eta *_2 \muc_2(s)(\Phi_2(s,x))| \\
                & \leq C \Big| \int_{\R^2} \eta(\Phi_2(s,x), \Phi_2(s,x) - x') \, \d \Big(  \muc_1(s) -  \muc_2(s)\Big)(x')\Big| \\
                & \leq C \W_1( \muc_1(s),  \muc_2(s)) \, .
            \end{split}
        \end{equation}
        By~\eqref{eq:flow is bounded} we have that for $x \in \bar B_R$ and $t \in [0,T]$ 
        \begin{equation} \label{eq:2211301450}
            |\r(\Phi_1(s,x))| \leq C(1 + |\Phi_1(s,x)|) \leq C \, .
        \end{equation}
        By \eqref{eq:2211301447}--\eqref{eq:2211301450} we get that for every $x \in \bar B_R$ and $t \in [0,T]$ 
        \begin{equation} \label{eq:2211301451}
            \begin{split}
                & \int_0^t \Big| v\Big( \eta *_2 \muc_1(s)(\Phi_1(s,x)) \Big) - v\Big( \eta *_2 \muc_2(s)(\Phi_2(s,x)) \Big) \Big|  |\r(\Phi_1(s,x))| \, \d s \\
                & \leq C \int_0^t \Big(|\Phi_1(s,x) - \Phi_2(s,x)| + \W_1( \muc_1(s),  \muc_2(s)) \Big) \, \d s \, .
            \end{split}
        \end{equation}

        \substep{3}{2} Let us estimate the second term in the right-hand side of~\eqref{eq:Phi1 - Phi2}. By the Lipschitz-continuity of $\Kcp$, we observe that for $x \in \bar B_R$ and $s \in [0,T]$
        \[
            \begin{split}
                |\Kcp*\mup_1(s)(\Phi_1(s,x))| & \leq  \int_{\R^2} |\Kcp(\Phi_1(s,x) - y)| \, \d \mup_1(s)(y)  \\
                & \leq \int_{\R^2} \Big( |\Kcp(0)| + C  |\Phi_1(s,x)| + C |y|\Big) \, \d \mup_1(s)(y)  \\
                & \leq C(1 + |\Phi_1(s,x)|) \leq C \, ,
            \end{split}
        \]
        where we used~\eqref{eq:bound on bar Y} and \eqref{eq:flow is bounded}. Then, as for \eqref{eq:2211301451}, we have that 
        \begin{equation} \label{eq:2211301452}
            \begin{split}
                & \int_0^t \Big| v\Big( \eta *_2 \muc_1(s)(\Phi_1(s,x)) \Big) - v\Big( \eta *_2 \muc_2(s)(\Phi_2(s,x)) \Big) \Big|  |\Kcp*\mup_1(s)(\Phi_1(s,x))| \, \d s \\
                & \leq C \int_0^t \Big(|\Phi_1(s,x) - \Phi_2(s,x)| + \W_1( \muc_1(s),  \muc_2(s)) \Big) \, \d s \, .
            \end{split}
        \end{equation}
    
        \substep{3}{3} Let us estimate the third term in the right-hand side of~\eqref{eq:Phi1 - Phi2}. By the Lipschitz continuity of $\r$, we get that 
        \begin{equation} \label{eq:2211301453}
            \int_0^t \|v\|_\infty  |\r(\Phi_1(s,x)) - \r(\Phi_2(s,x))| \, \d s   \leq C \int_0^t  |\Phi_1(s,x) -  \Phi_2(s,x) | \, \d s \, . 
        \end{equation}

        \substep{3}{4} Let us estimate the fourth term in the right-hand side of~\eqref{eq:Phi1 - Phi2}. We exploit the Lipschitz continuity of $\Kcp$ to deduce that 
        \[
            \begin{split}
                & |\Kcp*\mup_1(s)(\Phi_1(s,x)) - \Kcp*\mup_1(s)(\Phi_2(s,x))| \\
                & \quad \leq  \int_{\R^2} | \Kcp(\Phi_1(s,x) -  y) - \Kcp( \Phi_2(s,x) - y)| \, \d \mup_1(s)  \leq  C | \Phi_1(s,x) - \Phi_2(s,x)| \, , 
            \end{split}
        \]
        from which it follows that 
        \begin{equation} \label{eq:2211301454}
            \begin{split}
                & \int_0^t \|v\|_\infty |\Kcp*\mup_1(s)(\Phi_1(s,x)) - \Kcp*\mup_1(s)(\Phi_2(s,x))| \, \d s  \\
                & \leq C \int_0^t | \Phi_1(s,x) - \Phi_2(s,x)| \, \d s \, .
            \end{split}
        \end{equation}

        \substep{3}{5} Let us estimate the fifth term in the right-hand side of~\eqref{eq:Phi1 - Phi2}. By the Lipschitz continuity of $y \mapsto \Kcp(\Phi_2(s,x) - y)$ we have that 
        \[
            \begin{split}
                & |\Kcp*\mup_1(s)(\Phi_2(s,x)) - \Kcp*\mup_2(s)(\Phi_2(s,x))| \\
                & \quad = \Big| \int_{\R^2} \Kcp(\Phi_2(s,x) - y) \, \d \Big( \mup_1(s) - \mup_2(s) \Big)(y) \Big|  \\
                & \quad \leq  \E\Big( | \Kcp(\Phi_2(s,x) - \bar Y_1(s)) -  \Kcp(\Phi_2(s,x) - \bar Y_2(s)) |  \Big)  \leq C \E( |\bar Y_1(s) - \bar Y_2(s)|  ) \\
                & \quad \leq C \E(\|\bar Y_1 - \bar Y_2\|_\infty) \, ,
            \end{split}
        \]
        from which we infer that  
        \begin{equation}\label{eq:2211301455}
            \int_0^t \|v\|_\infty|\Kcp*\mup_1(s)(\Phi_2(s,x)) - \Kcp*\mup_2(s)(\Phi_2(s,x))|  \, \d s \leq C \E(\|\bar Y_1 - \bar Y_2\|_\infty) \, ,
        \end{equation}
        the constant $C$ depending also on $T$.

        \substep{3}{6} Combining~\eqref{eq:Phi1 - Phi2}, \eqref{eq:2211301451}, \eqref{eq:2211301452}, \eqref{eq:2211301453}, \eqref{eq:2211301454}, and~\eqref{eq:2211301455} we obtain that 
        \begin{equation*}
            \begin{split}
                |\Phi_1(t,x) - \Phi_2(t,x)|  & \leq  C \E(\|\bar Y_1 - \bar Y_2\|_\infty)  + C \int_0^t \W_1(\muc_1(s),\muc_2(s)) \, \d s   \\
                & \quad + C \int_0^t  |\Phi_1(s,x) - \Phi_2(s,x)|   \, \d s  \, .
            \end{split}
        \end{equation*}
        By Gr\"onwall's inequality this yields 
        \begin{equation*}
            |\Phi_1(t,x) - \Phi_2(t,x)|  \leq C e^{Ct} \Big( \E(\|\bar Y_1 - \bar Y_2\|_\infty)  + \int_0^t \W_1(\muc_1(s),\muc_2(s)) \, \d s \Big)
        \end{equation*}
        for $t \in [0,T]$ and, in particular,
        \begin{equation} \label{eq:2211301551}
            |\Phi_1(t,x) - \Phi_2(t,x)|  \leq C  \Big( \E(\|\bar Y_1 - \bar Y_2\|_\infty)  + \int_0^t \W_1(\muc_1(s),\muc_2(s)) \, \d s \Big)  \, ,
        \end{equation}
        for $t \in [0,T]$, with $C$ depending also on $T$. 

        \step{4} (Estimate of $|\Phi_2(t,x) - \Phi_2(t,x')|$) Let us estimate the second term in the right-hand side of~\eqref{eq:inequality with Phi}. By~\eqref{eq:flow Phi}, we have that 
        \begin{equation*} 
            \begin{split}
                & |\Phi_2(t,x) - \Phi_2(t,x')|  \leq | x  - x' | \\
                & \quad  + \Big| \int_0^t v\Big( \eta *_2 \muc_2(t)(\Phi_2(s,x)) \Big) \Big(  \r(\Phi_2(s,x)) + \Kcp*\mup_2(s)(\Phi_2(s,x))   \Big) \, \d s \\
                & \quad  - \int_0^t v\Big( \eta *_2 \muc_2(t)(\Phi_2(s,x')) \Big) \Big(  \r(\Phi_2(s,x')) + \Kcp*\mup_2(s)(\Phi_2(s,x'))   \Big) \, \d s \Big| \\
                & \leq | x  - x' | + \int_0^t \Big| v\Big( \eta *_2 \muc_2(t)(\Phi_2(s,x)) \Big) - v\Big( \eta *_2 \muc_2(t)(\Phi_2(s,x')) \Big) \Big|  |\r(\Phi_2(s,x))| \, \d s \\
                &  \quad + \int_0^t \Big| v\Big( \eta *_2 \muc_2(t)(\Phi_2(s,x)) \Big) - v\Big( \eta *_2 \muc_2(t)(\Phi_2(s,x')) \Big) \Big| |\Kcp*\mup_2(s)(\Phi_2(s,x))|   \, \d s \\
                & \quad + \int_0^t \|v\|_\infty | \r(\Phi_2(s,x)) - \r(\Phi_2(s,x')) | \, \d s \\
                & \quad + \int_0^t \|v\|_\infty | \Kcp*\mup_2(s)(\Phi_2(s,x)) - \Kcp*\mup_2(s)(\Phi_2(s,x')) | \, \d s \, .
            \end{split}
        \end{equation*}
        Reasoning similarly to {\itshape Step 2} (\ie, exploiting the Lipschitz continuity of $v$, $\eta$, $\r$, and $\Kcp$), one shows that for $x$, $x' \in \bar B_R$ and $t \in [0,T]$ 
        \begin{equation*} 
            |\Phi_2(t,x) - \Phi_2(t,x')| \leq |x - x'| + C \int_0^t |\Phi_2(s,x) - \Phi_2(s,x')| \, \d s  \, ,
        \end{equation*}
         which by Gr\"onwall's inequality yields $|\Phi_2(t,x) - \Phi_2(t,x')| \leq |x - x'| e^{Ct}$  for $x$, $x' \in \bar B_R$ and $t \in [0,T]$ and, in particular, 
        \begin{equation}\label{eq:30111602}
            |\Phi_2(t,x) - \Phi_2(t,x')| \leq C |x - x'|  \, ,
        \end{equation}
        for $x$, $x' \in \bar B_R$, where the constant $C$ depends also on $T$.

        Putting together~\eqref{eq:inequality with Phi}, \eqref{eq:2211301551}, and~\eqref{eq:30111602}, we conclude that~\eqref{eq:2211301645} holds true. 

        \step{5} (Estimate of $\W_1(\muc_1(t), \muc_2(t))$) Let $\gamma \in \Pcal(\R^2\x\R^2)$ be an optimal transport plan satisfying $\pi^i_\# \gamma = \muc_i(0)$ and
        \begin{equation} \label{eq:optimal gamma}
            \W_1(\muc_1(0), \muc_2(0)) = \int_{\R^2 \x \R^2} |x - x'| \, \d \gamma(x,x') \, .
        \end{equation}
        We observe that since $\muc_1(0)$ and $\muc_2(0)$ have both supports contained in the closed ball $\bar B_R$, the measure $\gamma$ is also concentrated on a set contained in $\bar B_R \x \bar B_R$, see~\cite[Theorem~5.10-(ii)-(e)]{Vil}. We consider the map $(x,x') \mapsto (\Phi_1(t,\pi^1(x,x')),\Phi_2(t,\pi^2(x,x')))$ and the transport plan $(\Phi_1(t,\pi^1),\Phi_2(t,\pi^2))_\# \gamma$, observing that it has marginals $\muc_1(t)$ and $\muc_2(t)$ since we have that $\pi^i_\# (\Phi_1(t,\pi^1),\Phi_2(t,\pi^2))_\# \gamma = \Phi_i(t,\cdot)_\# \pi^i_\# \gamma = \Phi_i(t,\cdot)_\# \muc_i(0) = \muc_i(t)$. From~\eqref{eq:2211301645} and~\eqref{eq:optimal gamma} it follows that 
        \[
            \begin{split}
                & \W_1(\muc_1(t), \muc_2(t))  \leq \int_{\R^2 \x \R^2} |X - X'| \, \d \Big( (\Phi_1(t,\pi^1),\Phi_2(t,\pi^2))_\# \gamma \Big) (X,X') \\
                & \quad = \int_{\bar B_R \x \bar B_R} |\Phi_1(t,x) - \Phi_2(t,x')| \, \d \gamma(x,x') \\
                & \quad \leq  C  \int_{\bar B_R \x \bar B_R} |x - x'| \, \d \gamma(x,x')  + C \E(\|\bar Y_1 - \bar Y_2\|_\infty) + C \int_0^t  \W_1( \muc_1(s),  \muc_2(s)) \, \d s   \\
                & \quad =  C  \W_1(\muc_1(0), \muc_2(0)) + C \E(\|\bar Y_1 - \bar Y_2\|_\infty) + C \int_0^t  \W_1( \muc_1(s),  \muc_2(s)) \, \d s  \, .
            \end{split}
        \] 
        By Gr\"onwall's inequality 
        \begin{equation*}
            \begin{split}
                \W_1(\muc_1(t), \muc_2(t))  & \leq  C e^{Ct} \Big( \W_1(\muc_1(0), \muc_2(0))   + \E(\|\bar Y_1 - \bar Y_2\|_\infty) \Big)  \\
                & \leq C \Big( \W_1(\muc_1(0), \muc_2(0))   + \E(\|\bar Y_1 - \bar Y_2\|_\infty) \Big) 
            \end{split}
        \end{equation*}
        for $t \in [0,T]$, where the constant $C$ also depends on $T$. This concludes the proof of~\eqref{eq:aimed Gronwall W1} and of the theorem.
    \end{proof}

    \begin{proposition} \label{prop:PDE formulation of PDE/SDE/ODE}
        Under the assumptions of Theorem~\ref{thm:N to infty}, the curve $\muc \in C^0([0,T];\Pcal_1(\R^2))$, the law $\mup \in \Pcal_1(C^0([0,T];\R^2))$, and the curve $Z = (Z_1,\dots,Z_L)$ from~\eqref{eq:limit PDE/SDE/ODE} are solutions to the ODE/PDE/ODE system: 
        \begin{equation}  \label{eq:PDE/SDE/ODE}
            \left\{
                \begin{aligned}
                    & \de_t \muc + \div_x \Big( v\big( \eta *_2 \muc \big) \big(  \r + \Kcp*\mup   \big) \muc \Big) = 0 \, ,\\  
                    & \de_t \mup - \kappa \Delta_y \mup + \div_y\Big( \Big( \frac{1}{L} \sum_{\ell=1}^L \Kpg(\cdot - Z_\ell(t)) -  \Kpc*\muc  \Big)\mup \Big)  = 0 \, ,\\
                    & \d Z_\ell(t)  = \Big( \frac{1}{L}\sum_{\ell'=1}^L \Kgg(Z_\ell(t) - Z_{\ell'}(t) )  + u_\ell(t)  \Big) \d t \, ,  \\ 
                    & \muc(0) = \muc_0 \, , \\
                    & \mup(0)  = \mup_0 \, , \\
                    & Z_\ell(0) =  Z_\ell^0 \, ,  \  \ell = 1,\dots, L \, , 
                \end{aligned}
            \right.
        \end{equation}
        where the PDEs are understood in the sense of distributions and $\mup_0$ is the law of $\bar Y^0$.
    \end{proposition}
    \begin{proof}
        The proof consists in deriving the PDE solved by $\mup$ using It\^{o}'s lemma and is obtained as in the proof of Proposition~\ref{prop:PDE formulation of ODE/SDE/ODE} {\itshape mutatis mutandis}. 
    \end{proof}

    \begin{theorem}  \label{thm:uniqueness PDE/PDE/ODE}
        Under the assumptions of Theorem~\ref{thm:N to infty} and further assuming that:
        \begin{itemize}
            \item $\mup_0 \in \Pcal_2(\R^2)$;
            \item $\mup_0$ has finite entropy, \ie,  $\mup_0 = \rhop_0(x) \, \d x$ for some function $\rhop_0 \in L^1(\R^2)$ satisfying $\int_{\R^2} \rhop_0(x) \log(\rhop_0(x)) \, \d x < +\infty$;
            \item $\mup_0 = \Law(\tilde Y^0)$;
        \end{itemize}
        the solution to~\eqref{eq:PDE/SDE/ODE} is unique.   
    \end{theorem}
    \begin{proof} 
        \step{1} Let us fix $\muc \in C^0([0,T];\Pcal_1(\R^2))$ and $Z = (Z_1,\dots, Z_L) \in C^0([0,T];\R^2)$ solution to the ODE in~\eqref{eq:PDE/SDE/ODE} involving $Z$. We start by observing that the PDE 
        \begin{equation}  \label{eq:PDE mup}
            \left\{
                \begin{aligned}
                    & \de_t \mup - \kappa \Delta_y \mup + \div_y\Big( \Big( \frac{1}{L} \sum_{\ell=1}^L \Kpg(\cdot - Z_\ell(t)) -  \Kpc*\muc  \Big)\mup \Big)  = 0 \, ,\\
                    & \mup(0)  = \mup_0 \, , \\
                \end{aligned}
            \right.
        \end{equation}
        has at most one solution $\mup \in \Pcal_1(C^0([0,T];\bar B_R))$, where $\bar B_R$ is a closed ball of radius $R>0$. As done in~\cite[Theorem~3.7]{AscCasSol}, we apply the result \cite[Theorem~3.3]{BogDPRocSta}. To check all the conditions, let us write the PDE using the same notation of \cite{BogDPRocSta}. Let $A(t,y) = \mathrm{Id}_2$ and $b(t,y) =  \frac{1}{L} \sum_{\ell=1}^L \Kpg(y - Z_\ell(t)) -  \Kpc*\muc(t)(y)$. Let us define the operator $\mathscr{L} \xi  =  \de_t \xi + \mathrm{tr}(A\nabla^2 \xi) + b \cdot \nabla \xi$. If $\mup \in \Pcal_1(C^0([0,T];\R^2))$ solves~\eqref{eq:PDE mup}, then it is a Radon measure\footnote{Indeed, $\mup$ can be seen as a Radon measure on $(0,T)\x\R^2$, since the duality $\xi \in C^0_c((0,T)\x \R^2) \mapsto \int_0^T \int_{\R^2} \xi(t,y) \, \mup(t)(y) \, \d t$ is a linear and continuous operator, \cf~\cite[Corollary~1.55]{AmbFusPal}.} on $(0,T) \x \R^2$ solving $\mathscr{L}^* \mup = 0$, \ie, $\int_{(0,T)\x \R^2} \mathscr{L} \xi \, \d \mup = 0$ for every $\xi \in C^\infty_c((0,T) \x \R^2)$. Trivially, $A$ is bounded and Lipschitz in the $y$ variable. By the Lipschitz continuity of $\Kpg$ and the boundedness of $Z$,
        \[
            \Big|   \frac{1}{L} \sum_{\ell=1}^L \Kpg(y - Z_\ell(t))\Big| \leq C\frac{1}{L} \sum_{\ell=1}^L|y - Z_\ell(t)| \leq C |y| + C \|Z\|_\infty \leq C(1+|y|)\,.
        \]
        Moreover, by the Lipschitz continuity of $\Kpc$,
        \[
            |\Kpc*\muc(t)(y)| \leq \int_{\R^2}   |\Kpc(y - x)|  \, \d \muc(t)(x) \leq \int_{\R^2} C \Big( |y| + |x| \Big) \, \d \muc(t)(x) \leq C(1+|y|) \, ,
        \]
        where the constant $C$ also depends on $R$. In conclusion, 
        \[
        |b(t,y) \cdot y| \leq C (1 + |y|^2) \, .    
        \]
        By the assumption $\mup_0 \in \Pcal_2(\R^2)$ we have that $\int_{\R^2} |y|^2 \, \d \mup_0(y) < +\infty$. Finally, the initial condition is attained  also in the sense
        \[
        \lim_{t \to 0} \int_{\R^2} \xi(y) \, \d \mup(t)(y) =   \int_{\R^2} \xi(y) \, \d \mup_0(y) \, ,
        \]
        since $t \in [0,T] \mapsto \int_{\R^2} \xi(y) \, \d \mup(t)(y)$ is a continuous function, \cf~also Footnote~\ref{footnote:distributional solution}.

        By~\cite[Theorem~3.3]{BogDPRocSta} we conclude that there is at most one family $(\mup(t))_{t \in [0,T]}$ that solves~\eqref{eq:PDE mup}.

        \step{2} Let now $\muc_i \in C^0([0,T];\Pcal_1(\R^2))$, $\mup_i \in \Pcal_1(C^0([0,T];\R^2))$, $i = 1,2$ (and $Z = (Z_1,\dots,Z_L) \in C^0([0,T];\R^2)$) be two solutions of~\eqref{eq:PDE/SDE/ODE}, \ie, for $i = 1,2$
        \begin{equation} \label{eq:2212041948}
            \left\{
                \begin{aligned}
                    & \de_t \muc_i + \div_x \Big( v\big( \eta *_2 \muc_i \big) \big(  \r + \Kcp*\mup_i   \big) \muc_i \Big) = 0 \, ,\\   
                    & \de_t \mup_i - \kappa \Delta_y \mup_i + \div_y\Big( \Big( \frac{1}{L} \sum_{\ell=1}^L \Kpg(\cdot - Z_\ell(t)) -  \Kpc*\muc_i  \Big)\mup_i \Big)  = 0 \, ,\\
                    & \d Z_\ell(t)  = \Big( \frac{1}{L}\sum_{\ell'=1}^L \Kgg(Z_\ell(t) - Z_{\ell'}(t) )  + u_\ell(t)  \Big) \d t \, ,  \\ 
                    & \muc_i(0) = \muc_0 \, , \\
                    & \mup_i(0)  = \mup_0 \, , \\
                    & Z_\ell(0) =  Z_\ell^0 \, ,  \  \ell = 1,\dots, L \, . \\
                \end{aligned}
            \right.
        \end{equation}
        Let $\bar Y^0$ be a $\R^2$-valued random variable with law $\mup_1(0) = \mup_2(0) = \mup_0$ and let us consider for $i=1,2$ the solutions to 
        \begin{equation} \label{eq:2212041951}
            \left\{
                \begin{aligned}
                    & \d \bar Y_i(t) = \Big(\frac{1}{L} \sum_{\ell=1}^L \Kpg( \bar Y_i(t) - Z_\ell(t))  -   \Kpc*\muc_i(t)(\bar Y_i(t)) \Big) \d t  + \sqrt{2\kappa} \, \d W(t) \, ,   \\
                    & \bar Y_i(0)  = \bar Y^0 \ \text{a.s.,}
                \end{aligned}
                \right.    
        \end{equation}
        which exist and are unique by {\itshape Substep 4.1} in the proof of Theorem~\ref{thm:N to infty}. Let us set $\mubp_i = \Law(\bar Y_i)$. (Notice the temporary difference between $\mubp_i$ and $\mup_i$.) By deriving the PDE solved by the law $\mubp_i$ using It\^{o}'s lemma (see the proof of Proposition~\ref{prop:PDE formulation of ODE/SDE/ODE}), we obtain that for $i=1,2$
        \[
                \left\{
                    \begin{aligned}
                        & \de_t \mubp_i - \kappa \Delta_y \mubp_i + \div_y\Big( \Big( \frac{1}{L} \sum_{\ell=1}^L \Kpg(\cdot - Z_\ell(t)) -  \Kpc*\muc_i  \Big)\mubp_i \Big)  = 0 \, ,\\
                        & \mubp_i(0)  = \mup_0 \, , \\
                    \end{aligned}
                \right.
        \]
        However, by~\eqref{eq:2212041948} and by the uniqueness proven in {\itshape Step 1}, we deduce that $\mubp_i = \mup_i$ for $i = 1,2$. Combining~\eqref{eq:2212041948} and~\eqref{eq:2212041951} we obtain that, for $i=1,2$
        \[
            \left\{
                \begin{aligned}
                    & \de_t \muc_i + \div_x \Big( v\big( \eta *_2 \muc_i \big) \big(  \r + \Kcp*\mup_i   \big) \muc_i \Big) = 0 \, ,\\   
                    & \d \bar Y_i(t) = \Big(\frac{1}{L} \sum_{\ell=1}^L \Kpg( \bar Y_i(t) - Z_\ell(t))  -   \Kpc*\muc_i(t)(\bar Y_i(t)) \Big) \d t  + \sqrt{2\kappa} \, \d W(t) \, ,\\
                    & \frac{\d Z_\ell}{\d t}(t)  =  \frac{1}{L}\sum_{\ell'=1}^L \Kgg(Z_\ell(t) - Z_{\ell'}(t) )  + u_\ell(t) \, ,  \\ 
                    & \muc_i(0) = \muc_0 \, , \\
                    & \bar Y_i(0)  = \bar Y^0 \ \text{a.s.,} \quad \mup_i = \Law(\bar Y_i) \, , \\
                    & Z_\ell(0) =  Z_\ell^0 \, ,  \  \ell = 1,\dots, L \, , \\ 
                \end{aligned}
            \right.
        \]
        By Theorem~\ref{thm:uniqueness of PDE/SDE/ODE}, the problem above has a unique solution, hence $\muc_1 = \muc_2$ and $\bar Y_1 = \bar Y_2$ a.s., yielding $\mup_1 = \Law(\bar Y_1) = \Law(\bar Y_2) = \mup_2$. This concludes the proof. 
    \end{proof}  

    \begin{remark} \label{rmk:absolutely continuous}
        Thanks to Theorem~\ref{thm:uniqueness of PDE/SDE/ODE}, if there exist absolutely continuous solutions to~\eqref{eq:PDE/SDE/ODE}, then $\muc$ and $\mup$ are {\itshape a fortiori} absolutely continuous by uniqueness. Under suitable conditions, the solutions are, in fact, absolutely continuous. 

        If $\muc(0) = \rhoc_0 \, \d x$, by~\cite[Theorem~2]{PicRos2} the measure $\muc(t)$ is absolutely continuous for all $t \in [0,T]$, hence there exists a density $\rhoc(t) \in L^1(\R^2)$ such that $\muc(t) = \rhoc(t) \, \d x$. This is a consequence of the Lagrangian representation of the solution to the non-local transport equation. 

        By \cite[Theorem~2.5]{BogDPRocSta} (see also~\cite[Corollary~3.9]{BogKryRoc}) there exists a function $\rhop$ locally H\"older continuous on $(0,T) \x \R^2$ such that $\mup = \rhop(t,y)  \, \d t \, \d y$.

        In conclusion,
        \begin{equation*} 
            \left\{
                \begin{aligned}
                    & \de_t \rhoc + \div_x \Big( v\big( \eta *_2 \rhoc \big) \big(  \r + \Kcp*\rhop   \big) \rhoc \Big) = 0 \, ,\\  
                    & \de_t \rhop - \kappa \Delta_y \rhop + \div_y\Big( \Big( \frac{1}{L} \sum_{\ell=1}^L \Kpg(\cdot - Z_\ell(t)) -  \Kpc*\rhoc  \Big)\rhop \Big)  = 0 \, ,\\
                    & \frac{\d Z_\ell}{\d t}(t)  =  \frac{1}{L}\sum_{\ell'=1}^L \Kgg(Z_\ell(t) - Z_{\ell'}(t) )  + u_\ell(t)  \, ,  \\ 
                    & \rhoc(0) = \rhoc_0 \, , \\
                    & \rhop(t) \, \d y  \weak \mup_0 \text{ as } t \to 0 \, , \\
                    & Z_\ell(0) =  Z_\ell^0 \, ,  \  \ell = 1,\dots, L \, . 
                \end{aligned}
            \right.
        \end{equation*}
    \end{remark}

    \subsection{Limit of optimal control problems as $N \to + \infty$} Let us consider the following cost functional for the limit problem obtained in~\eqref{eq:limit PDE/SDE/ODE}.  Let  $\J \colon L^\infty([0,T]; \U) \to \R$ be defined for every $u \in L^\infty([0,T];\U)$ by 
    \begin{equation} \label{def:J}
        \J(u) :=   \frac{1}{2}   \int_0^T    |u(t)|^2 \d t +  \int_0^T  \int_{\R^2 \x \R^2} \Hd(x - y) \, \d \muc(t) \x \mup(t)(x,y) \, \d t  \, , 
    \end{equation}
    where $\muc \in C^0([0,T];\Pcal_1(\R^2))$, $(\bar Y(t))_{t \in [0,T]}$ are obtained by  Theorem~\ref{thm:N to infty} as the unique solution to~\eqref{eq:limit PDE/SDE/ODE} and $\mup$ is the law of $(\bar Y(t))_{t \in [0,T]}$.

    \begin{theorem} \label{thm:Gamma convergence for N}
        Under the assumptions of Theorem~\ref{thm:N to infty}, the sequence of functionals $(\J_N)_{N \geq 1}$ $\Gamma$-converges to $\J$ as $N \to +\infty$ with respect to the weak* topology in $L^\infty([0,T];\U)$.\footnote{\cf~Footnote~\ref{footnote:metrizability}.}
    \end{theorem}
    \begin{proof}
        \step{1} (Asymptotic lower bound). Let us fix a sequence of controls $(u^N)_{N \geq 1}$, $u^N \in L^\infty([0,T];\U)$ such that $u^N \wstar u$ weakly* in $L^\infty([0,T];\U)$ as $N \to +\infty$ . Let us show that 
    \begin{equation}  \label{eq:liminf inequality N}
        \J(u) \leq \liminf_{N \to +\infty} \J_{N}(u^N) \, .    
    \end{equation}
    On the one hand, by definition~\eqref{def:JN}, we have that 
    \[
        \J_N(u^N) :=   \frac{1}{2}   \int_0^T    |u^N(t)|^2 \d t + \frac{1}{N}   \sum_{n=1}^N  \int_0^T  \int_{\R^2} \Hd(\bar X^N_n(t) - y) \, \d \mubp_N(t)(y) \, \d t  \, , 
    \]
    where $\bar X^N = (\bar X^N_1,\dots, \bar X^N_N)$, $(\bar Y^N(t))_{t \in [0,T]}$, and $Z^N = (Z^N_1,\dots, Z^N_L)$ are the unique strong solution to~\eqref{eq:ODE/SDE N} and $\mubp_N$ is the law of $(\bar Y^N(t))_{t \in [0,T]}$. On the other hand, 
    \[
        \J(u) :=   \frac{1}{2}   \int_0^T    |u(t)|^2 \d t +  \int_0^T  \int_{\R^2 \x \R^2} \Hd(x - y) \, \d \muc(t) \x \mup(t)(x,y) \, \d t  \, , 
        \]
        where $\muc \in C^0([0,T];\Pcal_1(\R^2))$, $(\bar Y(t))_{t \in [0,T]}$ are obtained by the unique solution to~\eqref{eq:limit PDE/SDE/ODE} and $\mup$ is the law of $(\bar Y(t))_{t \in [0,T]}$.
 
    By the weak sequential lower semicontinuity of the $L^2$-norm, we have that 
    \[
    \int_0^T |u(t)|^2 \, \d t \leq     \liminf_{N \to +\infty} \int_0^T |u^N(t)|^2 \, \d t \, .
    \]

    Let us prove the convergence 
    \begin{equation} \label{eq:2212011243}
        \frac{1}{N}   \sum_{n=1}^N  \int_0^T  \int_{\R^2} \Hd(\bar X^N_n(t) - y) \, \d \mubp_N(t)(y) \, \d t \to \int_0^T  \int_{\R^2 \x \R^2} \Hd(x - y) \, \d \muc(t) \x \mup(t)(x,y) \, \d t \, ,
    \end{equation}
    as $N \to +\infty$. This will conclude the proof of~\eqref{eq:liminf inequality N}. 

    Setting $\cHd(w) = \Hd(-w)$ and considering the empirical measures in~\eqref{eq:nucN}, we get that
    \[
        \frac{1}{N}   \sum_{n=1}^N \Hd(\bar X^N_n(t) - y) = \cHd * \nuc_N(t)(y) \, .
    \]
    Moreover,  by Fubini's theorem
    \[
        \int_{\R^2 \x \R^2} \Hd(x - y) \, \d \muc(t) \x \mup(t)(x,y) = \int_{\R^2} \cHd*\muc(t)(y)  \, \d \mup(t)(y) \, .
    \]
    These equations allow us to estimate
    \begin{equation} \label{eq:2212011226}
        \begin{split}
            & \Big| \frac{1}{N}   \sum_{n=1}^N  \int_0^T  \int_{\R^2} \Hd(\bar X^N_n(t) - y) \, \d \mubp_N(t)(y) \, \d t - \int_0^T  \int_{\R^2 \x \R^2} \Hd(x - y) \, \d \muc(t) \x \mup(t)(x,y) \, \d t \Big|  \\
            & = \Big| \int_0^T \Big( \int_{\R^2}  \cHd * \nuc_N(t)(y) \, \d \mubp_N(t)(y) - \int_{\R^2} \cHd*\muc(t)(y)  \, \d \mup(t)(y) \Big) \, \d t \Big| \\
            & \leq  \int_0^T \int_{\R^2} | \cHd * \nuc_N(t)(y) - \cHd*\muc(t)(y)| \, \d \mubp_N(t)(y)  \, \d t \\
            & \quad + \int_0^T \Big| \int_{\R^2} \cHd*\muc(t)(y)  \, \d  \mubp_N(t)(y)  -  \int_{\R^2} \cHd*\muc(t)(y') \, \d \mup(t) (y')  \Big| \, \d t  \, .
        \end{split}
    \end{equation}
    We estimate the first term in the right-hand side of~\eqref{eq:2212011226} using the Lipschitz continuity of $\cHd(y - \cdot )$ and Kantorovich's duality by
    \[
        \begin{split}
            | \cHd * \nuc_N(t)(y) - \cHd*\muc(t)(y)| & = \Big| \int_{\R^2} \cHd(y - x) \, \d \Big(\nuc_N(t) - \muc(t)\Big)(x) \Big| \\
            & \leq C \W_1(\nuc_N(t),\muc(t)) \leq \sup_{s \in [0,T]} C \W_1(\nuc_N(s),\muc(s))\, .
        \end{split}
    \]
    For the second term in the right-hand side of~\eqref{eq:2212011226}, we observe that $\cHd*\muc(t)$ is Lipschitz continuous, as 
    \[
    |\cHd*\muc(t)(y) - \cHd*\muc(t)(y')| \leq \int_{\R^2} | \cHd(y-x) - \cHd(y'-x) | \, \d \muc(t)(x) \leq C |y - y'|    
    \]
    hence, since $\mubp_N(t)$ is the law of $\bar Y^N(t)$ and $\mup(t)$ is the law of $\bar Y(t)$, 
    \[
        \begin{split}
           &  \Big| \int_{\R^2} \cHd*\muc(t)(y)  \, \d  \mubp_N(t)(y)  -  \int_{\R^2} \cHd*\muc(t)(y') \, \d \mup(t) (y')  \Big|\\
           &  \leq \E( | \cHd*\muc(t)(\bar Y^N(t)) - \cHd*\muc(t)(\bar Y(t)) | ) \leq \E(|\bar Y^N(t) - \bar Y(t)|) \leq \E(\|\bar Y^N - \bar Y\|_\infty) \, .
        \end{split}
    \]

    Combining the previous inequalities, \eqref{eq:2212011226} reads 
    \[
        \begin{split}
            & \Big| \frac{1}{N}   \sum_{n=1}^N  \int_0^T  \int_{\R^2} \Hd(\bar X^N_n(t) - y) \, \d \mubp_N(t)(y) \, \d t - \int_0^T  \int_{\R^2 \x \R^2} \Hd(x - y) \, \d \muc(t) \x \mup(t)(x,y) \, \d t \Big|   \\
            & \leq  CT \Big( \sup_{s \in [0,T]} C \W_1(\nuc_N(s),\muc(s))  + \E(\|\bar Y^N - \bar Y\|_\infty) \Big) \to 0 \quad \text{as } N \to +\infty \, ,
        \end{split}
    \]
    where the convergence follows from Theorem~\ref{thm:N to infty}. This proves \eqref{eq:2212011243}.

    \step{2} (Asymptotic upper bound). Let us fix $u \in L^\infty([0,T];\U)$. For every $N \geq 1$, let us set $u^N = u$. As in {\itshape Step 1}, we have that 
    \[
        \J_N(u^N) :=   \frac{1}{2}   \int_0^T    |u(t)|^2 \d t + \frac{1}{N}   \sum_{n=1}^N  \int_0^T  \int_{\R^2} \Hd(\bar X^N_n(t) - y) \, \d \mubp_N(t)(y) \, \d t  \, , 
    \]
    where $\bar X^N = (\bar X^N_1,\dots, \bar X^N_N)$, $(\bar Y^N(t))_{t \in [0,T]}$, and $Z^N = (Z^N_1,\dots, Z^N_L)$ are the unique strong solution to~\eqref{eq:ODE/SDE N} corresponding to the control $u^N = u$ and $\mubp_N$ is the law of $(\bar Y^N(t))_{t \in [0,T]}$. Moreover
    \[
        \J(u) :=   \frac{1}{2}   \int_0^T    |u(t)|^2 \d t +  \int_0^T  \int_{\R^2 \x \R^2} \Hd(x - y) \, \d \muc(t) \x \mup(t)(x,y) \, \d t  \, , 
        \]
        where $\muc \in C^0([0,T];\Pcal_1(\R^2))$, $(\bar Y(t))_{t \in [0,T]}$ are obtained by the unique solution to~\eqref{eq:limit PDE/SDE/ODE} and $\mup$ is the law of $(\bar Y(t))_{t \in [0,T]}$. Trivially, we have $u^N \wstar u$, hence we deduce \eqref{eq:2212011243} once again and, in particular, the asymptotic upper bound 
    \[
     \lim_{N \to +\infty} \J_{N}(u) = \J(u) \, ,    
    \]
    concluding the proof. 
    \end{proof}

    \begin{proposition} \label{prop:existence of optimal control}
        Under the assumptions of Theorem~\ref{thm:N to infty}, there exists an optimal control $u^* \in L^\infty([0,T];\U)$, \ie,  
        \[
          \J(u^*) = \min_{u \in L^\infty([0,T];\U)} \J(u)\, .  
        \] 
    \end{proposition}
    \begin{proof}
        The proof is completely analogous to the proof of Proposition~\ref{prop:existence of optimal control N}, as it is a general result about $\Gamma$-convergence.
    \end{proof}

    \section*{Declarations}

    \subsection*{Availability of data and materials} Not applicable.

    \subsection*{Competing interests} The authors declares that he has no competing interests.

    \subsection*{Funding} The author has been supported by the project ``Approccio integrato e predittivo per il controllo della criminalit\`a marittima'' in the program ``Research for Innovation'' (REFIN) - POR Puglia FESR FSE 2014-2020, Codice CUP: D94I20001410008. He is member of Gruppo Nazionale per l'Analisi Matematica, la Probabilit\`a e le loro Applicazioni (GNAMPA) of the Istituto Nazionale di Alta Matematica (INdAM) and has received funding from the GNAMPA 2022 project ``Approccio multiscala all'analisi di modelli di interazione'', Codice CUP: E55F22000270001. This work was supported by the Italian Ministry of University and Research under the Programme “Department of Excellence” Legge 232/2016 (Grant No. CUP - D93C23000100001).

    \subsection*{Authors' contributions} G.O. is the unique author of the manuscript. 

    \subsection*{Acknowledgements} The author thanks Giuseppe Maria Coclite for suggesting the problem and for interesting discussions about the model.

    \bigskip

    \bibliography{bibliography}
    \bibliographystyle{siam}

\end{document}